\documentclass[a4paper,11pt]{amsart}
\usepackage[english]{babel}
\usepackage[foot]{amsaddr}
\usepackage{amsmath}
\usepackage{amssymb}
\usepackage{graphicx}
\usepackage{amsthm}
\usepackage{hyperref}
\usepackage{thmtools}
\usepackage{thm-restate}
\usepackage{enumerate}
\usepackage{enumitem}
\usepackage{bbm}
\usepackage{mathrsfs}
\usepackage{a4wide}
\usepackage{comment}
\usepackage{tikz-cd}
\usepackage{appendix}
\usepackage{stix}
\usepackage{longtable}

\renewcommand{\tfrac}{\genfrac{}{}{}2}

\mathchardef\mm="2D

\newcommand{\eps}{\varepsilon}
\newcommand{\calX}{\mathcal{X}}
\newcommand{\calA}{\mathcal{A}}
\newcommand{\calB}{\mathcal{B}}
\newcommand{\calC}{\mathcal{C}}
\newcommand{\calD}{\mathcal{D}}

\newcommand{\calF}{\mathcal{F}}
\newcommand{\calG}{\mathcal{G}}

\newcommand{\calI}{\mathcal{I}}
\newcommand{\calL}{\mathcal{L}}
\newcommand{\calM}{\mathcal{M}}
\newcommand{\calP}{\mathcal{P}}

\newcommand{\R}{\mathbb{R}}
\newcommand{\N}{\mathbb{N}}
\newcommand{\calR}{\mathcal{R}}

\newcommand{\calT}{\mathcal{T}}

\newcommand{\fA}{\mbox{ for all }}

\newcommand{\dder}{\overline{\nabla}}

\newcommand{\dnabla}{\overline{\nabla}}
\newcommand{\ddiv}{\overline{\text{div}}}
\newcommand{\dd}{\mathrm{d}}

\newcommand{\Cyl}{\mathrm{Cyl}}
\newcommand{\sfPi}{\mathsf\Pi}
\newcommand{\sfP}{\mathsf{P}}

\newcommand{\grad}{\mathrm{grad}_{\Gamma}}

\newcommand{\sfJ}{\mathsf{J}}
\newcommand{\sfJn}{\mathsf{J}^{\mathrm{net}}}
\newcommand{\sfp}{\mathsf{p}}
\newcommand{\sfT}{\mathsf{T}}

\newcommand{\Ent}{\mathcal{E}\mathrm{nt}}
\newcommand{\teta}{\vartheta}

\newcommand{\bc}{\boldsymbol{c}}

\newcommand{\lnet}{\lambda^{\mathrm{net}}}
\newcommand{\jnet}{\sfJ^{\mathrm{net}}}
\newcommand{\Gnet}{G^\mathrm{net}}

\newtheorem{thm}{Theorem}[section]
\newtheorem{lm}[thm]{Lemma}

\newtheorem{cory}[thm]{Corollary}

\newtheorem{assu}[thm]{Assumption}
\theoremstyle{definition}
\newtheorem{defi}[thm]{Definition}
\newtheorem{fdefi}[thm]{Formal Definition}
\newtheorem{remark}[thm]{Remark}

\title[Mean-field population dynamics]{Generalized gradient structures for measure-valued population dynamics and their large-population limit}
\author{Jasper Hoeksema and Oliver Tse}
\address{Department of Mathematics and Computer Science, Eindhoven University of Technology, 5600 MB Eindhoven,The Netherlands; email addresses: j.hoeksema@tue.nl, o.t.c.tse@tue.nl}

\numberwithin{equation}{section}
\begin{document}

	\begin{abstract}
		We consider the forward Kolmogorov equation corresponding to measure-valued processes stemming from a class of interacting particle systems in population dynamics, including variations of the Bolker-Pacala-Dieckmann-Law model. Under the assumption of detailed balance, we provide a rigorous generalized gradient structure, incorporating the fluxes arising from the birth and death of the particles.
		
		Moreover, in the large population limit, we show convergence of the forward Kolmogorov equation to a Liouville equation, which is a transport equation associated with the mean-field limit of the underlying process. In addition, we show convergence of the corresponding gradient structures in the sense of Energy-Dissipation Principles, from which we establish a propagation of chaos result for the particle system and derive a generalized gradient-flow formulation for the mean-field limit. 
	\end{abstract}		
	
\maketitle
\vspace{-2.5em}

\tableofcontents


\section{Introduction}

An important goal in theoretical biology and population dynamics is to derive macroscopic equations from microscopic models \cite{Champagnat2006, Finkelshtein2009}. For many stochastic interacting particle systems involving birth, mutation, and death, these connections have been made rigorous. One such class of particle systems consists of spatially-structured models such as the Bolker-Pacala and Dieckmann-Law (BPDL) model \cite{Bolker1997,Law2000}. The dynamics of these particle systems can be described by jump processes on the space of finite positive measures and can be used to derive macroscopic models. 

The convergence of such measure-valued jump processes under a mean-field scaling to a large-population limit is shown for example in \cite{Fournier2004} via martingale techniques, and in \cite{Finkelshtein2009}, where an analytic approach to the convergence of rescaled moment equations is used. In both approaches, the limiting evolution is governed by a non-local evolution equation given by
\begin{equation}\label{eq:mfi1}
   \partial_t u_t(x) =\int_{\R^d} m(y,x)\,u_t(y) \,\dd y-u_t(x) \int_{\R^d}  c(x,y)\, u_t(y) \, \dd y.
\end{equation}
We will refer to \eqref{eq:mfi1} as the \emph{mean-field equation}. Here, $u_t$ represents the limiting density of particles at time $t$, and the parameter functions $m$ and $c$ are continuous and bounded functions stemming from birth, dispersal, and competition in the BPDL model. 

In recent years, there has been considerable activity in studying the mean-field equation \eqref{eq:mfi1} and the BPDL model in more general spaces, allowing for dynamics involving multiple species and combinations of discrete and continuous traits. See for example \cite{ Finkelshtein2021} for an overview of existing models, where instead of $\R^d$ the underlying space is an arbitrary locally compact Polish space. However, convergence in the large-population limit is not considered.

Meanwhile, powerful variational tools have been developed in the last decade for studying mean-field interacting jump processes and their limits under the assumption of detailed balance. To highlight only a few: \cite{Erbar2016} studied mean-field limits for measure-dependent jump processes; \cite{Erbar2016a} proved the convergence of the spatially-homogeneous Kac-process to the Boltzmann equation; \cite{Schlichting2016} investigated the macroscopic limit of Becker-D\"oring models; \cite{Kaiser2019} showed hydrodynamic limits for zero-range and exclusion processes; \cite{Maas2020} discussed convergence and higher-order approximations for chemical reaction networks, an approach that was subsequently used in the setting of discretized reaction-diffusion equations in \cite{MSW2022}.

In this work, we extend and apply these variational techniques to prove the mean-field limit for population dynamics over arbitrary compact Polish spaces, with bounded measurable parameters $m,c$ satisfying a detailed balance condition. In addition, we establish entropic propagation of chaos, which controls the discrepancy between the microscopic and macroscopic models in a precise sense. To the authors' knowledge, this is the first convergence result under such general assumptions. 

\medskip
To do so, we first introduce a new generalized gradient structure and rigorous variational formulation for the forward Kolmogorov equation (FKE) corresponding to the BPDL model, where the FKE describes the evolution of the law of the measure-valued process. Our formulation incorporates not only the equation itself but tracks the birth and death fluxes as well. This extends the generalized gradient-flow framework of \cite{PRST2020} due to the unboundedness of the underlying jump kernel, and the positivity of the fluxes.

We then show convergence of these generalized gradient structures under a mean-field scaling and the large-population limit in the sense of Energy Dissipation Principles (EDPs) (see \cite{Peletier2017}). The limiting gradient flow is the Liouville equation corresponding to the mean-field equation, namely a transport equation that describes the evolution of the law of a process that follows deterministic dynamics described by \eqref{eq:mfi1} but for possibly random initial conditions. This connection between the Liouville equation and the mean-field equation is made rigorous with the help of a modification of the superposition principle of \cite{ambrosio2008}.

In particular, we deduce that the laws determined by the FKE equation concentrate around the solution of the mean-field equation \eqref{eq:mfi1}, which due to the convergence of the associated free energies translates into an entropic propagation of chaos result, see Theorem \ref{thm:edp_propi}.

\medskip

\paragraph{\textbf{Outline}} The rest of this section is devoted to giving a brief overview of our setting and presenting the main results. In Section \ref{s:mf} the mean-field equation and corresponding gradient structure are introduced. We repeat this process in Sections \ref{s:fke} and \ref{s:liouv} for the forward Kolmogorov equation and the Liouville equation respectively, with the proof of a modified superposition principle delegated to Appendix \ref{s:super}. Finally, in Section \ref{s:edp}, we establish the EDP-convergence of the gradient structures, and prove both the convergence to the mean-field limit and the propagation of chaos. 

\subsection{Measure-valued population dynamics and mean-field limits}

We consider the forward Kolmogorov equation that corresponds to a generalized version of the BPDL model. In its classical form, the Bolker-Pacala model is a purely spatially-structured microscopic model for a population of plants involving the birth, dispersal, and either natural death or death by competition for resources and can be modeled as a jump process in the space of positive measures over $\R^d$. However, in certain models of adaptive evolution it is the mutation of traits that play a role, instead of spatial evolution (see \cite{Law2000,Champagnat2006,Champagnat2006a}). Moreover, if one wants to model multiple interacting species or marked configuration spaces, more general spaces than $\R^d$ are needed \cite{Kondratiev2006,Finkelshtein2021}).

Therefore, let the {\em trait} space be an arbitrary Polish space, denoted henceforth as $\calT$. We model the BPDL-dynamics at any time $t$ as an interacting particle system with particles $A_t^1,\dots,A_t^{N_t} \in \calT$ at positions $X_t^1,\dots,X_t^{N_t} \in \calT$, where the number of particles $N_t$ at time $t$ is not fixed since particles can be removed from and added to the system. 

Moreover, let $b\in \calB^+(\calT)$, $d,c\in \calB^+(\calT\times\calT)$ be non-negative measurable functions, $n>0$ a positive parameter, and $\gamma\in \calM^+_{loc}(\calT)$ a non-negative reference measure such that 
\[ \int_{\calT} d(x,y)\, \gamma(\dd y) = 1, \qquad \fA x\in \calT. \]
Then the BPDL-dynamics can be described as follows:
\begin{itemize}
    \item Each particle located at a position $x\in \calT$ has two exponential clocks: a \emph{seed} clock with rate $b(x)$ and a \emph{death} clock with rate $\tfrac{1}{n}\sum_{i=1}^{N_t} c(x,X_t^i)$. 
    \item If the death clocks rings, the particle is deleted.
    \item If the seed clock rings, a new particle is added at position $y\in\calT$ with probability $d(x,y)\gamma(\dd y)$.
\end{itemize}
Alternatively, we can describe these dynamics in the form of reacting particles. Namely, setting $m(x,y):=b(x)d(x,y)$, then with a little of abuse of notation we have
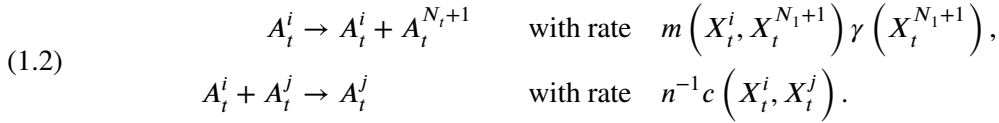
\begin{equation}\label{eq:ireaction}
\begin{aligned}
    A^i_t &\to A^i_t+A^{N_t+1}_t \quad &\mbox{with rate} \quad& m\left(X_t^i,X_t^{N_1+1}\right)\gamma\left(X_t^{N_1+1}\right),\\
  A^i_t+A^j_t &\to A^j_t \quad &\mbox{with rate} \quad &n^{-1}c\left(X^i_t,X^j_t\right).
\end{aligned}   
\end{equation}
We will refer to $m$ as the \emph{mutation kernel}, and $c$ as the \emph{competition} kernel. The parameter $n>0$ is called the \emph{system size}, in the sense that that the scaling $n^{-1}c$ guarantees that if the amount of particles in the system is of the order of $n$, the total rate of created or deleted particles is of the same order.

\medskip

Instead of looking at the individual positions of the particles, it is common to only consider the measure-valued process $\nu_t$ determined by the rescaled empirical measure
\begin{equation*}
\nu_t^n:=\frac{1}{n}\sum_{i=1}^{N(t)} \delta_{X_t^i}.
\end{equation*}
Here, $\nu_t\in \Gamma:=\calM^+(\calT)$ with $\calM^+(\calT)$ the space of finite non-negative measures. The infinitesimal generator $Q_n$ of this process is given for all $F\in C_c(\Gamma)$ by 
\[
    \qquad  (Q_n F)(\nu) = n \int_{\calT} \left(F\left(\nu+\tfrac{1}{n}\delta_x\right)-F(\nu)\right)\, \kappa^+[\nu](\dd x)+n \int_{\calT} \left(F\left(\nu-\tfrac{1}{n}\delta_x\right)-F(\nu)\right)\, \kappa^-[\nu](\dd x), 
\]
where $\kappa^{\pm}[\nu]\in \Gamma$ are the measure-dependent birth/death-kernels 
\[
    \kappa^+[\nu](\dd x) :=  \left(\int_{y\in \calT} m(y,x)\nu(\dd y)\right)\gamma(\dd x) ,\qquad \kappa^-[\nu](\dd x) := \left(\int_{y\in \calT}  c(x,y) \nu(\dd y)\right) \nu(\dd x).
\]
The law of the process is now given by the corresponding forward Kolmogorov equation
\begin{equation}\label{eq:Forward Kolmogorov}
     \partial_t \sfP_t^n = Q_n^* \sfP_t^n, \qquad \sfP^n_t\in \calP(\Gamma).\tag{$\sf FKE_n$}
\end{equation}

\medskip

Depending on the setting, this formulation can be made rigorous in various ways: for example via an analytical approach on configuration spaces as done in \cite{Finkelshtein2009}, which in fact models infinite configurations of particles over $\R^d$, or via martingale techniques with $\calT$ a closed subset of $\R^d$ and $\gamma=\mathscr{L}^d|_{\calT}$ (see \cite{Fournier2004}). Moreover, in the latter, under the assumption of continuous, bounded, and integrable mutation/competition kernels, it is also shown that the process converges in the large-population limit $n\to \infty$ to the mean-field equation \eqref{eq:mfi1}, which can be rewritten as
\begin{equation}\label{eq:mf}\tag{$\sf MF$}
    \partial_t \nu_t =\kappa^+[\nu_t]-\kappa^-[\nu_t], \qquad \nu_t\in \Gamma.
\end{equation}
While different choices of scalings are possible, the mean-field equation describes the macroscopic properties of the measure-valued process when the population is large. An alternative way is to study the evolution of the moments, which form a hierarchy similar to the BBGKY-hierarchy of correlation functions, and under the so-called Vlasov scaling the first moment or correlation function converges to \eqref{eq:mf}. For the case of infinite configurations over $\R^d$ this has been established, see \cite{Finkelshtein2010}, and both propagation of chaos in the Vlasov limit and the sub-Poissonian property have been established as well \cite{Finkelshtein2015}. 

\medskip

In this work, we do not consider the measure-valued process itself, but take the forward Kolmogorov equation \eqref{eq:Forward Kolmogorov} as a starting point, and show convergence to the mean-field equation in the sense that $\sfP^n_t\to \delta_{\nu_t}$ narrowly on $\calP(\Gamma)$ under suitable initial conditions. Throughout we assume the following: 
\begin{assu}\label{assu:massu}
The trait space $\calT$ is a compact Polish space, and moreover
\begin{equation*}
    \begin{aligned}
    \gamma&\in \Gamma \qquad&& && \mbox{(reference measure with finite mass)}\\
    m,c&\in \calB_b^+(\calT\times\calT) \qquad&& &&\mbox{(bounded rates)}\\
    c(x,x)&=0 \quad &&\fA x\in\calT\qquad &&\mbox{(no natural death)}\\
    m(y,x)&=c(x,y) \quad &&\fA x,y\in\calT\qquad &&\mbox{(detailed balance)}\\
    \end{aligned}
\end{equation*}
\end{assu}
Henceforth we equip the space $\Gamma$ with the narrow topology. Moreover, the assumption of no natural death means that particles can only be deleted due to competition with other particles. Together with the detailed balance condition this guarantees that the jump kernel is reversible with respect to an invariant measure $\Pi_n \in \calP(\Gamma)$, which is obtained as a push-forward of the Poisson measure $\pi_n$ with
\begin{equation*}
    \calP\left(\coprod_{N\geq 1}\calT^N\right) \ni \pi_n:=\frac{1}{e^{n \gamma(\calT)}-1}\sum_{N=1}^{\infty} \frac{n^N}{N!}\gamma^{\otimes N}.
\end{equation*}
This allows us to write the forward Kolmogorov equation as a gradient flow of the relative entropy with respect to $\Pi_n$, and equip it with a corresponding variational structure, see Theorem \ref{thm:ikf}. 

In light of similar results in \cite{Erbar2016,Maas2020} for mean-field jump processes on finite spaces and finite chemical reaction networks, one expects \eqref{eq:Forward Kolmogorov} to converge to the following \emph{Liouville equation}
\begin{equation}\label{eq:liouv}\tag{$\sf Li$}
    \partial_t \sfP_t +\mathrm{div}_{\Gamma} \left(\sfP_t\,\bigl(\kappa^+[\nu]-\kappa^-[\nu]\bigr)\right)=0, \qquad \sfP_t\in \calP(\Gamma).
\end{equation}
It is a transport equation that can be interpreted as the lifting of mean-field dynamics in $\Gamma$ to evolutions in $\calP(\Gamma)$, and describes the evolution of the law of random measures $\nu_t$ that all satisfy \eqref{eq:mf}. In particular, if $\nu_t$ a solution of \eqref{eq:mf} then $\sfP_t:=\delta_{\nu_t}$ is itself a solution of \eqref{eq:liouv}. 

It turns out that in our general setting this convergence holds as well, as will be stated in Theorem \ref{thm:i_sol}. Letting $V[\nu]=\kappa^+[\nu]-\kappa^-[\nu]$, we can therefore represent part of our results in Figure \ref{intro:diagram1}.
\begin{figure}[ht]\label{intro:diagram1}\centering
\begin{tikzcd}[row sep=large,column sep=large]
\eqref{eq:Forward Kolmogorov} \quad   \partial_t \sfP^n_t=Q_n^*\sfP^n_t \;\arrow[r, "n\to \infty", ""']  &  \;\eqref{eq:liouv}  &[-38pt] \partial_t \sfP_t+\mathrm{div}_{\Gamma}\,(\sfP_t V[\nu])=0, \arrow[shift right=.3em,swap]{d}{} &[0pt] \sfP_t^n, \sfP_t\in \calP(\Gamma),\\
      &  \;\eqref{eq:mf} &  \partial_t\nu_t=V[\nu_t],  \arrow[shift right=.3em,swap]{u}{} & \nu_t\in \Gamma:=\calM^+(\calT).
  \end{tikzcd}
\caption{Convergence in the large-population limit}
\end{figure}

This convergence is a direct consequence of the convergence of the associated gradient structures, which we will describe below. 

\subsection{Gradient-flow formulation}

Our first main result concerns the variational formulation of the equations \eqref{eq:Forward Kolmogorov}, \eqref{eq:mf}, \eqref{eq:liouv} and their specific gradient structure. Various gradient-flow formulations exist for jump processes, mean-field jump processes, and chemical reaction networks \cite{Erbar2016, Erbar2016a,Kaiser2019,Maas2020,PRST2020}.

In these works a common starting point is to describe the relation between $\rho_t$, representing either laws of some process or mean-field limits and generalized {\em fluxes} $j_t$ in the form of an abstract {\em continuity equation}. For example, in the case of independent particles following a common jump process over a graph, $\rho_t$ corresponds to the number of particles on a node at time $t$, and a choice of flux $j_t$ can be the so-called \emph{net} flux $j_t$, which is related to the number of particles going through an edge. 

However, we propose a slightly different structure, namely one that tracks the effective \emph{mass fluxes} for both creation (arising from mutation) and annihilation (arising from competition) separately. The use of mass fluxes instead of usual particle fluxes ensures that in our convergence results as $n\to \infty$ we have both convergences of laws and fluxes (see Theorem \ref{thm:maini_conv1}). 

Moreover, separating the effects of creation and annihilation (henceforth simply referred to as birth and death) instead of their combined contribution allows us to incorporate more information in our variational formulation. The downside is that we are forced to work with \emph{positive fluxes}, while the framework in the aforementioned examples involves either quadratic or generalized structures for signed net fluxes. In this sense we are closer to the variational representations stemming from large deviations, involving so-called one-way or unidirectional fluxes, see for example \cite{Mielke2014, Renger2019, Basile2021, Peletier2022}. Indeed, our structure is motivated by large deviation theory, as we will discuss briefly in Appendix \ref{s:ldpmot}.

\medskip
In all three cases, i.e.\ for \eqref{eq:Forward Kolmogorov}, \eqref{eq:mf} and \eqref{eq:liouv}, our proposed structure is similar to the classical notion of a gradient flow in the sense that they all satisfy an abstract Energy-Dissipation Balance. Since we will repeat the same concept three times on different levels and for different spaces, let us make the general and abstract concepts clear:

\begin{fdefi}
Given a {\em free energy} functional $\calF(\rho)$, a {\em dissipation potential} $\calR(\rho,j)$, a {\em Fisher information} functional $\calD(\rho)$, and a linear operator $B$ with dual $B^*$, we consider pairs of curves $(\rho,j)$ satisfying the continuity equation
\begin{equation}\label{eq:ce-formal}\tag{$\mathsf{CE}$}
   \partial_t\rho_t + B^* j_t = 0, \quad \mbox{for a.e.\ $t\in [0,T]$,}
\end{equation} 
and define the {\em EDP-functional}
\begin{equation*}
    \calI(\rho,j):=\int_0^T \calR(\rho_t,j_t) \, \dd t + \calF(\rho_T)-\calF(\rho_0) +\int_0^T \calD(\rho_t)\, \dd t.
\end{equation*}
Moreover, a \emph{gradient-flow solution} is a pair $(\hat{\rho},\hat{\jmath})$ satisfying \eqref{eq:ce-formal} with $I(\hat{\rho},\hat{\jmath})=0$.
\end{fdefi}

Throughout we require the non-negativity of $\calI$. For a  deeper look at the mathematical basis of this sort of setting, especially for generalized gradient systems incorporating net fluxes, see \cite{PRST2020}. 

In all three examples the generalized fluxes $j$ consist of two parts: $j^+$ and $j^-$, corresponding to birth and death. The continuity equations depend on the setting and are summarized in Table \ref{tab:CE}, with $\calM^+_{loc}$ as the space of non-negative Radon measures.

\begin{remark}
Note that the gradient-flow solution $(\hat{\rho},\hat{\jmath})$ is the null-minimizer of $\calI$, and satisfies the \emph{energy-dissipation balance} 
\begin{equation*}
    \calF(\hat \rho_T)+ \int_0^T \left(\calR(\hat \rho_t,\hat{\jmath}_t)+\calD(\hat \rho_t)\right) \, \dd t =\calF(\hat \rho_0).
\end{equation*}
Moreover, for small $T\ll 1$ one would expect
\[ I \approx \calR(\hat \rho,\hat{\jmath})+\langle \hat{\jmath}, B\,  \partial_{\rho} 
\calF \rangle +\calD(\hat \rho).  \]
In light of the generalized gradient-flow framework of \cite{PRST2020} and the relation to minimizing movement schemes, a formal minimization procedure provides the gradient-flow solution
\begin{equation*}
\begin{aligned}
 \partial_t \hat \rho+B^* \hat{\jmath}&=0\\
 \hat{\jmath}&=(\,\partial_2 \calR^*)(\hat \rho,-B\,  \partial_{\rho} \calF),\\
\end{aligned}
\end{equation*}
and that along the solution
\begin{equation}\label{eq:intro_fish}
    \calD(\hat \rho)=\calR^*(\hat \rho,-B\,  \partial_{\rho} \calF).
\end{equation}
where $\calR^*(\rho,w)$ is the dual of the dissipation potential $\calR$. Finally, note that along the gradient-flow solution the free energy $\calF$ is non-increasing, i.e.\ $\calF$ is a Lyapunov functional. 

These (in)equalities indeed hold in our setting. See also Appendix \ref{s:ldpmot}, where we compare the relation to generalized gradient flows for net fluxes, which follow from the above after a contraction argument, and the connection to the reversibility of the underlying process. 
\end{remark}

\begin{table}[ht]
\def\arraystretch{1.8}
    \centering
    \begin{tabular}{c|c|c|c|c}
    \hline
    &$\mathsf{CE}$&$\rho$ &$j=(j^+,j^-)$ & $B F$ \\ 
    \hline
        \eqref{eq:mf}&  \eqref{eq:imf_ce}& $\nu \in \Gamma$& $(\lambda^+,\lambda^-)\in \Gamma^2$&$(F,-F)$\\
        \eqref{eq:Forward Kolmogorov}&\eqref{eq:imv_ce} & $\sfP\in \calP(\Gamma)$  &$(\sfJ^+,\sfJ^-)\in \calM^+_{loc}(\Gamma\times \calT)^2$ &\;\;\;\;$(\dder^{n,+}F,\dder^{n,-}F)$\hfill \eqref{eq:intro_dder}\\
        \eqref{eq:liouv}&\eqref{eq:li_ce} &$\sfP\in\calP(\Gamma)$ &$(\sfJ^+,\sfJ^-)\in \calM^+_{loc}(\Gamma\times \calT)^2$ &$(\mathrm{grad}_{\Gamma}
        F,-\mathrm{grad}_{\Gamma}F)$\;\; \eqref{eq:intro_grad} \\
        \hline
    \end{tabular}
    \vspace*{3mm}
    \caption{Continuity equations}
    \label{tab:CE}
\end{table}

\medskip
Let $H(\mu_1,\mu_2)$ be the Hellinger distance, see \eqref{eq:mf_hell}, and $\Ent(\mu_1|\mu_2)$ the relative entropy of $\mu_1$ with respect to $\mu_2$ for two (possible infinite) locally finite Borel measures $\mu_1,\mu_2$:
\begin{equation*}
    \Ent(\mu_1|\mu_2) :=\left\{
\begin{aligned}
   & \int \phi\left(\frac{\dd \mu_1}{\dd \mu_2}\right) \dd \mu_2, &&\mbox{if $\mu_1\ll \mu_2$,}\\
   & +\infty, \qquad &&\mbox{otherwise,}
\end{aligned}\right.
\end{equation*} 
where 
\[ \phi(s)=s \log s-s+1.\]

With the full technical details contained in Theorems \ref{thm:equivlag}, \ref{thm:fke_main} and  \ref{thm:li_equiv}, we then have the following triple of results below,
\begin{thm}[Mean-field, cf.\ Theorem \ref{thm:equivlag}]\label{thm:imf}
Consider triples $(\nu,\lambda^+,\lambda^-)$, with $\nu_t,\lambda_t^{\pm}\in \Gamma$, satisfying the mean-field continuity equation
\begin{equation}\label{eq:imf_ce}\tag{$\mathscr{CE}$}
\partial_t \nu_t = \lambda_t^+-\lambda^-_t.  
\end{equation}
Define the dissipation potential $\calR_{MF}$, free energy $\calF_{MF}$ and Fisher information $\calD_{MF}$ as
\begin{equation*}
    \begin{aligned}
    \calR_{MF}(\nu,\lambda^+,\lambda^-)&:=\Ent(\lambda^+|\theta_{\nu})+\Ent(\lambda^-|\theta_{\nu}),\\
    \calF_{MF}(\nu)&:=\tfrac{1}{2}\Ent(\nu|\gamma),\\
    \calD_{MF}(\nu)&:=\left\{\begin{aligned}
    &2H^2(\kappa^+[\nu],\kappa^-[\nu]) \qquad && \mbox{if $\nu\ll\gamma$,}\\
    &+\infty,\qquad &&\mbox{otherwise,}
    \end{aligned}\right.
    \end{aligned}
\end{equation*}
where $\theta_{\nu}$ is the geometric mean of the expected birth and death fluxes, i.e.\ 
\begin{equation*}
    \theta_{\nu}:=\sqrt{\kappa^+[\nu] \kappa^-[\nu]}.
\end{equation*}

Then the corresponding EDP-functional $I_{MF}$ given by 
\begin{equation*}
    \calI_{MF}(\nu,\lambda^+,\lambda^-):=\int_0^T \calR_{MF}(\nu_t,\lambda^+_t,\lambda^-_t) \, \dd t + \calF_{MF}(\nu_T)-\calF_{MF}(\nu_0) +\int_0^T \calD_{MF}(\nu_t)\, \dd t,
\end{equation*}
is non-negative, and for any $\nu_0$ with $\calF(\nu_0)<\infty$ a unique gradient-flow solution $(\hat \nu,\hat \lambda^{+},\hat \lambda^{-})$ exists, with $\hat \nu_t$ equal to the unique strong solution to \eqref{eq:mf} and $\hat \lambda_t^{\pm}=\kappa^{\pm}[\hat \nu_t]$ for almost every $t\in [0,T]$. 
\end{thm}

As mentioned, although treating birth and death separately provides us with additional information, this prohibits the use of some of the previous works for gradient structures because of the positivity of the fluxes. However, there is still a strong connection to the variational formulations for jump processes arising from the large deviations of fluxes as seen in \cite{Renger2019} and \cite{Basile2021}, see for example Appendix \ref{s:ldpmot} on the equivalence of the EDP-functional to the expected rate functional. 

\begin{remark}
It is straightforward to verify that if $\dd \nu=u \dd \gamma$
\begin{align*}
   \calR^*_{MF}(\nu,\partial_{\nu} \calF_{MF},-\partial_{\nu} \calF_{MF})&=\int_{\calT^2} 1_{u(x)>0} c(x,y)  \left(\sqrt{u(x)}-1\right)^2\gamma(\dd x)\nu(\dd y), \\
   \calD_{MF}(\nu)&=\int_{\calT^2} c(x,y)  \left(\sqrt{u(x)}-1\right)^2\gamma(\dd x)\nu(\dd y),
\end{align*}
and hence it is not directly clear that the relation \eqref{eq:intro_fish} holds. However, as will be shown for Theorem \ref{thm:equivlag}, at least along the solution $\hat \nu_t$ the equivalence holds for a.e. $t\in [0,T]$.
\end{remark}

\begin{thm}[Forward Kolmogorov, cf.\ Theorem \ref{thm:fke_main}]\label{thm:ikf}
Consider triples $(\sfP,\sfJ^+,\sfJ^-)$, with $\sfP_t\in \calP(\Gamma)$ and $J_t^{\pm}\in \calM_{loc}(\Gamma\times \calT)$, satisfying the continuity equation
\begin{align}\label{eq:imv_ce}\tag{$\mathsf{CE}_{n}$}
\langle F, \partial_t \sfP_t\rangle = \langle \dder^{n,+}F,\sfJ_t^+\rangle+\langle \dder^{n,-}F,\sfJ_t^-\rangle, \qquad \forall F\in C_c(\Gamma),
\end{align}
where
\begin{equation}\label{eq:intro_dder}
    (\dder^{n,\pm} F)(\nu,x):=n\left(F(\nu\pm \tfrac{1}{n}\delta_x)-F(\nu)\right).
\end{equation}
Define the $n$-dependent Fisher information $\calD_{n}$ as stated in Definition \ref{defi:fke}, free energy
\[
     \calF_n(\sfP):=\frac{1}{2n} \Ent(\sfP|\Pi_n),
\]
and dissipation potential 
\begin{align*}
    \calR_n(\sfP,\sfJ^+,\sfJ^-)&:=\Ent(\sfJ^+|\Theta_{\sfP}^{n,+})+\Ent(\sfJ^-|\Theta_{\sfP}^{n,-}),
\end{align*}
where, with a little abuse of notation (see \eqref{eq:Theta}), 
\begin{equation*}
\Theta_{\sfP}^{n,\pm}(\nu,x):=\sqrt{ \Big(\sfP(\nu)\kappa^{\pm}[\nu]\Big) \left(\sfP(\nu\pm\tfrac{1}{n}\delta_x)\kappa^\mp[\nu\mp\tfrac{1}{n}\delta_x]\right)}.
\end{equation*}
Then the corresponding EDP-functional $I_{n}$ given by 
\begin{equation*}
    \calI_{n}(\sfP,\sfJ^+,\sfJ^-):=\int_0^T \calR_{n}(\sfP_t,\sfJ^+_t,\sfJ^-_t) \, \dd t + \calF_{n}(\sfP_T)-\calF_{n}(\sfP_0) +\int_0^T \calD_{n}(\sfP_t)\, \dd t,
\end{equation*}
is non-negative, and for any $\sfP_0$ with $\calF_n(\sfP_0)<\infty$ a unique gradient-flow solution $(\hat{\sfP},\hat{\sfJ}^{\pm})$ exists, with $\hat{\sfP}_t$ equal to a weak solution to \eqref{eq:Forward Kolmogorov} and $\hat{\sfJ}_t^{\pm}=\hat{\sfP}_t \kappa_{\nu}^{\pm}$ for almost every $t\in [0,T]$.
\end{thm}

Similar to the mean-field case, the dissipation potential consists of relative entropies with respect to geometric averages, now of forward and backward rates along a transition $\nu\to \nu \pm \tfrac{1}{n}\delta_{x}$. 
Moreover, note that in contrast to the framework of \cite{PRST2020}, we employ fluxes $\sfJ^{\pm}$ that are not finite measures. This is due to the unboundedness of $\kappa_{\nu}$ as the mass of $\nu$ grows, which implies that the underlying jump kernel over $\Gamma$ is itself unbounded as well, see Section \ref{s:fke}. 

\medskip

For the Liouville equation, let us define $\Cyl_c(\Gamma)$ as the space of compactly supported smooth cylinder functions of the form 
\begin{equation*}
    F(\nu)=g\left(\langle 1,\nu\rangle,\langle f_1,\nu\rangle,\dots,\langle f_m,\nu\rangle \right),\qquad g\in C^{\infty}_c(\mathbb{R}^{m}),\;m \in \N,
\end{equation*}
where $f_1,\dots,f_m\in C_b(\calT)$, and $\grad$ is the distributional gradient defined by
\begin{equation}\label{eq:intro_grad}
    \mathrm{grad}_{\Gamma}\, F(\nu,x)= (\nabla g)\left(\langle 1,\nu\rangle,\langle f_1,\nu\rangle,\dots,\langle f_m,\nu\rangle \right) \cdot (1,f_1(x),\dots,f_m(x))^\top.
\end{equation}

\begin{thm}[Liouville, cf.\ Theorem \ref{thm:li_equiv}]\label{thm:liouvi}
Consider triples $(\sfP,\sfJ^+,\sfJ^-)$, with $\sfP_t\in \calP(\Gamma)$,  $\sfJ^{\pm}\in \calM_{loc}(\Gamma\times \calT)$, satisfying the continuity equation
\begin{equation}\label{eq:li_ce}\tag{$\mathsf{CE}_{\infty}$}
\langle F ,\partial_t \sfP_t\rangle = \langle \mathsf{grad}_{\Gamma} F,\sfJ^+_t\rangle-\langle \mathsf{grad}_{\Gamma} F,\sfJ^-_t\rangle, \quad \forall F\in \mathrm{Cyl}_c(\Gamma).
\end{equation}
Define the Fisher information $\calD_{\infty}$ as stated in Definition \ref{defi:liouv}, free energy 
\begin{align*}
    \calF_{\infty}(\sfP)&:=\frac{1}{2}\int_{\Gamma} \Ent(\nu|\gamma) \, \dd \sfP,
\end{align*}
and dissipation potential
\[
    \calR_{\infty}(\sfP,\sfJ^+,\sfJ^-) :=\Ent(\sfJ^+|\Theta_{\sfP}^{\infty})+\Ent(\sfJ^-|\Theta_{\sfP}^{\infty}),\qquad \Theta_{\sfP}^{\infty}(\dd \nu,\dd x):=\theta_{\nu}(\dd x)\sfP(\dd \nu).
\]
Then the corresponding EDP-functional $I_{\infty}$ given by 
\begin{equation*}
    \calI_{\infty}(\sfP,\sfJ^+,\sfJ^-):=\int_0^T \calR_{\infty}(\sfP_t,\sfJ^+_t,\sfJ^-_t) \, \dd t + \calF_{\infty}(\sfP_T)-\calF_{\infty}(\sfP_0) +\int_0^T \calD_{\infty}(\sfP_t)\, \dd t,
\end{equation*}
is non-negative, and for any $\sfP_0$ with $\calF_{\infty}(\nu_0)<\infty$ a unique gradient-flow solution $(\hat{\sfP},\hat{\sfJ}^{\pm})$ exists, with $\hat{\sfP}_t$ a weak solution to \eqref{eq:liouv} and $\hat{\sfJ}_t^{\pm}=\hat{\sfP}_t \kappa_{\nu}^{\pm}$ for almost every $t\in [0,T]$.

Finally, for any $(\sfP,\sfJ^+,\sfJ^-)$ such that $I_{\infty}(\sfP,\sfJ^+,\sfJ^-)<\infty$, there exists (with a little abuse of notation) a Borel probability measure $\Omega$ over curves satisfying the mean-field continuity equation \eqref{eq:imf_ce} such that for all $t$ the time marginals $(e_t)_{\#} \Omega$ are equal to $\sfP_t$, and
\begin{equation}\label{eq:superp}
    \calI_{\infty}(\sfP,\sfJ^+,\sfJ^-):=\int \calI_{MF}(\nu,\lambda^+,\lambda^-) \,\dd \Omega. 
\end{equation}
\end{thm}

The statement of \eqref{eq:superp} is the aforementioned superposition principle, which is a modified version of the superposition principle \cite{Ambrosio2014} in metric measure spaces, and the ones used in \cite{Erbar2016}, \cite{Erbar2016a}. It allows one to essentially jump back and forth between the Liouville equation and the mean-field dynamics, and in particular, provides us with the non-negativity of $\calI_{\infty}$ and uniqueness of gradient-flow solutions.

\subsection{Convergence results} 
Our final and most important result is that the above gradient structures converge in the sense of \emph{EDP-convergence} (e.g.\ see \cite{Peletier2017,Peletier2022}), a generalization of the evolutionary $\Gamma$-convergence approach stated by \cite{Sandier2004,Serfaty2011} and expanded on in \cite{Mielke2016}, which implies convergence of the gradient-flow solutions and their free energies. 

We say that a sequence $(\sfP^n,\sfJ^{n,+},\sfJ^{n,-})\in \mathsf{CE}_n$ converges to some $(\sfP,\sfJ^+,\sfJ^-)\in \mathsf{CE}$ if for all $t\in [0,T]$ the probability measures $\sfP_t^n$ converge narrowly to $\sfP_t$ in $\calP(\Gamma)$, and $\sfJ^{n,\pm}_t(\dd \nu,\dd x) \,\dd t$ converge vaguely to $\sfJ^{\pm}_t (\dd \nu,\dd x)\, \dd t$ in $\calM_{loc}([0,T]\times \Gamma\times \calT)$. Again postponing technicalities, see Theorem \ref{thm:main_conv1}, we have the following lower semi-continuity and compactness result:
\begin{thm}[cf.\ Theorem \ref{thm:main_conv1}]\label{thm:maini_conv1} 
The sequence of free energies $\calF_n$ $\varGamma$-converges to $\calF_{\infty}$. 

Moreover, the sequence of Fisher-information functionals and dissipation potentials are all sequentially lower semicontinuous. In particular, for any sequence $(\sfP^n,\sfJ^{n,+},\sfJ^{n,-})\in \mathsf{CE}_n$ converging to a $(\sfP,\sfJ^+,\sfJ^-)\in \mathsf{CE}_{\infty}$ such that $\calF_{n}(\sfP_0^n)\to \calF_{\infty}(\sfP_0)$ as well, we have
\[ \liminf_{n\to \infty} \calI_n(\sfP^n,\sfJ^{n,+},\sfJ^{n,-})\geq \calI_{\infty}(\sfP,\sfJ^{+},\sfJ^{-}). \]
Finally, for any sequence $(\sfP^n,\sfJ^{n,+},\sfJ^{n,-})\in \mathsf{CE}_n$ such that 
\[\begin{aligned}
\limsup_{n\to \infty} \calF_n(\sfP_0^n)<\infty,\\
\limsup_{n\to \infty} \calI_n(\sfP^n,\sfJ^{n,+},\sfJ^{n,-})<\infty,
\end{aligned}\]
there exists a subsequence converging to some $(\sfP,\sfJ^+,\sfJ^-)\in \mathsf{CE}_{\infty}$.
\end{thm}

Here the notion of EDP-convergence or evolutionary $\varGamma$-convergence (where the $\varGamma$ is not to be confused with our space of positive measures $\Gamma$) relates to the $\varGamma$-convergence of the free energies $\calF_n$ and suitable liminf-estimates for the dissipation potentials and Fisher-information functionals (or local slopes in a metric setting). 

In certain applications or for certain notions of convergence (e.g.\ see \cite{MMP2021}) one also establishes $\varGamma$-convergence for the total dissipation $\calR_n+\calD_n$ when written as functionals over $C([0,T];\calP(\Gamma))$. Moreover, $\varGamma$-convergence of the functionals $\calI_n$ over such path-spaces are related to the large deviations of the underlying process \cite{Kraaij2019}, as we briefly discuss in Appendix \ref{s:ldpmot}. In our framework this would require that for every $(\sfP,\sfJ^+,\sfJ^-)\in \mathsf{CE}_{\infty}$, we can find a sequence $(\sfP^n,\sfJ^{n,+},\sfJ^{n,-})\in \mathsf{CE}_n$ that converges to $(\sfP,\sfJ^+,\sfJ^-)$ and satisfies the limsup-estimate 
\[ \limsup_{n\to \infty} \calI_n(\sfP^n,\sfJ^{n,+},\sfJ^{n,-}) \leq \calI_{\infty}(\sfP,\sfJ^+,\sfJ^-). \]

However, in this paper we restrict ourselves only to the liminf-estimates, which is sufficient to obtain convergence of the solutions, an approach also taken in \cite{Erbar2016, Erbar2016a,Maas2020}. Namely, by a lower semicontinuity and compactness argument, Theorem \ref{thm:maini_conv1} implies the convergence of both the solutions and the free energies $\calF_n$, if the initial data are well prepared. 

\begin{thm}[cf.\ Theorem \ref{lm:main_conv2}]\label{thm:i_sol}
Suppose that $\sfP_0^n \to \sfP$ with $\calF_{n}(\sfP_0^n)\to \calF_{\infty}(\sfP_0)$ as well. Then for the sequence $\hat{\sfP}^n$ of gradient-flow solutions to \eqref{eq:Forward Kolmogorov}, and $\hat{\sfP}$ the gradient-flow solution to \eqref{eq:liouv}, we have that for all $t\in [0,T]$
\begin{equation*}
 \hat{\sfP}_t^n \to \hat{\sfP}_t \mbox{\; narrowly,\quad and}\quad \lim_{n\to \infty} \calF_n(\hat{\sfP}^n_t)= \calF_{\infty}(\hat{\sfP}_t).
\end{equation*}
In particular, if $\sfP_0=\delta_{\hat \nu_0}$ and $\hat \nu_t$ is the solution to the mean-field problem \eqref{eq:mf}, then for all $t\in [0,T]$
\begin{align*}
     \hat{\sfP}_t^n \to \delta_{\hat{\nu}_t} \mbox{\; narrowly,\quad and}\quad \lim_{n\to \infty} \frac{1}{n} \Ent(\hat{\sfP}^n_t|\Pi_n)=\Ent(\hat \nu_t|\gamma).
\end{align*}
\end{thm}
The second half of Theorem \ref{thm:i_sol}, on the concentration around mean-field solutions and convergence of entropies, follows directly from the definition of $\calF_{\infty}$ and uniqueness. 

\medskip

For interacting particle systems where the number of particles is fixed at $n\in \N$ the narrow convergence  $\hat{\sfP}_t^n\to \delta_{\hat{\nu}_t}$ is equivalent to propagation of chaos in the sense of Snitzman \cite{Sznitman1991}, and would imply narrow convergence of the $k$-particle marginals at time $t$ to $\nu_t^{\otimes k}$. However, in our setting, this implies convergence of the $k$-correlation functions, see \cite{BGSS2020}. 

Moreover, the convergence of the free energies $\calF_n$ implies the stronger notion of entropic propagation of chaos if the initial condition is sufficiently regular. 

\begin{thm}[cf.\ Theorem \ref{thm:edp_prop}]\label{thm:edp_propi}
Suppose that $\sfP^n_0\to \delta_{\hat \nu_0}$ with $C^{-1} \leq \dd \hat \nu_0/\dd \gamma\leq C$ for some $C>0$. If
\[ \lim_{n\to \infty} \frac{1}{n} \Ent(\hat{\sfP}_0^n|\Pi_{n,\hat \nu_0})=0, \]
then
\[ \lim_{n\to \infty} \frac{1}{n} \Ent(\hat{\sfP}_t^n|\Pi_{n,\hat \nu_t})=0, \qquad \fA t\ge 0, \]
where $\Pi_{n,\nu}\in \calP(\Gamma)$ stems from the Poisson measure $\pi_{n,\nu}$with intensity measure $\nu$, i.e.\ 
 \[
\pi_{n,\nu}:=\frac{1}{e^{n \nu(\calT)}-1} \sum_{N=1}^{\infty} \frac{n^N}{N!} \nu^{\otimes N}.
 \]
\end{thm}

To the authors' knowledge, this is the first entropic propagation of chaos result for bounded competition kernels over compact Polish spaces, under the assumption of detailed balance.

\subsubsection*{\textbf{Comments}} 

We have given an overview of the generalized gradient structures that we introduced for the forward Kolmogorov equation of our underlying interacting particle system and eluded to how this sequence of structures converges to a gradient structure induced by the mean-field limit. Throughout, we assumed bounded measurable rates $m,c$ over a compact Polish space $\calT$ satisfying the detailed balance condition $m(x,y)=c(x,y)$ and $c(x,x)=0$ for all $x,y\in \calT$, and we would like to briefly touch on possible relaxations of these assumptions. 

First, for the limit inferior in Theorem  \ref{thm:main_conv1}, there is a technical issue concerning the possible non-continuity of the competition kernel $c$, which we resolve by an approximation argument from large deviation theory \cite{Hoeksema2020}, see Appendix \ref{s:nonc}. This argument can be straightforwardly extended to unbounded rates $m$ and $c$ under certain exponential integrability estimates with respect to the reference measure $\gamma$. However, the uniqueness of solutions and well-posed of variational formulations would be less clear.  

Moreover, it should be noted that although for brevity and clarity we chose $\calT$ to be compact, many of the listed results carry over to the case of $\calT$ Polish with finite $\gamma$, under suitable choices of topologies and by bootstrapping from the tightness of $\gamma$. For $\sigma$-finite $\gamma$, this is not necessarily the case and would depend strongly on newly constructed estimates on the propagation of tightness. 

A more fundamental restriction is the detailed balance assumption, which is necessary to phrase the variational structures in terms of generalized gradient systems and the evolution in terms of a gradient flow. However, there exist possible extensions and decompositions of variational structures for jump processes that do not assume detailed balance or even complex balance, see for example \cite{Zimmer2018} for an overview. Therefore, in future work, the authors plan to generalize the variational methods outlined here to more general evolutions.

\medskip
\subsubsection*{\textbf{Acknowledgments}} \sloppy{The authors acknowledge support from NWO Vidi grant 016.Vidi.189.102 on "Dynamical-Variational Transport Costs and Application to Variational Evolution".}


\subsection{Notation}
Below we collect some of the notation used throughout this paper.
\begin{longtable}{ll}
  $\calT$   & \mbox{trait space, Assumption \ref{assu:massu}} \\
  $m,c$ &\mbox{mutation/competition kernel, Assumption \ref{assu:massu}}\\
  $\gamma$ &\mbox{reference measure, Assumption \ref{assu:massu}}\\
  $n$ &\mbox{system size, Assumption \ref{assu:massu}}\\
  $\Ent$ & \mbox{relative entropy \eqref{eq:mf_ent}} \\
  $H$ & \mbox{Hellinger distance \eqref{eq:mf_hell}} \\
  $\Psi,\Psi^*$ & \mbox{dual pair \eqref{eq:defpsi},\eqref{eq:defpsis}}  \\
  $\calM^+$ &\mbox{space of finite non-negative measures, with narrow topology} \\
  $\calM^+_{loc}$ &\mbox{space of non-negative Radon measures, with vague topology} \\
  $\Gamma:=\calM^+(\calT)$ & \mbox{state space of measure-valued process} \\
  $\Gamma_n\subset \Gamma$ & \mbox{space of positive atomic measures with common mass $\tfrac{1}{n}$ \eqref{eq:gamman}} \\
  $\kappa^{\pm}_{\nu}=\kappa^{\pm}[\nu]$ & \mbox{measure-dependent birth/death kernels \eqref{eq:defk}} \\
  $\theta_{\nu}$ &\mbox{geometric mean of $\kappa^{+}_{\nu}$ and $\kappa^{-}_{\nu}$, Definition \ref{defi:mf}} \\
  $\mathscr{CE}$ & \mbox{continuity equation for mean-field $\eqref{eq:mf1}$, Definition \ref{defi:mfcont}}\\
  $\calR_{MF},\calF_{MF},\calD_{MF}$ & \mbox{ingredients of EDP-functional $\calI_{MF}$ for $\eqref{eq:mf1}$ , Definition \ref{defi:mf}}\\
  $Q_n,Q_n^*$ &\mbox{generator and dual generator \eqref{eq:defiq} of \eqref{eq:FKEn}} \\
  $\bar \kappa_n$ & \mbox{jump kernel \eqref{eq:defibk} corresponding to \eqref{eq:FKEn}}\\
  $L_n$ &\mbox{rescaled empirical measure map \eqref{eq:defiLn}}\\
  $\pi_n,\Pi_n$ &\mbox{invariant measures for particle system \eqref{eq:defipn} and measure-valued process \eqref{eq:defiPn}}\\
  $\sfT^{n,\pm}$
  &\mbox{creation/annihilation mappings \eqref{eq:defit}} \\
  $\dder^{n,\pm}, \mathrm{div}^{n,\pm}$ &\mbox{discrete $\Gamma_n$-gradient \eqref{eq:defidder} and divergence \eqref{eq:defiddiv}} \\
  $\vartheta_{\sfP}^{\pm}$ & \mbox{expected fluxes \eqref{eq:defief}}\\
  $\Theta_{\sfP}^{n,\pm}$ &\mbox{geometric average $\vartheta_{\sfP}^{\pm}$ along transition, Definition \eqref{defi:cen}} \\
  $\mathsf{CE}_n$ & \mbox{continuity equation for \eqref{eq:FKEn}, Definition \eqref{defi:cen}} \\
  $\calR_{n},\calF_{n},\calD_{n}$ & \mbox{ingredients of EDP-functional $\calI_{n}$ for \eqref{eq:FKEn}, Definition \ref{defi:fke}}\\
  $d_{TV,w}, W$ & \mbox{weighted total variation metric \eqref{eq_fked}/transportation metric \eqref{defi:Wm} over $\calP(\Gamma)$}\\
  $\mathsf{CE}_{\infty}$ &\mbox{continuity equation for $\eqref{eq:liouv2}$, Definition \ref{defi:contli}}\\
  $\calR_{\infty},\calF_{\infty},\calD_{\infty}$ &\mbox{ingredients of EDP-functional $\calI_{\infty}$ for \eqref{eq:liouv}, Definition \ref{defi:liouv}}\\
\end{longtable}

\section{Mean-field system}\label{s:mf}

In this section, we will discuss the gradient-flow formulation of the mean-field equation under the detailed balance condition. Let us first make precise the context of Theorem \ref{thm:imf}, and embed it within the more general statement of Theorem \ref{thm:equivlag} below. 

\medskip

Recall that the trait space $\calT$ is a compact Polish space, and $\Gamma:=\calM^+(\calT)$ is the space of finite non-negative measures over $\calT$ equipped with the narrow topology. Fix a reference measure $\gamma\in \Gamma$, and rates $m,c$ satisfying Assumption \ref{assu:massu}, i.e.\ $m,c\in \calB_b(\calT \times \calT)$ with $m(x,y)=c(y,x)$ for all $x,y\in \calT$, and $c(x,x)=0$ for all $x\in \calT$. The mean-field equation then reads 
\begin{equation}\tag{$\sf MF$}\label{eq:mf1}
    \partial_t \nu_t =\kappa^+[\nu_t]-\kappa^-[\nu_t],
\end{equation}
with measure-dependent birth and death kernels $\kappa^{\pm}:\Gamma \to \Gamma$ given by 
\begin{equation}\label{eq:defk}
\kappa^+[\nu](\dd x) :=   \int_{y\in \calT} c(x,y)\gamma(\dd x)\nu(\dd y) ,\qquad  \kappa^-[\nu](\dd x)  :=  \int_{y\in \calT}  c(x,y) \nu(\dd x) \nu(\dd y).
\end{equation}

Routinely, we will also adopt the shorthand notation $\kappa_\nu^{\pm} := \kappa^{\pm}[\nu]$. Now, setting $c_{\nu}(x):=\int_{\calT}  c(x,y)\, \nu(\dd y)$, it is clear that that $\kappa^+_{\nu}=c_{\nu} \gamma$, $\kappa^-_{\nu}=c_{\nu} \nu$, and the dynamics simplify to 
\begin{equation*}
    \partial_t \nu(\dd x)=c_{\nu}(x)(\gamma(\dd x)-\nu(\dd x)). 
\end{equation*}
Strong solutions to \eqref{eq:mf1} in either total variation or appropriate $L^1$ spaces follow straightforwardly via classical methods, see Section \ref{ss:mf_sol}. 

\medskip
The total variation norm $\|\cdot\|_{TV}$ on $\calM(\calT)$ is defined as
\[
\|\mu\|_{TV}:=\sup  \left\{ \int_{\calT} f \, \dd \mu : f\in \calB_b(\calT), \, \|f\|_{\infty}\leq 1 \right\}, \qquad \mu\in \calM(\calT),
\]
and the squared Hellinger distance $H^2$ is given by
\begin{equation}\label{eq:mf_hell}
    H^2(\nu,\eta):=\frac{1}{2}\int_{\calT} \left(\sqrt{\frac{\dd \nu}{\dd \sigma}}-\sqrt{\frac{\dd \mu}{\dd \sigma}}\right)^2 \dd \sigma,
\end{equation}
with $\sigma$ a measure dominating both $\mu$ and $\nu$. Note that the definition \eqref{eq:mf_hell} is independent of the choice for the dominating measure $\sigma$, and $\sigma=\nu+\eta$ is always admissible.

Moreover, recall the entropy function $\phi: \R_{\geq 0}\to  \R_{\geq 0}$ and its Legendre dual $\phi^*:\mathbb{R}\to\mathbb{R}$  by
\begin{equation*}
    \phi(s):=s \log s -s+1, \qquad \phi^*(z):=e^{z}-1,
\end{equation*}
and the relative entropy of $\nu$ with respect to $\mu$ as
\begin{equation}\label{eq:mf_ent}
    \Ent(\nu|\mu):=\left\{\begin{aligned} &
    \int_{\calT} \phi\left(\frac{\dd \nu}{\dd \mu} \right) \dd \mu, &&\quad \mbox{if $\nu\ll \mu$},\\
    &+\infty &&\quad \mbox{otherwise.}
    \end{aligned}\right.
\end{equation}

We will consider curves satisfying the continuity equation
\begin{equation}\label{contmf}\tag{$\mathscr{CE}$} 
    \partial_t \nu_t = \lambda_t^+-\lambda_t^-,
\end{equation}
in an appropriately weak sense. 

\begin{defi}[Mean-field continuity equation]\label{defi:mfcont}
A triple $(\nu,\lambda^+,\lambda^-)$ satisfies the mean-field continuity equation $\mathscr{CE}$ if
\begin{enumerate}
    \item the curve $[0,T]\ni t\mapsto\nu_t\in \Gamma$ is absolutely continuous with respect to $\|\cdot\|_{TV}$,
    \item the Borel family $(\lambda_t^\pm)_{t\in[0,T]}\subset\Gamma$ satisfies $\int_0^T \|\lambda_t^{\pm}\|_{TV} \, \dd t<\infty$,
    \item for every $s,t\in [0,T]$ and all $f\in C_b(\calT)$
     \begin{equation*}
        \int_{\calT} f \dd \nu_t - \int_{\calT} f \dd \nu_s = \int_s^t \left(\int_{\calT} f \dd \lambda_r^+-\int_{\calT} f \dd \lambda_r^- \right) \, \dd r, \qquad \mbox{for all $s,t$ with $0\leq s,t\leq T$.}
    \end{equation*}
\end{enumerate}
\end{defi}
We will refer to $\lnet=\lambda^+-\lambda^-$ as the \emph{net flux}. 

\begin{remark}
When seen as approximations of particle systems the birth/death fluxes $\lambda^{\pm}_t$ represent the \emph{observed} amount of mass being created/annihilated around a certain point, and $\nu_t$ represents the density of the particles, while $\kappa_{\nu}^{\pm}$ correspond to the \emph{expected} birth and death fluxes of the BPDL model. 
\end{remark}

\begin{remark}[Time-regularity]
As we will see in Lemma \ref{lm:mf_timereg}, if there exist a common dominating measure for $\{\nu_t,\lambda^+_t,\lambda_t^-\}_{t\in [0,T]}$ then the continuity equation holds in a strong sense: $\nu_t$ is an a.e.\ differentiable map from $[0,T]$ to $(\Gamma,\|\cdot\|_{TV})$ and 
\begin{equation*}
    \partial_t \nu_t=\lambda_t^+-\lambda_t^-, \qquad \mbox{for a.e.\ $t\in [0,T]$}.
\end{equation*}
\end{remark}

\begin{defi}\label{defi:mf}
Let $\theta_{\nu}$ be the geometric average of $\kappa_{\nu}^+$ and $\kappa_{\nu}^-$, i.e.\ 
\begin{equation*}
     \dd \theta_{\nu}:=\sqrt{\frac{\dd \kappa_{\nu}^+}{\dd \sigma} \frac{\dd \kappa_{\nu}^-}{\dd \sigma}} \dd \sigma, 
\end{equation*}
for any dominating measure $\sigma$. We define the following objects:
\begin{itemize}
    \item The dissipation potential $\calR_{MF}:\Gamma^3 \to [0,+\infty]$,
\[
    \calR_{MF}(\nu,\lambda^+,\lambda^-):=\Ent(\lambda^+|\theta_{\nu})+\Ent(\lambda^-|\theta_{\nu}),
\]
and the dual dissipation potential $\calR^*_{MF}:\Gamma\times \calB_b(\calX)^2 \to [0,+\infty]$,
\[
\calR^*_{MF}(\nu,w^+,w^-):=\int_{\calT} (e^{w^{+}}-1)\,\dd \theta_{\nu}+\int_{\calT} (e^{w^{-}}-1)\,\dd \theta_{\nu}.
\]
\item The free energy $\calF_{MF}:\Gamma\to [0,+\infty]$,
\[
  \calF_{MF}(\nu):=\tfrac{1}{2} \Ent(\nu|\gamma),
\]
and Fisher information $\calD_{MF}:\Gamma\to [0,+\infty]$,
\[
    \calD_{MF}(\nu):=\left\{\begin{aligned}
    &2H^2(\kappa_{\nu}^+,\kappa_{\nu}^-),&& \qquad \mbox{if $\nu\ll\gamma,$}\\
    &+\infty,&& \qquad \mbox{otherwise.}
    \end{aligned}\right.
\]
\item The \emph{EDP-functional} $\calI_{MF}:\mathscr{CE}\to [0,+\infty]$ for all curves with $\calF_{MF}(\nu_0)<\infty$
\begin{equation}\label{eq:mf_edpf}
    \calI_{MF}(\nu,\lambda^+,\lambda^-):=\int_0^T \calR_{MF}(\nu_t,\lambda_t,\lambda^-_t) \, \dd t + \calF(\nu_T)-\calF(\nu_0)+\int_0^T \calD_{MF}(\nu_t) \, \dd t.
\end{equation}
\end{itemize}

\end{defi}

\begin{remark}
Since $\theta_{\nu}(\calT)<\infty$ by Lemma \ref{lm:mf_est} all objects above are well-defined, and it is straightforward to verify via the dual representation of the entropy that $\calR_{MF}, \calR^*_{MF}$ are truly dual objects in the sense that 
\[ \calR(\nu,\lambda^+,\lambda^-):=\sup_{w^{\pm} \in \calB_b(\calT)} \left\{ \int_{\calT} w^+ \dd \lambda^++\int_{\calT} w^- \dd \lambda^- -\calR^*(\nu,w^+,w^-) \right\},  \]
and vice versa.
\end{remark}
\begin{remark}
If $\nu\ll \gamma$ with $\dd \nu=u \dd \gamma$, note that $\dd \theta_{\nu}=c_{\nu} \sqrt{u}\,  \dd \gamma$, and that the Fisher information simplifies to 
\[
    \calD_{MF}(\nu)= \int_{\calT} c_{\nu} \left(\sqrt{u}-1\right)^2 \dd \gamma.
\]
\end{remark}

We are now able to fully state the variational characterization of strong solutions to the mean-field equation \eqref{eq:mf}. 

\begin{thm}\label{thm:equivlag} For any $(\nu,\lambda^+,\lambda^-)\in \mathscr{CE}$ with $\calF_{MF}(\nu_0)<\infty$, we have $\calI_{MF}(\nu,\lambda^+,\lambda^-)\geq 0$ and 
\[ \calI_{MF}(\nu,\lambda^+,\lambda^-)=0 \iff \left\{ \begin{aligned}
\quad &\mbox{$\nu_t$  is the unique strong solution to \eqref{eq:mf}}, \quad \\
\quad \lambda^{\pm}_t&=(\kappa^{\pm}_{\nu_t}) \quad \mbox{for a.e.\ $t\in [0,T]$.} \quad
\end{aligned} \right. \]

Moreover, whenever $\calF_{MF}(\nu_0)<\infty$ and $\calI_{MF}(\nu,\lambda^+,\lambda^-)<\infty$ the chain rule for $\calF_{MF}$ holds: $\calF_{MF}(\nu_t)$ is absolutely continuous and 
\begin{equation*}
    \frac{\dd}{\dd t} \calF_{MF}(\nu_t)=\tfrac{1}{2} \int_{\calT} \log \frac{\dd \nu_t}{\dd \gamma}\,\dd (\lambda^+_t-\lambda^-_t) \, \dd t, \qquad \mbox{for a.e.\ $t\in [0,T]$}.
\end{equation*}
\end{thm}

The proof of Theorem \ref{thm:equivlag} is postponed to Section \ref{ss:mf_grad}, where we establish the main technical ingredient, namely the chain rule for the entropy functional.

\begin{remark}
The non-negativity of $\calI_{MF}$ and the fact that null-minimizers are solutions to \eqref{eq:mf} is related to the formal equivalence 
\begin{equation*}
    \calI_{MF}(\nu,\lambda^+,\lambda^-)=\int_0^T \calL(\nu_t,\lambda_t^+,\lambda^-_t) \,\dd t,
\end{equation*}
where $\calL$ is the so-called \emph{Lagrangian} given by
\[ \calL(\nu,\lambda^+,\lambda^-):=\Ent(\lambda^+|\kappa_{\nu}^+)+\Ent(\lambda^+|\kappa_{\nu}^-).\]
Note that $\calL$ is non-negative and zero if only if $\lambda^{\pm}=\kappa_{\nu}^{\pm}$. 
Although we do not prove the full equivalence in this work, it does play a role in the intuition and motivation behind the EDP-functional $\calI_{MF}$ with the Lagrangian $\calL$ stemming from a large deviation perspective, as seen in Appendix \ref{s:ldpmot}. 
\end{remark}

\subsection{A priori estimates}\label{ss:mf_est}\,

In this section, we will collect some elementary estimates and results that are either necessary for the well-posedness of the mean-field equation and the corresponding gradient structure, or necessary to do the same for the Liouville equation in Section \ref{s:liouv}. 

Let $\Psi^*$ be given as 
\begin{equation}\label{eq:defpsis}
    \Psi^*(z):=2 (\cosh(z)-1) = e^z+e^{-z}-2, 
\end{equation}
and its dual $\Psi:=(\Psi^*)^*$
\begin{equation}\label{eq:defpsi}
    \Psi(s)=s \log \left(\frac{s+\sqrt{s^2+4}}{2}\right)-\sqrt{s^2+4}+2 
\end{equation}

\begin{lm}\label{lm:mf_est}
Let $M:=\|c\|_{\infty} (1+\gamma(\calT))$. Then the following estimates hold:
\begin{enumerate}[label=(\roman*)]
    \item The measures $\kappa_{\nu}^{\pm}$ and $\theta_{\nu}$ are finite: 
    \begin{equation}\label{eq:mf_mb1b}
          \kappa^{\pm}_{\nu}(\calT)\leq M (1+\nu(\calT)^2),
    \end{equation}
    and
\begin{equation}\label{eq:lmthetaa1}
    \theta_{\nu
}(\calT)\leq M (1+\nu(\calT)^2)
\end{equation}
    \item For any birth/death fluxes $\lambda^{\pm}\in \calM^+(\calT)$, net flux $\lnet=\lambda^+-\lambda-$, and $w^{\pm}, w\in \calB(\calT)$,  
    \begin{equation*}
          \begin{aligned}
     \int_{\calT} |w^{\pm}| \,\dd \lambda^{\pm}&\leq \Ent(\lambda^{\pm}|\theta_{\nu})+\int_{\calT} \Psi^*(w) \, \dd \theta_{\nu} + \theta_{\nu}(\calT), \\
     \int_{\calT} |w|\,\dd |\lnet| \, &\leq \calR_{MF}(\nu,\lambda^+,\lambda^-)+\int_{\calT} \Psi^*(w) \, \dd \theta_{\nu}, 
    \end{aligned}
    \end{equation*}
\item For any birth/death fluxes $\lambda^{\pm}\in \Gamma$,
\begin{equation}\label{eq:mfactionb1}
    \phi\left(\frac{\lambda^{\pm}(\calT)}{M(1+\nu(\calT)^2)} \vee 1\right) M\leq \calR_{MF}(\nu,\lambda^+,\lambda^-)
\end{equation}
\end{enumerate}
\end{lm}

\begin{remark}
Although the estimate for $\theta_{\nu}$ can be made more precise, namely 
    \begin{equation*}
        \theta_{\nu}(\calT)\leq \|c\|_{\infty} \gamma(\calT)^{1/2} \nu(\calT)^{3/2},
    \end{equation*}
we will not require it for our results. 
\end{remark}

\begin{proof}
\emph{(i)} With $\theta_{\nu}:=\sqrt{\dd \kappa_{\nu}^+/\dd \sigma \, \dd \kappa_{\nu}^-/\dd \sigma} \, \sigma$ for any dominating measure $\sigma$ we have by H\"older's inequality
\begin{equation*}
    \theta_{\nu}(\calT)\leq \sqrt{\kappa_{\nu}^+(\calT)\kappa_{\nu}^-(\calT)}.
\end{equation*}
Note that $\kappa^+_{\nu}(\calT)\leq \|c\|_{\infty} \gamma(\calT) \nu(\calT)$, and $\kappa^-_{\nu}(\calT)\leq \|c\|_{\infty} \nu(\calT)^2$, which provides \eqref{eq:mf_mb1b}. Since $z\leq 1+z^2$ for all $z\geq 0$ \eqref{eq:lmthetaa1} follows directly. 

\emph{(ii)} First, suppose that $w\in \calB_b(\calT)$. Using the elementary inequality  $ e^{|a|}\leq e^a+e^{-a}$ we derive by duality of the entropy
\begin{align*}
    \int_{\calT} |w|\, \dd \lambda^{\pm} &\leq \Ent(\lambda^{\pm}|\theta_{\nu})+ \int_{\calT} (e^{|w|}-1)\, \dd \theta_{\nu}\\
    &\leq \Ent(\lambda^{\pm}|\theta_{\nu})+\int_{\calT} \Psi^*(w) \, \dd \theta_{\nu} + \theta_{\nu}(\calT). 
\end{align*}
Next, fix any measurable function $w\in \calB(\calT)$ and set its $k$-truncation $w_k:=\max\{\min\{w,k\},-k\}$. Since $\Psi^*$ is even and monotone, by dominated convergence applied to the left-hand side and monotone convergence to the right-hand side, the inequality holds for $w$ as well.

\emph{(iii)} Without loss of generality, suppose that $\calR_{MF}$ is finite. Set $a(\nu):=(1+\nu(\calT)^2)^{-1}$, and note that $0\leq a(\nu)\leq 1$. With $\tilde \phi(s):=\phi(s\vee 1)$ the monotone relaxation of $\phi$, we then have the following chain of inequalities,
\begin{equation*}
    \begin{aligned}
       \int_{\calT} \phi\left(\frac{\dd\lambda^{\pm}}{\dd\theta_{\nu}}\right) \dd \theta_{\nu}&\geq \int_{\calT} \tilde \phi\left(\frac{\dd\lambda^{\pm}}{\dd\theta_{\nu}}\right) \dd \theta_{\nu} \\
&\geq \int_{\calT} \tilde \phi\left(\frac{\dd(a(\nu) \lambda^{\pm})}{\dd(a(\nu) \theta_{\nu})}\right) \dd (a(\nu) \theta_{\nu})\\
&\geq \tilde \phi\left(\frac{a(\nu) \lambda^{\pm}(\calT)}{a(\nu) \theta_{\nu}(\calT)}\right) a(\nu) \theta_{\nu}(\calT),
    \end{aligned}
\end{equation*}
where the last inequality follows from Jensen's inequality. By convexity of $\tilde \phi$ and $\tilde \phi(0)=0$ the latter expression is monotone in $\theta_{\nu}(\calT)$, and hence by \eqref{eq:lmthetaa1} we find 
\begin{equation*}
    \tilde \phi\left(\frac{\lambda^{\pm}(\calT)}{M(1+\nu(\calT)^2)}\right) M\leq \calR_{MF}(\nu,\lambda^+,\lambda^-).
\end{equation*}
\end{proof}

We will briefly state the improvement of regularity in time of $\nu_t$ if there exists a common dominating measure. The proof is similar to Corollary 4.14 of \cite{PRST2020} and therefore omitted here. 

\begin{lm}\label{lm:mf_timereg}
Let $(\nu,\lambda^+,\lambda^-)\in \mathscr{CE}$ and suppose that there exists a measure $\ell\in \Gamma$ such that $\nu_t,\lambda_t^{\pm}\ll \ell$ for all $t\in [0,T]$.

Then there exists an absolutely continuous and a.e.\ differentiable map $u:[0,T]\to L^1(\calT,\ell)$ and maps $g^{\pm}:[0,T]\to L^1(\calT,\ell)$ such that $u_t=\dd \nu_t/\dd \ell$, $g_t^{\pm}=\dd \lambda_t^{\pm}/\dd \ell $ and 
\begin{equation*}
    \partial_t u_t(x)=g_t^+(x)-g_t^-(x), \qquad \mbox{for a.e.\ $t\in [0,T]$}.
\end{equation*}
In particular, the continuity equation holds in the strong sense, namely that $\nu_t$ is an a.e.\ differentiable map from $[0,T]$ to $(\Gamma,\|\cdot\|_{TV})$ and
\begin{equation*}
    \partial_t \nu_t=\lambda_t^+-\lambda_t^-, \qquad \mbox{for a.e.\ $t\in [0,T]$}.
\end{equation*}
\end{lm}

Next, we will list two results that are either necessary for the chain rule in Section \ref{ss:fke_chain} or the superposition principle and well-posedness of the continuity equation in Section \ref{s:liouv}.

\begin{lm}\label{lm:mf_lift}
 For any $0\leq a\leq 1$, $z\in \R$
\begin{equation}\label{eq:mf_psi}
    \Psi^*(a z)\leq a^2\Psi^*(z).
\end{equation}

Moreover, for any net flux $\lnet\in \calM(\calT)$, $w\in \calB(\calT)$
\begin{align}
   \Psi\left(\frac{\|\lnet\|_{TV}}{M(1+\nu(\calT))}\right) M &\leq \calR_{MF}(\nu,\lambda^+,\lambda^-).\label{eq:mfactionb2} 
\end{align}
\end{lm}
\begin{proof}
It is straightforward to check that $\Psi^*(z)/z^2$ is monotone increasing for $z\geq 0$, from which the first statement follows. 

Now, for the net flux, it is convenient to go through the dual representation. Set $a(\nu):=(1+\nu(\calT))^{-1}$. By duality, for any $w\in \calB_b(\calT)$ 
\begin{equation}\label{eq:pdual2}
    \begin{aligned}
       \calR_{MF}(\nu,\lambda^+,\lambda^-)\geq a(\nu)\int_{\calT}  w(x) \,\dd \lnet - \int_{\calT} \Psi^*\big(a(\nu)w(x)\big) \,\dd \theta_{\nu}.
    \end{aligned}
\end{equation}
However, by \eqref{eq:mf_psi},
\begin{equation*}
    \int_{\calT} \Psi^*\big(a(\nu)w(x)\big) \,\dd \theta_{\nu} \leq \int_{\calT} \Psi^*\big(w(x)\big) a(\nu)^2 \,\dd \theta_{\nu} \leq M \Psi^*(\|w\|_{\infty}). 
\end{equation*}
Taking the supremum over all $w\in \calB_b(\calT)$ in \eqref{eq:pdual2} we find \eqref{eq:mfactionb2}.
\end{proof}

\begin{lm}\label{lm:contweakf1}
Let $\{f_i\}_{i\in \N} \subset C_b(\calT)$ be a countable and dense set of bounded continuous functions. Suppose $(\nu,\lambda^+,\lambda^-)$ is such that
\begin{enumerate}[label=(\roman*)]
        \item the curve $[0,T]\ni t\mapsto\nu_t\in \Gamma$ is narrowly continuous
    \item $(\lambda_t^\pm)_{t\in[0,T]}\subset\Gamma$ is a Borel family with
    \[ \int_0^T \calR_{MF}(\nu_t,\lambda^+_t,\lambda^-_t) \, \dd t<\infty \]
    \item For all $i \in \N$
     \begin{equation*}
        \int_{\calT} f_i\, \dd \nu_t - \int_{\calT} f_i\, \dd \nu_s = \int_s^t \left(\int_{\calT} f_i \,\dd \lambda_r^+-\int_{\calT} f_i \,\dd \lambda_r^- \right) \dd r, \qquad \mbox{for all $s,t$ with $0\leq s,t\leq T$.}
        \end{equation*}
\end{enumerate}

Then $(\nu,\lambda^+,\lambda^-)\in \mathscr{CE}$, i.e.\ the triple satisfies the mean-field continuity equation.
\end{lm}

\begin{proof}
Since $\nu_t$ is narrowly continuous its mass is uniformly bounded in time, hence let $C:=\sup_{t\in [0,T]} \nu_t(\calT)$. By \eqref{eq:mfactionb1} and monotonicity of $\phi(\cdot 
\vee 1)$ we have for a.e. $t\in [0,T]$, 
\begin{equation*}
    \phi\left(\frac{\lambda_t^{\pm}(\calT)}{M(1+C^2)} \vee 1\right) M\leq \calR_{MF}(\nu_t,\lambda^+_t,\lambda^-_t),
\end{equation*}
and therefore by convexity of $\phi(\cdot 
\vee 1)$
\[ \int_0^T \lambda_t^{\pm}(\calT)<\infty. \]
Since the measures $\lambda_t^{\pm}(\dd x)\, \dd t \in \calM^+([0,T]\times \Gamma)$ are finite, by density of $f_i$ in $C_b(\calT)$ it is clear that for all $f\in C_b(\calT)$ 
     \begin{equation*}
        \int_{\calT} f \,\dd \nu_t - \int_{\calT} f \,\dd \nu_s = \int_s^t \left(\int_{\calT} f \,\dd \lambda_r^+-\int_{\calT} f \,\dd \lambda_r^- \right) \, \dd r, \qquad \mbox{for all $s,t$ with $0\leq s,t\leq T$.}
        \end{equation*}
By a monotone class argument this can be extended to all $f\in \calB_b(\calT)$ and we derive that $\nu_t$ is indeed TV-absolutely continuous and $(\nu,\lambda^+,\lambda^-)\in \mathscr{CE}$.
\end{proof}

\subsection{Strong solutions}\label{ss:mf_sol}

Strong solutions to \eqref{eq:mf1} exist and are unique, and we list the most important properties here. It should be noted that these arguments apply even without the detailed balance condition $m(x,y)=c(y,x)$ and only require both $\|m\|_{\infty}$ and $\|c\|_{\infty}<\infty$ to be finite, but for simplicity, we will restrict ourselves to our framework. Moreover, in all results the time window $T>0$ is arbitrary.

\begin{defi}\label{defi:mf_strong}
A \emph{strong} solution to \eqref{eq:mf1} is any TV-absolutely continuous and a.e.\ differentiable mapping $\nu:[0,T]\to (\Gamma,\|\cdot\|_{TV})$ satisfying
\begin{equation}\label{eq:mf2}
\begin{aligned}
 \partial_t \nu_t(\dd x) &=\kappa_{\nu_t}^+(\dd x)-\kappa_{\nu_t}^-(\dd x)\\
\end{aligned}
\end{equation}
\end{defi}

Recall that  $\kappa^+_{\nu}(\dd x)=\langle c(x,\cdot),\nu \rangle \gamma(\dd x)$ and $\kappa^-_{\nu}(\dd x)=\langle c(x,\cdot),\nu \rangle \nu(\dd x)$.

\begin{remark}\label{rem:mf_weakstrong}
Note that if $\nu$ is a strong solution to \eqref{eq:mf1} automatically $(\nu,\kappa^{+}_{\nu},\kappa_{\nu}^-) \in \mathscr{CE}$. 

Vice versa, if $(\nu,\kappa^{+}_{\nu},\kappa_{\nu}^-) \in \mathscr{CE}$ then $\nu_t$ is a strong solution. Namely, any TV-absolutely continuous curve $\nu_t$ possesses a common dominating measure $\ell \in \Gamma$, which implies $\kappa_{\nu_t}^{\pm}\ll \ell+\gamma$. By Lemma \ref{lm:mf_timereg} the curve $\nu$ is indeed a a.e.\ differentiable mapping to $(\Gamma,\|\cdot\|_{TV})$
\end{remark}
\begin{lm}\label{lm:mf_strongsol}
For any $\bar \nu\in \Gamma$ there exist a unique strong solution $\nu_t$ to \eqref{eq:mf1} such that $\nu_0=\bar \nu$. 

Moreover, if $\bar \nu\ll \gamma$, then also $\nu_t\ll \gamma$ for all $t\in [0,T]$. 
\end{lm}

The proof is an adaptation from \cite[Proposition 7.2]{Fournier2004}, which is stated for Lebesgue absolutely continuous measures over $\calT=\R^d$. In short, the linear dependence of the birth flux on the mass of $\nu$ gives a bound on this mass uniform in time, in which case both $\kappa^{\pm}_{\nu}$ are Lipschitz in $\nu$ on $(\Gamma,\|\cdot\|)$, and classical existence theory can be applied. 

\begin{proof}
First, note that for the linear case of
\begin{equation*}
    \partial_t \nu_t(\dd x) = b_t(\dd x)-c_t(x)\nu_t(\dd x),
\end{equation*}
with $c_t\in \calB_b$ uniformly bounded and $b_s \in \Gamma$ with $\int_0^T \|b_s\|_{TV} \, \dd t<\infty$ with a common dominating measure, it is easy to verify that a unique strong non-negative solution exists and is given by   
\begin{equation*}
    \nu_t:=e^{-\int_0^t c_s(x) \dd s}\left(\int_0^t b_s e^{\int_0^s c_r \, \dd r} \dd s+\nu_0\right).
\end{equation*}
We now set $\nu^0_t:=\bar \nu$ for all $t\in[0,T]$, and perform the implicit Picard iteration
\begin{equation*}
    \partial_t \nu^{k+1}_t(\dd x)=\langle c(x,\cdot),\nu^k)\rangle \gamma(\dd x) -\langle c(x,\cdot),\nu^k \rangle \nu^{k+1}(\dd x), \qquad \nu^{k+1}_0:=\bar \nu,
\end{equation*}
i.e.\ $\nu^{k+1}=(\mathcal{G}\nu^k)$ with
\begin{equation*}
    (\mathcal{G}\nu)_t(\dd x):=e^{-\int_0^t \langle c(x,\cdot),\nu_r \rangle \dd s}\left(\int_0^t \langle c(x,\cdot),\nu_s)\gamma(\dd x)\rangle e^{\int_0^s \langle c(x,\cdot),\nu_r \rangle \, \dd r} \dd s+\bar \nu(\dd x)\right).
\end{equation*}
It is straightforward to check that for all $t\in [0,T]$
\[\sup_{k\ge 1}  \nu^k_t(\calT) \leq e^{\|c\|_{\infty} \gamma(\calT)t}\bar \nu(\calT)\leq e^{\|c\|_{\infty} \gamma(\calT)T}\bar \nu(\calT)=:C.\]

We will show that $\calG$ is contractive under a suitable metric on the space of curves with initial data $\bar \nu$ and mass bounded by $C$. This implies there exists a $TV$-absolutely continuous curve $\nu$ such that 
\begin{equation*}
    \nu_t-\nu_s = \int_s^t \left(\kappa_{\nu_r}+\kappa_{\nu_r}^-\right) \, \dd r, \qquad \mbox{for all $s,t$ with $0\leq s,t\leq T$.}
\end{equation*}
Moreover, since in the iterations $\nu^k\ll \bar \nu+\gamma$ for all $\nu$ it is clear that we obtain strong solutions in $L^1(\bar \nu+\gamma)$. In particular, for $\bar \nu\ll \gamma$ we have $\nu_t\ll \gamma$ for all $t\in [0,T]$ as well.

Now, note that $\langle c(x,\cdot),\nu\rangle$ depends Lipschitz on $\nu$ in $(\Gamma,\|\cdot\|_{TV})$ due to the uniform bound on mass. This implies that there exists a constant $K$ such that for any two admissible curves $\nu,\tilde \nu$:
\begin{equation*}
       \|(\calG \nu)_t-(\calG \tilde \nu)_t\|_{TV}\leq K \int_0^t \|\nu_s-\tilde \nu_s\|_{TV} \, \dd s, \qquad \mbox{for all $t\in [0,T]$.}
\end{equation*}
Hence, by a Gronwall-type argument, we find that for any $\eps>0$ for all $t\in [0,T]$
\[
   \|(\calG \nu)_t-\calG (\tilde \nu)_t\|_{TV} e^{-(K+\eps) t} \le\frac{K}{K+\eps}\left(\sup_{s\in[0,T]} \|\nu_s-\tilde \nu_s\|_{TV} e^{-(K+\eps) s} \right),
\]
thus yielding the contraction required to apply the Banach fixed-point theorem.
\end{proof}

Finally, for the use in entropic propagation chaos of Theorem \ref{thm:edp_prop}, it is convenient to characterize the conditions for which $u_t$ is bounded from above and below. The following statement follows directly from a Gronwall-type argument. 

\begin{lm}\label{lm:mfsolb}
Suppose $\nu_0$ is such that $C^{-1} \leq \dd \nu_0/\dd \gamma (x) <C$ for some constant $C>0$ and all $x\in \calT$. Then there exist a constant $C_T>0$ such that for the corresponding solution
\begin{equation*}
    C_T^{-1} \leq \frac{\dd \nu_0}{\dd \gamma} (x) <C_T, \qquad \fA x\in \calT, \fA t\in [0,T].
\end{equation*} 
\end{lm}

\subsection{Variational characterization}\label{ss:mf_grad}

We will now prove the non-negativity of our EDP-functional $\calI_{MF}$ and the characterization of strong solutions to \eqref{eq:mf1} as minimizers of $\calI_{MF}$. To do so we first need the prove the chain rule for the free energy $\calF_{MF}$ along curves with finite $\calI_{MF}$. 

\medskip
There is an important technical issue concerning the Fisher information, in the sense that on curves with finite $\calI_{MF}$ the chain rule inequality holds for the following replacement:
\begin{equation*}
    \calD^-_{MF}(\nu):=\int_{\calT} \Psi^*\left(\frac{1}{2}\log u\right) \dd \theta_{\nu} = \int_{u>0} c_{\nu}(x)\left(\sqrt{u}-1\right)^2 \dd \gamma, 
\end{equation*}
for any $\nu\ll \gamma$ with $u:=\dd \nu/\dd \gamma$. Note that $0\leq \calD^-_{MF}(\nu)\leq \calD_{MF}(\nu)$ and $\calD^-_{MF}=\calR_{MF}^*(\partial_{\nu} \calF_{MF})$.

We will see the same principle arise in Section \ref{s:fke} for the variational characterization of the forward Kolmogorov equation, which is also observed in \cite[Section 5]{PRST2020}.

\begin{lm}\label{lm:mf_chain}
For any curve $(\nu,\lambda^+,\lambda^-)\in \mathscr{CE}$ with $\calF_{MF}(\nu_0)<\infty$ and $\calI_{MF}(\nu,\lambda^+,\lambda^-)<\infty$ it holds that $[0,T]\ni t\mapsto \calF_{MF}(\nu_t)$ is absolutely continuous and a.e. differentiable with
\[ 
 \frac{\dd \,}{\dd\, t} \calF_{MF}(\nu_t)=\frac{1}{2} \int_{\calT} \log \left(\frac{\dd \nu_t}{\dd \gamma}\right)\,\dd \lnet_t, \qquad \mbox{for a.e.\ $t\in [0,T]$}.
\]
Moreover, for such a curve 
\[\calI_{MF}(\nu,\lambda^+,\lambda^-)\geq \calI_{MF}^-:=\int_0^T \left(\calR_{MF}(\nu_t,\lambda^+_t,\lambda^-_t) +\frac{1}{2} \int_{\calT} \log \frac{\dd \nu_t}{\dd \gamma}\,\dd \lnet_t+\calD^-_{MF}(\nu_t)\right) \, \dd t \geq 0.\]
\end{lm}

\begin{remark}
In fact, for such curves, for a.e. $t$ both the terms 
\begin{equation*}
    \int_{\calT} \log \left(\frac{\dd \nu_t}{\dd \gamma}\right)\,\dd \lambda^{\pm}_t,
\end{equation*}
will be finite, and hence 
\[ \int_{\calT} \log \left(\frac{\dd \nu_t}{\dd \gamma}\right)\,\dd \lambda_t=\int_{\calT} \log \left(\frac{\dd \nu_t}{\dd \gamma}\right)\,\dd \lambda_t^+-\int_{\calT} \log \left(\frac{\dd \nu_t}{\dd \gamma}\right)\,\dd \lambda_t^-,\]
\end{remark}

\begin{proof}
Fix any curve $(\nu,\lambda^+,\lambda^-)\in \mathscr{CE}$ with $\calF_{MF}(\nu_0)<\infty$. We will show that whenever $\calI_{MF}<\infty$ the mapping $t \mapsto \Ent(\nu_t|\gamma)$ is absolutely continuous and satisfies the chain rule, i.e.\ 
\begin{equation*}
    \frac{\dd \, \Ent(\nu_t|\gamma)}{\dd\, t}=\int_{\calT} \log \left(\frac{\dd \nu_t}{\dd \gamma}\right)\dd (\lambda^+_t-\lambda^-_t) \, \dd t, \qquad \mbox{for a.e.\ $t\in [0,T]$}.
\end{equation*}

Suppose that $\calI_{MF}<\infty$. Since $\Ent$ is bounded from below, $\Ent(\nu_0|\gamma)<\infty$ implies that
\begin{equation*}
\int_0^T \calR_{MF}(\nu_t,\lambda^+_t,\lambda^-_t)\, \dd t<\infty, \qquad \int_0^T \calD_{MF}(\nu_t) \, \dd t < \infty.
\end{equation*}
In particular for a.e.\ $t\in [0,T]$ it holds that $\nu_t\ll \gamma$, $\lambda_t^{\pm}\ll \theta_{\nu_t}$, and in turn $\theta_{\nu_t} \ll \gamma$. In fact, due to TV-continuity of $\nu_t$, we have $\nu_t\ll \gamma$ for all $t\in[0,T]$. Moreover, $\int_0^T \lambda^{\pm}_t(\calT)<\infty$ and $\sup_t \nu_t(\calT)<\infty$.

Setting $u_t:=\dd \nu_t/\dd \gamma$, we have
\[
\dd \theta_{\nu_t}=c_{\nu_t} \sqrt{u_t} \,\dd \gamma,
\]
and in particular $\theta_{\nu_t}(\{u_t=0\})=0$. Similarly, $\lambda_t^{\pm}\ll\theta_{\nu_t}$ for a.e.\ $t$ and hence $u_t>0$ for $\lambda^{\pm}_t,\lnet_t$-a.e.\ $x$ for such $t$ as well. Furthermore, since for a.e. $t$ we have $\lambda_t^{\pm}\ll\theta_{\nu_t}\ll \gamma$ we find by Lemma \ref{lm:mf_timereg} that $u: [0,T]\to L^1(\calT,\gamma)$ is absolutely continuous and differentiable at a.e. $r\in [0,T]$.  

Consider any such $r$ with $\calR_{MF}(\nu_r,\lambda_r^+,\lambda_r^-),\calD(\nu_r)<\infty$. By Lemma \ref{lm:mf_est}, for any $w\in \calB_b(\calT)$, 
\begin{equation}\label{eq:equivlagb2}
\left| \int_{\calT} w \,\dd \lnet_r \right| \leq  \int_{\calT}|w| \,\dd |\lambda_r| \leq \calR_{MF}(\nu_r,\lambda^+_r,\lambda^-_r)+\int_{\calT} \Psi^*(w) \dd \theta_{\nu_r}.
\end{equation}
Now let $\phi_m$ be the convex and uniformly Lipschitz regularizations of $\phi$ constructed by using the truncations $\phi_m':=[\phi']_m=\max\{\min\{\phi,m\},-m\}$ and $\phi(s):=\int_1^s \phi_m'(z)\, \dd z$. Note that $\phi_m'$ converges pointwise to $\phi'$, and both $\phi_m$ and $|\phi'_m|$ converge monotonically to $\phi$ and $|\phi'|$ respectively. 

Moreover, note that $\phi'(u_r)=\log u_r$ is $\theta_{\nu_r}$-a.e.\ finite, and similarly $\lambda^{\pm}_r$-a.e.\ as well. Therefore, since $\Psi^*$ is even and monotone on $\R_{\geq 0}$ we derive
\begin{equation*}
\begin{aligned}
    \int_{\calT} \Psi^*(\tfrac{1}{2}\phi'_m(u_r)) \dd \theta_{\nu_r} &\leq \int_{\calT} \Psi^*(\tfrac{1}{2}\phi'(u_r)) \dd \theta_{\nu_r}= \calD^-_{MF}(\nu_r).\\
\end{aligned}
\end{equation*}
Recall that $\calD^-_{MF}(\nu_r)\leq \calD_{MF}(\nu_r)$. By substituting $w=\tfrac{1}{2}\phi_m'$ in \eqref{eq:equivlagb2} we find
\begin{equation*}
  \frac{1}{2}\int_{\calT}  \phi_m'(u_r) \,\dd \lnet_r\leq \frac{1}{2} \int_{\calT}   |\phi_m'(u_r)|\, \dd |\lnet_r|\leq \calR_{MF}(\nu_r,\lambda^+_r,\lambda^-_r)+\calD^-_{MF}(\nu_r),
\end{equation*}
and after a monotone convergence argument
\begin{equation}\label{eq:equivlagb3b}
  \frac{1}{2}\int_{\calT}  \phi'(u_r) \, \dd \lnet_r\leq \frac{1}{2} \int_{\calT}   |\phi'(u_r)|\,\dd |\lnet_r|\leq \calR_{MF}(\nu_r,\lambda^+_r,\lambda^-_r)+\calD^-_{MF}(\nu_r).
\end{equation}

Note that for every $m$ the function $\phi_m$ is smooth and uniformly Lipschitz, thus the functional $\int \phi_m(u_r) \,\dd \gamma$ is $\|\cdot\|_{TV}$-Lipschitz continuous and hence absolutely continuous by TV-regularity of $\nu_r$. Moreover, since $\lambda_r^{\pm}\ll \gamma$ and $u_r$ is a.e.\ differentiable in $L^1(X,\gamma)$ it is straightforward to check that
\begin{equation*}
    \int_{\calT} \phi_m(u_t) \,\dd \gamma-\int_{\calT} \phi_m(u_s) \,\dd \gamma=\int_s^t \int_{\calT} \phi_m'(u_r)\,\dd (\lambda^+_r-\lambda^-_r) \, \dd r, \qquad \mbox{for all $s,t\in [0,T]$}.
\end{equation*}
Therefore, since $\Ent(\nu_0|\gamma)$ is finite by assumption and the functionals $\int \phi_m(u_t)\, \dd \gamma$ converge monotonically to $\Ent(\nu_t|\gamma)$, we find 
\begin{equation*}
    \begin{aligned}
      \frac{1}{2} \left|\int_{\calT} \phi(u_t) \,\dd \gamma-\int_{\calT} \phi(u_0) \, \dd \gamma\right|&\leq \frac{1}{2}\limsup_{m\to \infty} \int_0^T \int_{\calT} |\phi_m'(u_r)|\,\dd |\lambda_r^+-\lambda_r^-|\, \dd r \\
       &\leq \int_0^T \left(\calR_{MF}(\nu_r,\lambda_r^+,\lambda_r^-)+\calD^-_{MF}(\nu_r)\right) \, \dd r.
    \end{aligned}
\end{equation*}
In particular $\Ent(\nu_t|\gamma)$ is finite for all $t\in [0,T]$, and after repeating the argument for $s,t\in [0,T]$ we conclude by a dominated convergence argument that
\begin{equation*}
          \int_{\calT} \phi(u_t)\, \dd \gamma-\int_{\calT} \phi(u_s) \, \dd \gamma
       =\int_s^t \int_{\calT} \phi'(u_r) \,\dd \lnet_r  \, \dd r,
\end{equation*}
and 
\begin{equation*}
\begin{aligned}
    \calI_{MF} &= 
  \int_0^T \left(\calR_{MF}(\nu_t,\lambda^+_t,\lambda_t^-)+\frac{1}{2}\int_{\calT} \phi'(u_t)\, \dd \lnet_t+\calD_{MF}(\nu_t) \right) \, \dd t \\
   &\geq \int_0^T \left(\calR_{MF}(\nu_t,\lambda^+_t,\lambda_t^-)+\frac{1}{2}\int_{\calT} \phi'(u_t)\, \dd \lnet_t+\calD^-_{MF}(\nu_t) \right) \, \dd t \geq 0.
\end{aligned}
\end{equation*}
\end{proof}

We are now finally in a position to prove Theorem \ref{thm:equivlag}. With the chain rule above all that remains is on one hand showing that $\calI^-_{MF}(\nu,\lambda^+,\lambda_t^-)=0$ implies that $\lambda_t^{\pm}=\kappa_{\nu_t}^{\pm}$ for a.e. $t$, and on the other hand, showing that if $\nu$ is a strong solution it holds that $\calI^-_{MF}(\nu,\kappa_{\nu}^+,\kappa_{\nu}^-)=0$ and $\calD^-_{MF}=\calD_{MF}$ for a.e. $t\in [0,T]$. The second part again involves proving a chain rule, but now along the solution curve. 

\begin{proof}[Proof of Theorem \ref{thm:equivlag}]
First, consider any $(\nu,\lambda^+,\lambda^-)\in \mathscr{CE}$ with $\calF_{MF}(\nu_0)<\infty$, and $\calI_{MF}=0$. By Lemma \ref{lm:mf_chain}, 
\begin{equation*}
  \int_0^T \left(\calR_{MF}(\nu_t,\lambda^+_t,\lambda_t^-)+\frac{1}{2}\int_{\calT} \phi'(u_t)\, \dd \lnet_t+\calD^-_{MF}(\nu_t) \right) \, \dd t = 0.
\end{equation*}
Now, recall that $\dd \theta_{\nu}=c_{\nu} \sqrt{u}\, \dd \gamma$. Setting $g_t^{\pm}:=\dd \lambda^{\pm}_t/\dd \theta_{\nu}$, it holds that $\log (u_t)\, g_t^{\pm}<\infty$ for $\theta_{\nu_t}$-a.e.\ $x$ and a.e.\ $t$, and by the inequality \eqref{eq:equivlagb3b} that \ $|\log u_t|\, |g_t^+-g_t^-|$ is $\theta_{\nu_t}$-integrable. Therefore, by straightforward algebraic manipulations we find that for a.e.\ $t$,
\begin{equation*}
\begin{aligned}
   &\calR_{MF}(\nu_t,\lambda^+_t,\lambda_t^-)+\frac{1}{2}\int_{\calT} \phi'(u_t)\, \dd \lnet_t+\calD_{MF}^-(\nu_t)\\
   &=\int_{\calT} \left(\phi(g^+_t)+\tfrac{1}{2} \log(u_t) g^+_t+ \phi^*(\tfrac{1}{2}\log u_t) + \phi(g^-_t)-\tfrac{1}{2} \log(u_t) g^-_t+ \phi^*(-\tfrac{1}{2}\log u_t) \right) \dd \theta_{\nu_t}.
\end{aligned}
\end{equation*}
Due to the duality between $\phi$  and $\phi^*$ this expression is zero if only if $\theta_{\nu_t}$-a.e.\
\[g_t^{\pm}=(\phi')^{-1}(\mp \tfrac{1}{2}\log u_t).\]
Recalling that $\theta_{\nu}=c_{\nu} \sqrt{u} \gamma$, $\kappa_{\nu}^+=c_{\nu} \gamma$ and $\kappa^-_{\nu}=c_{\nu} u \gamma$ we find that indeed for a.e.\ $t$, 
\[ \lambda^{\pm}_t = \kappa_{\nu_t}^{\pm}. \]

\medskip

Vice versa, assume that $\nu_t$ is a strong solution with $\calF_{MF}(\nu_0)<\infty$. Recall that $\nu_t\ll \gamma$ for all $t\in [0,T]$ by Lemma \ref{lm:mf_strongsol}, and hence $\kappa_{\nu_t}^{\pm}\ll \gamma$ as well. Therefore we can again write $u_t:=\dd \nu_t/\dd \gamma$, $\kappa_{\nu}^+=c_{\nu} \gamma$, $\kappa^-_{\nu}=c_{\nu} u \gamma$ and $\theta_{\nu}=c_{\nu} \sqrt{u} \gamma $. Moreover, $u: [0,T]\to L^1(\calT,\gamma)$ is absolutely continuous and a.e.\ differentiable, and thus for every regularized entropy function:
\begin{equation*}
    \int_{\calT} \phi_m(u_T) \,\dd \gamma-\int_{\calT} \phi_m(u_0) \,\dd \gamma=\int_0^T \int_{\calT} c_{\nu_t} \phi_m'(u_t)(1-u_t)\,\dd \gamma \, \dd t.
\end{equation*}
Note that the latter expression is non-positive since $\phi'(z)(z-1)$ is non-negative, due to the convexity of $\phi_m$ and $\phi_m(1)=0$. Moreover, recall that the regularized entropies converge for every $\nu$, are non-negative, and $\Ent(\nu_0|\gamma)<0$ by assumption. Therefore
\[ \limsup_{m\to \infty} \int_0^T \int_{\calT} c_{\nu_t} \phi_m'(u_t)(u_t-1)\,\dd \gamma \, \dd t \leq \Ent(\nu_0|\gamma) < \infty.\] 

It is clear that to obtain $\calI_{MF}=0$ it is sufficient to prove that for any $\nu$ with $\nu \ll \gamma$, 
\begin{equation*}
   \frac{1}{2} \lim_{m\to \infty} \int_{\calT} c_{\nu} \phi_m'(u)(u-1)\,\dd \gamma=\calR(\nu,\kappa_{\nu}^+,\kappa_{\nu}^-)+\calD(\nu).
\end{equation*}
By non-negativity of the integrand both 
\begin{equation*}
\lim_{m\to \infty}      \int_{u=0} c_{\nu} \phi_m'(u)(u-1)\,\dd \gamma < \infty.
\end{equation*}
and
\begin{equation*}
 \lim_{m\to \infty}    \int_{u>0} c_{\nu} \phi_m'(u)(u-1)\,\dd \gamma < \infty.
\end{equation*}
Since $\phi'(0)=-m$ this implies that in fact for all $m$
\begin{equation*}
   \int_{u=0} c_{\nu} \phi_m'(u)(u-1)\,\dd \gamma=m \int_{u=0} c_{\nu} \,\dd \gamma,
\end{equation*}
but since the former is finite after taking the limit $m\to \infty$, we deduce that 
\[\int_{u=0} c_{\nu} \,\dd \gamma=0,\]
and hence $\gamma(\{u=0,c_{\nu}>0)=0$. Moreover, by monotone convergence we have 
\begin{equation*}
\int_{u>0} c_{\nu} \log(u)(u-1)\,\dd \gamma=\lim_{m\to \infty}     \int_{u>0} c_{\nu} \phi_m'(u)(u-1)\,\dd \gamma.
\end{equation*}
Note by straightforward algebraic manipulation that 
\begin{equation*}
    \tfrac{1}{2}\log(z)(z-1)=\phi(\sqrt{z})\sqrt{z}+\phi(1/\sqrt{z})\sqrt{z}+(\sqrt{z}-1)^2\, \qquad \fA z>0.
\end{equation*}
Therefore
\begin{equation*}
\begin{aligned}
\frac{1}{2}\int_{u>0} c_{\nu} \log(u)(u-1) \dd \gamma 
&=\int_{u>0,\,c_{\nu}>0}c_{\nu} \left(\phi\left(\sqrt{u}\right)\sqrt{u}+\phi\left(1/\sqrt{u}\right)\sqrt{u(x)}+(\sqrt{u}-1)^2 \right) \dd \gamma\\
&=\int_{u>0,\,c_{\nu}>0}\left(\phi\left(\frac{\dd \kappa_{\nu}^+}{\dd \theta_{\nu}}\right)\frac{\dd \theta_{\nu}}{\dd \gamma}+\phi\left(\frac{\dd \kappa_{\nu}^-}{\dd \theta_{\nu}}\right)\frac{\dd \theta_{\nu}}{\dd \gamma}+c_{\nu}(\sqrt{u}-1)^2 \right) \dd \gamma.\\
\end{aligned}
\end{equation*}
Since all terms are non-negative we can separate terms and reduce the expression to
\begin{equation*}
    \tfrac{1}{2}\int_{u>0} c_{\nu} \log(u)(u-1) \dd \gamma = \calR_{MF}(\nu,\kappa_{\nu}^+,\kappa_{\nu}^-) +\calD_{MF}(\nu).
\end{equation*}
Here the equality follows from the fact that $\gamma(\{u=0,c_{\nu}>0)=0$ and hence
\begin{equation*}
   \int_{u>0,\,c_{\nu}>0} c_{\nu} (\sqrt{u}-1)^2 \dd \gamma= \int_{\calT} c_{\nu} (\sqrt{u}-1)^2 \dd \gamma=\calD_{MF}(\nu),
\end{equation*}
i.e.\ $\calD_{MF}^ -(\nu)=\calD_{MF}(\nu)$, and 
\[ \int_{u>0,\,c_{\nu}>0} \phi\left(\frac{\dd \kappa_{\nu}^{\pm}}{\dd \theta_{\nu}}\right)\frac{\dd \theta_{\nu}}{\dd \gamma} \dd \gamma = \Ent(\kappa^{\pm}_{\nu}|\theta_{\nu}).\]
\end{proof}

\section{Forward Kolmogorov equation}\label{s:fke}

In the introduction, we discussed how the BPDL model describes a measure-valued process $\nu^n_t$ in $\Gamma$ involving particles being created and annihilated, with the corresponding Forward Kolmogorov equation
\begin{equation}\label{eq:FKEn}\tag{$\sf FKE_n$}
     \partial_t \sfP_t = Q_n^* \sfP_t,
\end{equation}
where $\sfP_t \in \calP(\Gamma)$ for all $t\in [0,T]$ and $Q_n^*$ is the dual of the infinitesimal generator $Q_n$ with
\begin{equation}\label{eq:defiq}
    \qquad  (Q_n F)(\nu) = n \int_{\calT} \big(F(\nu+\tfrac{1}{n}\delta_x)-F(\nu)\big)\, \kappa^+_{\nu}(\dd x)+n \int_{\calT} \big(F(\nu-\tfrac{1}{n}\delta_x)-F(\nu)\big)\, \kappa^-_{\nu}(\dd x), 
\end{equation}
for all $F\in C_c(\Gamma)$. Throughout this section the parameter $n>0$ will be fixed.

In the case of $\calT=\mathbb{R}^d$ it is shown in \cite{Fournier2004} that a measure-valued process with generator $Q_n$ exists, and is in fact a jump process in $\Gamma$ corresponding to the jump kernel $\bar \kappa_n$ shown below. However, for our general setting with $\calT$ a compact Polish space, we will take \eqref{eq:FKEn} simply as a starting point, and do not consider the existence or convergence of the measure-valued process $\nu_t^n$ itself---even though we will sometimes borrow the language of jump processes for illustration purposes.

In this section, we will state the general version of Theorem \ref{thm:ikf}, by showing that a detailed balance condition holds, establishing a generalized gradient structure for the Forward-Kolmogorov equation, and characterizing the solutions as minimizers of corresponding EDP-functionals. Similar to Section \ref{s:mf} we first give an overview of the ingredients to state the main results and then leave the proofs for the existence of solutions and the variational characterization to Sections \ref{ss:fke_sol} and \ref{ss:fke_chain}. 

Note that due to the fact that 
\[\sup_{\nu\in \Gamma} \kappa^{\pm}_{\nu}(\calT)=+\infty,\] 
the operator $Q_n$ is not bounded on $\calB_b(\Gamma)$. If it were, suitable solutions and possible variational formulation would fall into the framework of \cite{PRST2020}, where triples $(V,\pi,\kappa)$ are considered, with $V$ a Polish space,  $\pi$ a finite measure, and $\kappa(x,\dd y)$ a jump kernel satisfying a detailed balance condition with respect to $\pi$ and the boundedness condition
\begin{equation*}
    \sup_{x} \int_{V} k(x,\dd y) < \infty.
\end{equation*}
They construct solutions to the forward Kolmogorov equation that are absolutely continuous to $\pi$ and characterize them as minimizers of a suitable EDP functional involving the net flux. In this section, we generalize part of this framework to unbounded kernels and so-called one-way or uni-directional fluxes and tailor it to our setting of interacting particle systems. 

\medskip

Namely, let the rescaled empirical measure mapping $L_n:\coprod_{N\geq 1} \calT^{N} \to \Gamma$ be given as 
\begin{equation}\label{eq:defiLn}
    L_n(x_1,\dots,x_N):=\frac{1}{n} \sum_{i=1}^N \delta_{x_i}.
\end{equation}
and let $\Gamma_n\subset \Gamma$ be the space of finite \emph{positive} discrete measures with common unit weight $\tfrac{1}{n}$, i.e.
  \begin{equation}\label{eq:gamman}
 \Gamma_n := L_n\left(\coprod_{N\geq 1}\calT^N\right).
\end{equation}
Note that the operators $Q_n, Q_n^*$ can be represented as
\begin{equation*}
\begin{aligned}
(Q_n F)(\nu)&= \int_{\Gamma_n} \left(F(\eta)-F(\nu)\right) \bar \kappa(\nu,\dd \eta), \\
(Q_n^* \sfP)(\dd \nu) &= \int_{\eta \in \Gamma_n}  \sfP(\dd \eta) \bar \kappa_n(\eta,\dd \nu)  -  \sfP(\dd \nu)\int_{\eta \in \Gamma_n}  \bar \kappa_n(\nu,\dd \eta),
\end{aligned}
\end{equation*}
where $\bar \kappa_n(\nu,\cdot) \in \calM^+(\Gamma_n)$ for all $\nu\in \Gamma_n$ is a jump kernel over $\Gamma_n$ given by 
\begin{equation}\label{eq:defibk}
      \bar\kappa_n(\nu,\dd \eta) := n \int_{\calT}\delta_{\nu + \tfrac{1}{n} \delta_x}(\dd\eta)\, \kappa^+_{\nu}(\dd x) + n \int_{\calT} \delta_{\nu-\tfrac{1}{n}\delta_x}(\dd\eta)\,\kappa^-_{\nu}(\dd x).
\end{equation}
Moreover, we consider Poisson measures $\Pi_n\in \calP(\Gamma_n)$ induced by the reference measure $\gamma$. Namely, with the measure $\pi_n \in \calP(\coprod_{N\geq 1}\calT^N)$ given by
\begin{equation}\label{eq:defipn}
    \pi_n:=\frac{1}{e^{n \gamma(\calT)}-1}\sum_{N=1}^{\infty} \frac{n^N}{N!}\gamma^{\otimes N},
\end{equation}
we define
\begin{equation}\label{eq:defiPn}
    \Pi_n:=(L_n)_{\#} \pi.
\end{equation}
We will show in Lemma \ref{lm:revers} that the measures $\Pi_n$ are invariant measures of \eqref{eq:FKEn} and that $\bar \kappa_{n}$ satisfies the detailed balance condition with respect to $\Pi_n$, i.e.\ we have the symmetry 
\begin{equation}\label{eq:fke_rev1}
     \Pi_n(\dd \nu) \bar \kappa_n(\nu,\dd \eta)=\Pi_n(\dd \nu) \bar \kappa_n(\nu,\dd \eta).
\end{equation}
It is straightforward to check that even though $\bar \kappa_n$ is unbounded, we still have the weighted integrability condition
\begin{equation*}
    \sup_{\nu\in \Gamma_n} \left\{ (1+\nu(\calT)^{-2}) \int_{\Gamma_n} \bar \kappa_{\nu}(\nu,\dd \eta)\right\} < +\infty.
\end{equation*}
Therefore we can still bootstrap from gradient-flow solutions in the sense of \cite{PRST2020} for regularized triples $(\Gamma_n,\Pi_n,\bar \kappa^{\eps}_n)$, after passing from a net flux to a one-way flux formulation, see Appendix \ref{s:ldpmot}, to obtain unique gradient-flow solutions as defined in Section \ref{ss:fke_sol}. 
\medskip

To discuss the continuity equation and the dissipation potentials properly, we need to introduce some additional notation. We define the following creation and annihilation operators: 
\begin{equation}\label{eq:defit}
    \begin{aligned}
    \sfT^{n,+}&:\Gamma_n\times \calT\to\Gamma_{n}\times \calT,\qquad \sfT^{n,+}(\nu,x) = (\nu + \tfrac{1}{n}\delta_x,x) =: (\sfT_x^{n,+}\nu,x),\\
    \sfT^{n,-}&:\Gamma_{n}\times \calT\to\Gamma_{n}\times \calT,\qquad \sfT^{n,-}(\nu,x) = (\nu - \tfrac{1}{n}\delta_x,x) =: (\sfT_x^{n,-}\nu,x),
\end{aligned}
\end{equation}
with the convention that $\sfT^{n,-}(\nu,x)=\nu$ if $x\notin \mathrm{supp}(\nu)$. Note that  $\sfT^{n,-} \circ \sfT^{n,+}=\mathsf{Id}$ always holds, and $\sfT^{n,+} \circ \sfT^{n,-} (\nu,x)=(\nu,x)$ whenever $x\in \mathrm{supp}(\nu)$. 

We further define the discrete $\Gamma_n$-gradients $\dnabla^{n,\pm} : C_c(\Gamma_n)\to C_c( \Gamma_n\times \calT)$:
\begin{equation}\label{eq:defidder}
 (\dnabla^{n,\pm} F)(\nu,x) := n(F(\sfT_x^{n,\pm}\nu) - F(\nu)),
\end{equation}
and the corresponding $\Gamma_n$-divergence $\ddiv^{n,\pm} : \calM_{loc}^+(\Gamma_n\times \calT)\to \calM_{loc}(\Gamma_n)$, dual to $\dnabla^{n,\pm}$, given by
\begin{equation}\label{eq:defiddiv}
 (\ddiv^{n,\pm} \sfJ) = n\left(\sfp^{\Gamma_n}_\#\sfJ - (\sfp^{\Gamma_n}\circ\sfT^{n,\pm})_\# \sfJ\right),
\end{equation}
where $\sfp^{\Gamma_n}:{\Gamma_n}\times \calT\to \Gamma_n$ denotes the projection to the first variable. 

We consider the familes of curves satisfying 
\begin{equation}\label{eq:fke_cke}\tag{$\mathsf{CE}_n$}
    \partial_t\sfP_t + (\ddiv^{n,+} \sfJ_t^+)+(\ddiv^{n,-} \sfJ_t^-)=0
\end{equation}
in the following appropriate distributional sense. 

\begin{defi}[Continuity equation]\label{defi:cen}\,

A triple  $(\sfP,\sfJ^+,\sfJ^-)$ satisfies the continuity equation $\mathsf{CE}_n$, if 
\begin{enumerate}
\itemsep0.1em 
\item\label{fke:cond0} the curve $[0,T]\ni t\mapsto \sfP_t\in \calP(\Gamma_n)$ is narrowly continuous,
\item\label{fke:cond1} the Borel family $(\sfJ^{\pm}_t)_{t\in [0,T]}\in \calM^+_{loc}(\Gamma_n\times \calT)$ satisfies 
\[ \mathrm{supp}(\sfJ^-_t) \subseteq \left\{ (\nu,x)\, :\, \nu(\calT)\geq \tfrac{2}{n}, \, x\in \mathrm{supp}(\nu) \right\}, \]
\item\label{fke:cond3} $\int_0^T \int_{\Gamma_n\times \calT} (1+\nu(\calT)^2)^{-1}\,\dd \sfJ^{\pm}_{t} \, \dd t<\infty$,
    \item\label{fke:cond4} for every $s,t\in [0,T]$ and all $F\in C_c(\Gamma_n)$
     \begin{equation}\label{eq:fkecont1}
     \int_{\Gamma_n} F(\nu) \,\dd \sfP_t - \int_{\Gamma_n} F(\nu) \,\dd \sfP_s = \int_s^t \int_{\Gamma_n\times \calT} \left((\dder^{n,+} F) \,\dd \sfJ_r^++(\dder^{n,-} F) \, \dd \sfJ_r^{-} \right) \, \dd r.
    \end{equation}
\end{enumerate}
\end{defi}

Throughout we will call arbitrary measures $\sfJ^{\pm} \in \calM^+_{loc}(\Gamma_n\times \calT)$ \emph{admissible} if 
\[ \mathrm{supp}(\sfJ^-) \subseteq \left\{ (\nu,x)\, :\, \nu(\calT)\geq \tfrac{2}{n}, \, x\in \mathrm{supp}(\nu) \right\}\]
and 
\[\int_{\Gamma_n\times \calT} (1+\nu(\calT)^2)^{-1}\,\dd \sfJ^{\pm}<\infty.\]
Moreover, since $\Gamma_n$ is a closed subspace of the Polish space $\Gamma$, the extension of $\sfP$ to $\calP(\Gamma)$ and the extension of $\sfJ^{\pm}$ to $\calM^+_{loc}(\Gamma\times \calT)$ are well-defined. For simplicity we will simply refer to them as $\sfP$, $\sfJ^{\pm}$ as well, and drop the $n$-dependence in most arguments. 

It is also clear that for any admissible $\sfJ^{\pm}$
\begin{equation*}
(\dder^{n,\pm} F)(\nu,x):=n\left(F(\nu\pm \tfrac{1}{n}\delta_x)-F(\nu)\right), \qquad (\nu,x)\in \mathrm{supp}(\sfJ^\pm)
\end{equation*}
and in particular \eqref{eq:fkecont1} is equivalent to 
     \begin{equation*}
        \int_{\Gamma} F(\nu) \,\dd \sfP_t - \int_{\Gamma} F(\nu) \,\dd \sfP_s = \int_s^t \int_{\Gamma\times \calT} \Big( n\big(F(\nu+\tfrac{1}{n}\delta_x)-F(\nu)\big)\,\dd \sfJ_r^++n\big(F(\nu-\tfrac{1}{n}\delta_x)-F(\nu)\big)\, \dd \sfJ_r^- \Big) \, \dd r.
    \end{equation*}
for all $F\in C_c(\Gamma)$. Note that this can again be extended to all $F\in \calB_c(\Gamma)$ via a monotone class argument. 

\begin{remark}
Condition (\ref{fke:cond1}) represents the restriction that particles can only be deleted if there are at least two particles in the system, consistent with the fact that $\sfP\in \calP(\Gamma_n)$ and hence the underlying process never attains $\nu=0$.

Moreover, condition (\ref{fke:cond3}) reflects the unboundedness of the observed fluxes $\sfJ^{\pm}$, which stems from the unboundedness of the birth/death kernels $\kappa^{\pm}_{\nu}$ in $\nu$.
\end{remark}

\begin{remark}
Whenever $\sfJ^{\pm}$ are of the form 
\[ \sfJ_t^{\pm}(\dd \nu,\dd x)=\sfP_t(\dd \nu) \lambda^{\pm}[t,\nu](\dd x)\]
with $\lambda^{\pm}[t,\nu]\in \calM^+(\calT)$ for all $\nu \in \Gamma$ and $t\in [0,T]$, the continuity equation \eqref{eq:fkecont1}  describes the forward Kolmogorov equation corresponding to an interacting birth/death process with the birth/death kernels $\lambda^{\pm}[t,\nu]$ depending on both time and the empirical measure of the particles $\nu$. The time-dependent jump kernel is then given by 
\[
\bar \kappa_{n,t}(\dd \nu,\dd \eta)= n \left(\int_{\calT}\delta_{\nu +\tfrac{1}{n}\delta_x}(\dd\eta)  \, \lambda^{+}[t,\nu](\dd x) + \int_{\calT} \delta_{\nu-\tfrac{1}{n}\delta_x}(\dd\eta)\, \lambda^{-}[t,\nu](\dd x)\right).
\]
\end{remark}

\medskip

In order to define the dissipation potentials, let us introduce the measures $\teta_\sfP^{\pm} \in \calM_{loc}^+(\Gamma\times \calT)$
\begin{equation}\label{eq:defief}
	\teta_\sfP^{\pm}(\dd\nu\,\dd x) := \sfP(\dd \nu)\kappa^{\pm}_{\nu}(\dd x).
\end{equation}
Note that for any curve $(\sfP_t)_{t\in [0,T]}$ the measures $\sfJ_t^{\pm}:=\teta_{\sfP_t}^{\pm}$ satisfy the conditions \eqref{fke:cond1} and \eqref{fke:cond3}, where the latter holds because $c(x,x)=0$. 

Moreover, as will be shown in Lemma \ref{lm:revers}, we have the following symmetry
\begin{equation}\label{eq:fke_rev2}
    \vartheta_{\Pi_n}^{\pm}=\sfT^{n,\mp}_{\#}\vartheta^{\mp}_{\Pi_n}.
\end{equation}
from which the detailed balance condition $\eqref{eq:fke_rev1}$ directly follows.  

\begin{defi}\label{defi:fke}
Let $\Theta^{n,\pm}_{\sfP}\in \calM_{loc}(\Gamma\times \calT)$ be the geometric average of $\vartheta^{\pm}_{\sfP}$ and $\sfT^{n,\mp}_{\#}\vartheta^{\mp}_{\sfP}$, i.e.\ 
\begin{equation}\label{eq:Theta}
    \Theta^{n,\pm}_{\sfP}(\dd \nu,\dd x):=\sqrt{\frac{\dd \vartheta^{\pm}_{\sfP}}{\dd \Sigma}\frac{\dd (\sfT^{n,\mp}_{\#}\vartheta^{\mp}_{\sfP})}{\dd\Sigma}}\, \,\dd \Sigma,
\end{equation}
for any dominating measure $\Sigma$. 

The dissipation potential $\calR_n:\calP(\Gamma)\times \calM_{loc}^+(\Gamma\times \calT)^2\to[0,+\infty]$ and dual dissipation potential $\calR^*_n:\calP(\Gamma)\times \calB_c(\Gamma \times \calT)^2$ are given by
\begin{equation*}
\begin{aligned}
  \calR_n(\sfP,\sfJ^+,\sfJ^-)&:=\Ent(\sfJ^{+}|\Theta^{n,+}_{\sfP})+\Ent(\sfJ^{-}|\Theta^{n,-}_{\sfP}),\\
     \calR_n^*(\sfP,\omega^+,\omega^-)&:=\int_{\Gamma\times \calT} (e^{\omega^{+}}-1)\, \dd \Theta^{n,+}_{\sfP}+\int_{\Gamma\times \calT} (e^{\omega^{-}}-1)\, \dd \Theta^{n,-}_{\sfP}
\end{aligned}
\end{equation*}

For the free energy $\calF_n:\calP(\Gamma)\to[0,+\infty]$ and Fisher information $\calD_n:\calP(\Gamma)\to [0,+\infty]$ 
\begin{equation*}
\begin{aligned}
\calF_n(\sfP)&:=\tfrac{1}{2n} \Ent(\sfP|\Pi_n)\\
\calD_n(\sfP)&:=\left\{ 
    \begin{aligned}
    &H^2(\vartheta_{\sfP}^{+},\sfT_{\#}^{n,-}\vartheta_{\sfP}^{-})+H^2(\vartheta_{\sfP}^{-},\sfT_{\#}^{n,+}\vartheta_{\sfP}^{+})\qquad &&  \mbox{if $\sfP\ll \Pi_n$,}\\
    &+\infty &&  \mbox{otherwise}.\\
    \end{aligned}\right.
\end{aligned}
\end{equation*}

For the \emph{EDP-functional} $\calI_{n}:\mathsf{CE}_n\to [0,+\infty]$ for all curves with $\calF_{n}(\sfP_0)<\infty$
\begin{equation*}
    \calI_{n}(\sfP,\sfJ^+,\sfJ^-):=\int_0^T \calR_{n}(\sfP_t,\sfJ_t^+,\sfJ^-_t) \, \dd t + \calF_n(\nu_T)-\calF_n(\nu_0)+\int_0^T \calD_{n}(\sfP_t) \, \dd t.
\end{equation*}
\end{defi}

\begin{remark}
The definition of $\Theta_{\sfP}^{n,\pm}$ is independent of the dominating measure $\Sigma$. Moreover, formally 
\[\Theta_{\sfP}^{n,+}(\nu,x)=\sqrt{ (\sfP(\nu)\kappa^+[\nu])(\sfP(\nu+\tfrac{1}{n}\delta_x)\kappa^-[\nu+\tfrac{1}{n}\delta_x}]), \]
i.e.\ it represents the geometric mean of the expected fluxes going forwards and backwards along the transition $\nu \leftrightarrow\nu+\tfrac{1}{n}\delta_x$.

In addition, due to the symmetry \eqref{eq:fke_rev2} the measures $\Theta_{\sfP}^{n,\pm}$ simplify whenever $\sfP\ll \Pi_n$, i.e.\ if $\dd \sfP = U \dd \Pi_n$ we have
\[
    \Theta^{n,\pm}_{\sfP}(\dd \nu,\dd x)=\sqrt{U(\nu)U(\nu\pm\tfrac{1}{n}\delta_x)}\, \vartheta_{\Pi_n}^{\pm}(\dd \nu,\dd x).
\]
\end{remark}

\begin{remark}\label{rem:fke_fischereq}
Note that $\calD_n$ is a jointly convex function in $(\vartheta^{\pm}_{\sfP},\sfT_{\#}^{n,\mp} \vartheta_{\sfP}^{\mp})$, and lower semicontinuous if $\calF_{n}$ is bounded.
Moreover, it is straightforward to check that whenever $\sfP\ll \Pi_n$ with $\dd \sfP = U \Pi_n$ it holds
\begin{align*}
    \calD_n(\sfP)&=\frac{1}{2}\int_{\Gamma\times \calT}  \left(\sqrt{U(\nu+\tfrac{1}{n}\delta_x)}-\sqrt{U(\nu)}\right)^2 \dd \vartheta^{+}_{\Pi_n}+\frac{1}{2}\int_{\Gamma\times \calT}  \left(\sqrt{U(\nu-\tfrac{1}{n}\delta_x)}-\sqrt{U(\nu)}\right)^2 \dd \vartheta^{-}_{\Pi_n}\\
&=\int_{\Gamma\times \calT}  \left(\sqrt{U(\nu\pm\tfrac{1}{n}\delta_x)}-\sqrt{U(\nu)}\right)^2 \dd \vartheta^{\pm}_{\Pi_n}.
\end{align*}
\end{remark}

Finally, for technical purposes, we also introduce a version for net fluxes.
\begin{defi}
The \emph{upward} net flux $\sfJn$ is defined as 
\begin{equation*}
    \jnet:=\sfJ^+-\sfT^{n,-}_{\#} \sfJ^{-}
\end{equation*}
\end{defi}
Note that $\jnet(\nu,x)$ can be interpreted as the net flux along the jump $\nu\leftrightarrow\nu+\tfrac{1}{n}\delta_x$.  

The continuity equation for the net flux reduces to 
 \begin{equation*}
      \int_{\Gamma} F(\nu) \,\dd \sfP_t - \int_{\Gamma} F(\nu) \,\dd \sfP_s = \int_s^t \int_{\Gamma\times \calT}  n\big(F(\nu+\tfrac{1}{n}\delta_x)-F(\nu)\big)\,\dd \jnet_r\, \dd r
\end{equation*}

\medskip

We are now in a position to give the general version of Theorem \ref{thm:ikf}.
\begin{thm}\label{thm:fke_main}
For any $(\sfP,\sfJ^+,\sfJ^-)\in \mathsf{CE}_{n}$ with $\calF_n(\sfP_0)<\infty$ we have $\calI_{n}(\sfP,\sfJ^+,\sfJ^-)\geq 0$,
\begin{equation}\label{eq:fken_equiv}
    \calI_{n}(\sfP,\sfJ^+,\sfJ^-)=0 \implies \left\{ \begin{aligned}
\quad &\mbox{$\sfP_t$ is a weak solution to \eqref{eq:FKEn}} \quad \\
\quad &\sfJ^{\pm}_t=\sfP_t \kappa_{\nu}^{\pm} \quad \mbox{for a.e.\ $t\in [0,T]$}, \quad
\end{aligned} \right.
\end{equation}
and there exist a unique gradient-flow solution, i.e.\ a curve $(\sfP)$ such that $\calI_{n}(\sfP,\sfP_t \kappa_{\nu}^{+},\sfP_t \kappa_{\nu}^{-})=0$.

Moreover, whenever $\calF_n(\sfP_0)<\infty$ and $\calI_{n}(\sfP,\sfJ^+,\sfJ^-) < \infty$, the chain rule for $\calF_{n}$ and the net flux holds holds: $\calF_{n}(\sfP_t)$ is absolutely continuous and 
\begin{equation*}
    \frac{\dd \,}{\dd t} \calF_n(\sfP_t)=\frac{n}{2}\int_{\Gamma\times \calT} (\log U(\nu+\tfrac{1}{n}\delta_x)-\log U(\nu))\,\dd \jnet_t \, \dd t, \qquad \mbox{for a.e.\ $t\in [0,T]$.}
\end{equation*}
\end{thm}

\medskip

The proof of Theorem~\ref{thm:fke_main} is postponed to Section \ref{ss:fke_chain}, and follows from the existence of a gradient-flow solution via EDP-convergence of a sequence of regularized problems established in Section \ref{ss:fke_sol}, and its uniqueness via a convexity argument. 

\begin{remark}
Similar to the mean-field case, the non-negativity of $\calI_{n}$ and the identification of solutions to \eqref{eq:mf} as null-minimizers of $\calI_n$ is related to the formal equivalence
\begin{equation*}
    \calI_{n}(\sfP,\sfJ^+,\sfJ^-)=\int_0^T \calL_n(\sfP_t,\sfJ^+_t,\sfJ^-_t) \,\dd t,
\end{equation*}
where $\calL_n$ is the so-called \emph{Lagrangian} given by
\[ \calL_n(\sfP,\sfJ^+,\sfJ^-):=\Ent(\sfJ^+|\sfP \kappa_{\nu}^+)+\Ent(\sfJ^-|\sfP \kappa_{\nu}^-).\]
We discuss the implication of this relation in Appendix \ref{s:ldpmot}. 
\end{remark}

\begin{remark}[Net flux]
To show the existence of gradient-flow solutions in the sense of null-minimizers of $\calI_n$ we will have to jump from gradient-flow solutions in the sense of \cite{PRST2020}, see Theorem \ref{thm:fke_solex}. The expressions for net-fluxes are in fact contractions of those for one-way or uni-directional fluxes, as discussed in Section \ref{s:ldpmot}, which we use to show that the two notions of gradient-flow solutions are equivalent. 
\end{remark}

\subsection{A priori estimates}\label{ss:fke_est}

Below we will state the estimates and identities necessary to prove the chain rule and establish the existence of solutions. 

Recall that $\vartheta_{\sfP}^{\pm}$ satisfies the same restrictions (Conditions \eqref{fke:cond1} and \eqref{fke:cond3}) as the fluxes $\sfJ^{\pm}$. This is easily verified, but since we will use it repeatedly let us state it here precisely. 

\begin{lm}\label{lm:Tidenty}
For any $\sfP\in \calP(\Gamma_n)$
\[ \mathrm{supp}(\vartheta_{\sfP}^-) \subseteq \left\{ (\nu,x)\, :\, \nu(\calT)\geq \tfrac{2}{n}, \, x\in \mathrm{supp}(\nu) \right\}. \]
In particular, for any $\omega\in C_c(\Gamma \times \calT)$ 
\begin{align*}
\int_{\Gamma\times \calT} \omega(\nu,x) \,\dd (\sfT^{n,\pm}_{\#} \vartheta_{\sfP}^\pm)&=\int_{\Gamma\times \calT} \omega(\nu\pm\tfrac{1}{n}\delta_x,x)\, \dd  \vartheta_{\sfP}^\pm,
\end{align*}
and 
\[ \sfT^{n,\mp}_{\#} \circ\sfT^{n,\pm}_{\#} \vartheta_{\sfP}^{\pm}=\vartheta_{\sfP}^{\pm}.\]
Finally, 
\begin{equation*}
    \sfT^{n,\pm}_{\#} \Theta^{n,\pm}_{\sfP} = \Theta^{n,\mp}_{\sfP}.
\end{equation*}
\end{lm}

The above identities allow us to prove the symmetry condition that implies the detailed balance condition \eqref{eq:fke_rev1}.
\begin{lm}[Detailed balance]\label{lm:revers}
\begin{equation*}
    \vartheta_{\Pi_n}^{\pm}=\sfT^{n,\mp}_{\#}\vartheta^{\mp}_{\Pi_n}.
\end{equation*}

\end{lm}

\begin{proof}
Fix an arbitrary $\omega\in C_c(\Gamma\times \calT)$, and for any ordered collection of $N$ variables in $\calT$ set $\mathbf{x}^{N}:=(x_1,\dots,x_N) \in X^N$. We then have the following. 
\begin{align*}
    \int_{\Gamma\times \calT}\omega(\nu,x) \,\dd \vartheta^+_{\Pi_n}&=\frac{1}{e^{n \gamma (\calT)}-1}\sum_{N=1}^{\infty} \frac{n^N}{N!} \int_{\calT^N}  \left(\int_{\calT} \omega\left(L_n(\mathbf{x}^{N}),x\right) \kappa^+\left[L_n(\mathbf{x^N})\right](\dd x) \right)  \gamma^{\otimes N}(\dd \mathbf{x}^N),\\
  \int_{\Gamma\times \calT}\omega(\nu,x) \, \dd(\sfT^{n,-}_{\#} \vartheta^-_{\Pi_n})&=\frac{1}{e^{n\gamma(\calT)}-1}\sum_{N=1}^{\infty} \frac{n^N}{N!} \int_{\calT^N}  \left(\int_{\calT} \omega\left(L_n(\mathbf{x}^{N})-\tfrac{1}{n}\delta_x,x\right) \kappa^-\left[L_n(\mathbf{x}^N)\right](\dd x) \right)  \gamma^{\otimes N}(\dd \mathbf{x}^N).
\end{align*}
Since $\kappa^-[\tfrac{1}{n}\delta_{y}]=0$ for any $y\in \calT$, the sum in the right-hand side of the last expression starts from $N=2$, thus reducing the expression to
\[\frac{1}{e^{n \gamma(\calT)}-1}\sum_{N=1}^{\infty} \frac{n^{N+1}}{(N+1)!} \int_{\calT^{N+1}}  \left(\int_{\calT} \omega\left(L(\mathbf{x}^{N+1})-\tfrac{1}{n}\delta_x,x\right) \kappa^-\left[L_n(\mathbf{x}^{N+1})\right](\dd x) \right)  \gamma^{\otimes (N+1)}(\dd \mathbf{x}^{N+1}),\]
It is clear that, for our desired equality, it is enough to show that for every $N$,
\begin{equation*}
\begin{aligned}
 n \int_{\calT^{N+1}} & \left(\int_{\calT^2} \omega\left(L_n(\mathbf{x}^{N+1})-\tfrac{1}{n}\delta_x,x\right)c(x,y)\, L_n(\mathbf{x}^{N+1})^{\otimes 2}(\dd x,\dd y)\right) \, \gamma^{\otimes (N+1)}(\dd \mathbf{x}^{N+1})\\
  &=(N+1)\int_{\calT^{N}} \left(\int_{X^2} \omega\left(L_n(\mathbf{x}^{N}),x\right)c(x,y)\gamma(\dd x)  L_n(\mathbf{x}^{N})(\dd y)\right)\, \gamma^{\otimes N}(\dd \mathbf{x}^N).
\end{aligned}
\end{equation*}
To do so, note that since $c(x,x)=0$,
\begin{align*}
  &n \int_{X^2} \omega\left(L_n(\mathbf{x}^{N+1})-\tfrac{1}{n}\delta_x,x\right)c(x,y)\, L_n(\mathbf{x}^{N+1})^{\otimes 2}(\dd x,\dd y)\\
  &\hspace{8em}=\frac{1}{n}\sum_{i=1}^{N+1} \sum_{j\neq i} \omega\left(L_n(\mathbf{x}^{N+1})-\tfrac{1}{n}\delta_{x_i},x_i\right)c(x_i,x_j)\\
  &\hspace{8em}=\sum_{i=1}^{N+1} \int_{\calT}  \omega\left(L_n(\mathbf{x}^{N+1})-\tfrac{1}{n}\delta_{x_i},x_i\right)c(x_i,y) \left(L_n(\mathbf{x}^{N+1})-\tfrac{1}{n}\delta_{x_i}\right)(\dd y).
\end{align*}
Hence, by symmetry of $\gamma^{\otimes (N+1)}$, we obtain
\begin{align*}
\int_{\calT^{N+1}} &\left(\sum_{i=1}^{N+1} \int_{\calT}  \omega\left(L_n(\mathbf{x}^{N+1})-\tfrac{1}{n}\delta_{x_i},x_i\right)c(x_i,y) \left(L_n(\mathbf{x}^{N+1})-\tfrac{1}{n}\delta_{x_i}\right)(\dd y) \right) \gamma^{\otimes (N+1)}(\dd \mathbf{x}^{N+1})\\
&=(N+1) \int_{\calT^N} \left( \int_{\calT^2}  \omega\left(L_n(\mathbf{x}^{N}),x\right)c(x,y) \gamma(\dd x)L_n(\mathbf{x}^{N})(\dd y) \right) \gamma^{\otimes N}(\dd \mathbf{x}^{N}),
\end{align*}
as desired. 
\end{proof}

Recall from Lemma \ref{lm:mf_est} that that
    \begin{equation*}
          \kappa^{\pm}_{\nu}(\calT)\leq M (1+\nu(\calT)^2),
    \end{equation*}
where $M:=(1+\gamma(\calT))\|c\|_{\infty}$. Now let 
\[ M_n :=\max\bigl\{1+2/n^2,2\bigr\} M, \]
and the jointly convex and lower semicontinuous function $\Upsilon:\R_{\geq 0}^3\to [0,+\infty]$ given by
\begin{equation*}
    \Upsilon(w,u,v):= \begin{cases}
    \sqrt{u v} & \quad \mbox{if $w=0$,}\\
    \phi\left(\frac{w}{\sqrt{u v}}\right)\sqrt{u v} & \quad \mbox{if $u,v>0$,}\\
    +\infty & \quad \mbox{if $w>0$, and either $u=0$ or $v=0$.}
    \end{cases}
\end{equation*}
We then have the following result. 

\begin{lm}\label{lm:fke_est} The following statements hold:
\begin{enumerate}[label=(\roman*)]
    \item For all $\sfP$  
\begin{equation*}
        \int_{\Gamma\times \calT} (1+\nu(\calT))^{-2}\,\Theta_{\sfP}^{n,\pm}(\dd\nu\,\dd y)\leq \int_{\Gamma\times \calT}  (1+\nu(\calT)^2)^{-1}\,\Theta_{\sfP}^{n,\pm}(\dd\nu\,\dd y) \le M_n.
\end{equation*}

\item For any $\sfP$, admissible $\sfJ^{\pm}$, and net flux $\jnet=\sfJ^+-\sfT^{n,-}_{\#}\sfJ^-$,  $\omega\in \calB(\Gamma\times \calT)$, we have 
\begin{equation*}
 \int_{\Gamma\times \calT} |\omega|\,\dd |\jnet| \, \leq \calR_n(\sfP,\sfJ^+,\sfJ^-)+\int_{\Gamma\times \calT} \Psi^*(\omega) \, \dd \Theta_{\sfP}^{n,+}.
\end{equation*}

Moreover,
\begin{subequations}\label{eq:fke_actionb}
\begin{align}
    \phi\left(1 \vee \frac{1}{M_n} \int_{\Gamma\times \calT} (1+\nu(X)^2)^{-1}\,\sfJ^{\pm}(\dd \nu,\dd x) \right)M_n\, \leq \calR_n(\sfP,\sfJ^+,\sfJ^-), \label{eq:fke_actionb1} \\
    \Psi\left(\frac{1}{M_n} \int_{\Gamma\times \calT} (1+\nu(X))^{-1}\,|\jnet|(\dd \nu,\dd x) \right)M_n\, \leq \calR_n(\sfP,\sfJ^+,\sfJ^-).  \label{eq:fke_actionb2}
\end{align}
\end{subequations}
\item For all admissible $\sfP,\sfJ^{\pm}$,
\begin{equation}\label{eq:fke_entequiv2}
      \Ent(\sfJ^{\pm}|\Theta^{n,\pm})=\int_{\Gamma\times \calT} \Upsilon\left(\frac{\dd \sfJ^\pm}{\dd \Sigma},\frac{\dd \vartheta_{\sfP}^\pm}{\dd \Sigma},\frac{\dd (\sfT^{n,\mp}_{\#}\vartheta_{\sfP}^\mp)}{\dd \Sigma}\right)\dd  \Sigma,
\end{equation}
for any common dominating measure $\Sigma$. Moreover, if $\dd \sfP=U \dd  \Pi_n$,
\begin{equation*}
    \Ent(\sfJ^{\pm}|\Theta^{n,\pm})=\int_{\Gamma\times \calT} \Upsilon\left(\frac{\dd \sfJ^{\pm}}{\dd \vartheta_{\sfP}^{\pm}},U(\nu),U(\nu\pm\tfrac{1}{n}\delta_x)\right)\dd \vartheta^{\pm}_{\sfP}.
\end{equation*}
\end{enumerate}
\end{lm}

\begin{remark}
Since $M_n\leq 3M$ for all $n\geq 1$ the estimates \eqref{eq:fke_actionb} are uniform in $n$, which we will use in the EDP-convergence to establish tightness of sequences $\sfJ^{n,\pm}$ under bound on $\calI_n$. Moreover, the representation \eqref{eq:fke_entequiv2} is used to deduce the lower-semicontinuity of $\calI_n$ for sequences of curves. 
\end{remark}

\begin{proof}
\emph{(i)} For any $x^*\in \calT$, $\nu \in \Gamma$, we have
\begin{align*}
    \max\{\kappa^{\pm}(\calT), \kappa^{\pm}[\sfT^{n,\pm}_{x^*}(\nu)](\calT)\}&\le M \max\left\{1+\nu(\calT)^2,\, 1+(\sfT^{n,+}_{x^*}(\nu))(\calT)^2,\, 1+(\sfT^{n,-}_{x^*}(\nu))(\calT)^2\right\}\\
    &\leq M_n (1+\nu(\calT)^2)
    \end{align*}
due to the inequality \[1+(\tfrac{1}{n}+z)^2\leq 1+\tfrac{2}{n^2}+2z^2, \qquad \fA z\geq 0.\]
In particular, 
\begin{equation*}
\max \left\{ \int_{\Gamma\times \calT} (1+\nu(\calT)^2)^{-1} \dd \teta_{\sfP}^\pm, \int_{\Gamma\times \calT} (1+\nu(\calT)^2)^{-1} \dd \sfT^{n,\mp}_{\#} \teta_{\sfP}^\pm \right\} \leq M_n, 
\end{equation*}
and hence the desired statement follows after applying Jensen's inequality. 

\emph{(ii)} By duality we have for any $\omega \in \calB_c(\Gamma\times \calT)$, 
\begin{align*}
    \int_{\Gamma\times \calT} \omega^+\, \dd \sfJ^++\int_{\Gamma\times \calT} \omega^- \,\dd\sfJ^-\leq \calR_n(\sfP,\sfJ^+,\sfJ^-)+\int_{\Gamma\times \calT} (e^{\omega^+}-1) \,\dd \Theta_{\sfP}^{n,+}+\int_{\Gamma\times \calT} (e^{\omega^-}-1)\, \dd \Theta_{\sfP}^{n,-}.
\end{align*}
Substituting $\omega^+=\omega$, $ \omega^-=-\omega\circ \sfT^{n,-}$ and using the fact that $\sfT^{n,-}_{\#}\Theta^{n,-}_{\sfP}=\Theta^{n,+}_{\sfP}$ we derive
\begin{align*}
    \int_{\Gamma\times \calT} \omega \,\dd \jnet \leq \calR_n(\sfP,\sfJ^+,\sfJ^-)+\int_{\Gamma\times \calT} \Psi^*(\omega) \,\dd \Theta_{\sfP}^{n,+}.
\end{align*}
Since $\Psi^*$ is even we can replace $\omega$ and $\sfJ$ by their absolutes in the inequality, after substituting for $\omega$ appropriately, and we conclude with a monotone convergence argument. The inequalities  \eqref{eq:fke_actionb1} and \eqref{eq:fke_actionb2} now follow similarly as in Lemma \ref{lm:mf_est} via respectively Jensen's inequality and a dual approach. 

\medskip

\emph{(iii)} Let us only consider $\sfJ^+$, $\Theta^{n,+}$ (the case for $\sfJ^-$, $\Theta^{n,-}$ is similar). Suppose $\Ent(\sfJ^{+}|\Theta^{n,+})<\infty$ and recall that 
\begin{equation*}
    \Theta^{n,+}_{\sfP}(\dd \nu,\dd x):=\sqrt{\frac{\dd \vartheta^{+}_{\sfP}}{\dd \Sigma}\frac{\dd (\sfT^{n,-}_{\#}\vartheta^{-}_{\sfP})}{\dd\Sigma}}\, \,\dd \Sigma,
\end{equation*}
where $\Sigma$ is a dominating measure, e.g  $\Sigma=\vartheta^{+}_{\sfP}+\sfT^{n,-}_{\#}\vartheta^{-}_{\sfP}$. 
Then $\sfJ^+\ll \Theta_{\sfP}^{n,+} \ll  \Sigma$, and it follows that $\sfJ^+$-a.e.\ $\dd \vartheta_{\sfP}^{+}/\dd \Sigma$, $\dd (\sfT^{n,-}_{\#}\vartheta_{\sfP}^{-})/\dd \Sigma>0$, from which one can easily verifies \eqref{eq:fke_entequiv2}.

Vice versa, suppose that 
\[\int_{\Gamma\times \calT} \Upsilon\left(\frac{\dd \sfJ^+}{\dd \Sigma},\frac{\dd \vartheta_{\sfP}^+}{\dd \Sigma},\frac{\dd (\sfT^{n,-}_{\#}\vartheta_{\sfP}^-)}{\dd \Sigma}\right)\dd  \Sigma <\infty,\]
for some dominating measure $\Sigma$. Then again $\sfJ^+$-a.e.\ we have that $\dd \vartheta_{\sfP}^{+}/\dd \Sigma$, $\dd (\sfT^{n,-}_{\#}\vartheta_{\sfP}^{-})/\dd \Sigma>0$, and by super-linearity of $\phi$ deduce that in fact $\sfJ^+\ll \tilde \Sigma$ for any dominating measure of $\vartheta_{\sfP}^+$ and $\sfT^{n,-}_{\#}\vartheta_{\sfP}^-$, which together implies $\sfJ^+\ll \Theta_{\sfP}^{n,+}$ and the result follows similarly as above. 
\end{proof}

Finally, we discuss the time-regularity of $\sfP_t$ for admissible curves and state the analog of Lemma \ref{lm:mf_timereg}. Let the weighted total variation metric $d_{TV,w}$ be given as 
\begin{equation}\label{eq_fked}
    \begin{aligned}
         d_{TV,w}(\sfP^1,\sfP^2)&:=\int_{\Gamma} (1+\nu(\calT)^2)^{-1}\, \dd |\sfP^1-\sfP^2|.
    \end{aligned}
    \end{equation}
Note that $d_{TV,w}$ is lower semicontinuous with respect to the narrow topology, and while convergence in $d_{TV,w}$ does not directly imply narrow convergence, it does so on narrowly pre-compact sets. 

\begin{lm}\label{lm:fke_timereg}
For any $(\sfP,\sfJ^+,\sfJ^-)\in \mathsf{CE}_n$ we have $\fA s,t\in [0,T]$:
\begin{equation}\label{eq:fle_timereg1}
    d_{TV,w}(\sfP_s,\sfP_t)\leq 4 n \max\Bigl\{1+\tfrac{2}{n^2},2\Bigr\} \int_s^t \int_{\Gamma\times \calT} (1+\nu(\calT)^2)^{-1} \dd (\sfJ_r^{+}+\sfJ^-_r)\, \dd r.
\end{equation}
Suppose in addition that $\sfP_t\ll \Pi_n, \sfJ^{\pm}_t\ll \teta_{\Pi_n}^{\pm}$ for all $t\in [0,T]$ and set 
\begin{align*}
    \ell:=(1+\nu(\calT)^2)^{-1} \Pi_n,\qquad
    \Sigma^{\pm}:=(1+\nu(\calT)^2)^{-1} \teta^{\pm}_{\Pi_n}.
\end{align*} 
Then there exists an absolutely continuous and a.e. differentiable map $U:[a,b]\to L^1(\calP(\Gamma),\ell)$ and maps $G^\pm:[0,T]\to L^1(\Sigma^{\pm})$ such that $U_t=\dd \sfP_t/\dd \Pi_n$, $G_t^{\pm}=\dd \sfJ^{\pm}/\dd \teta_{\Pi_n}^{\pm}$, and 
\begin{equation}\label{eq:fke_timereg3}
\begin{aligned}
    \partial_t U_t(\nu) &= n \int_{\calT} (G_t^-(\nu+\tfrac{1}{n}\delta_x,x)-G_t^+(\nu,x)) \, \kappa^+_{\nu}(\dd x)\\
    &\hspace{4em}+ n\int_{\calT} (G_t^+(\nu-\tfrac{1}{n}\delta_x,x)-G_t^-(\nu,x)) \, \kappa^-_{\nu}(\dd x).
\end{aligned}
\end{equation}
Alternatively, in terms of the net-flux $\jnet=G \teta_{\sfP}^+$ with $\Gnet:=G^+-G^-\circ \sfT^{n,+}$,
\begin{equation*}
\begin{aligned}
    \partial_t U_t(\nu)= n \int_{\calT} \Gnet_t(\nu-\tfrac{1}{n}\delta_x,x) \, \kappa^-_{\nu}(\dd x)-n \int_{\calT} \Gnet_t(\nu,x) \, \kappa^+_{\nu}(\dd x).
\end{aligned}
\end{equation*}
\end{lm}

\begin{remark}
Note that the estimate \eqref{eq:fle_timereg1} for the weighted total variation metric blows up as $n\to \infty$. For the proof of EDP-convergence we instead use a weaker metric, the transportation-like metric $W$ defined by \eqref{eq:c_uni_m}, which does behave uniform-in-$n$ for a sequence of curves with finite $\limsup_{n\to \infty} \calI_n$.
\end{remark}

\begin{proof}
Due to the continuity equation and after a monotone class argument, we have the crude estimate
\begin{equation*}
  \left|  \int_{\Gamma} F \dd (\sfP_t-\sfP_s) \right| \leq n \int_s^t \left[\int_{\Gamma\times \calT} (|F(\nu+\tfrac{1}{n}\delta_x)|+|F(\nu)|) \, \dd \sfJ_r^{+} + \int_{\Gamma\times \calT} (|F(\nu-\tfrac{1}{n}\delta_x)|+|F(\nu)|) \, \dd \sfJ_r^{-} \right] \dd r,
\end{equation*}
for any $F\in \calB_c(\Gamma)$.
Now fix $F\in \calB_c(\Gamma)$, and let $K:=\sup_{\nu\in \Gamma} F(\nu)(1+\nu(\calT)^2)$. Note that by the bounds of Lemma \ref{lm:fke_est} for any $\nu\in \Gamma_n$, we have the following estimates 
\begin{align*}
    |F|(\nu)&\leq K (1+\nu(\calT)^2)^{-1}\\
    |F|(\nu+\tfrac{1}{n}\delta_x)&\leq K \big(1+(\tfrac{1}{n}+\nu(\calT))^2\big)^{-1}
    \leq K (1+\nu(\calT)^2)^{-1},\\
    |F|(\nu-\tfrac{1}{n}\delta_x)&\leq K \big(1+(\tfrac{-1}{n}+\nu(\calT))^2\big)^{-1}
    \leq K \max\{1+\tfrac{2}{n^2},2\} (1+\nu(\calT)^2)^{-1},
\end{align*}
and therefore
\begin{equation*}
  \left|  \int_{\Gamma} F \dd (\sfP_t-\sfP_s) \right| \leq 4 n K \max\Bigl\{1+\tfrac{2}{n^2},2\Bigr\} \int_s^t \int_{\Gamma\times \calT} (1+\nu(\calT)^2)^{-1} \dd( \sfJ^+_r+\sfJ^-_r) \, \dd r.
\end{equation*}
Taking the supremum over all $F\in \calB_c(\Gamma)$ with $\sup_{\nu\in \Gamma} F(\nu)(1+\nu(\calT)^2)\leq 1$ we conclude that
\begin{align*}
   d_{TV,w}(\sfP_s,\sfP_t)&=\int_{\Gamma} (1+\nu(\calT)^2)^{-1}\, \dd |\sfP_t^1-\sfP_s^2|\\
   &\leq 4 n \max\Bigl\{1+\tfrac{2}{n^2},2\Bigr\} \int_s^t (1+\nu(\calT)^2)^{-1} (\dd \sfJ^+_r+\sfJ^-_r) \, \dd r.
\end{align*}

Next, suppose that $\sfP_t\ll \Pi_n, \sfJ^{\pm}_t\ll \teta_{\Pi_n}^{\pm}$ for all $t\in [0,T]$. Let $U_t=\dd \sfP_t/\dd \Pi_n$, $G_t^{\pm}=\dd \sfJ^{\pm}/\dd \teta_{\Pi_n}^{\pm}$. Note that by the absolutely continuity of $\sfP_t$ with respect to $d_{TV,w}$, the map $t\mapsto U_t$ is absolutely continuous in $L^1(\ell)$. Moreover, for every $F\in \calB_c(\Gamma)$ the continuity equation reads as 
\begin{align*}
    \int_{\Gamma} F (U_t-U_s) \,\dd \Pi_n &= \int_s^t \int_{\Gamma\times \calT} (F(\nu+\tfrac{1}{n}\delta_x)-F(\nu)) G_r^+(\nu,x) \, \dd \teta_{\Pi_n}^+ \, \dd r\\
    &\hspace{2em} + \int_s^t \int_{\Gamma\times \calT} (F(\nu-\tfrac{1}{n}\delta_x)-F(\nu)) G_r^-(\nu,x) \, \dd \teta_{\Pi_n}^- \, \dd r.
\end{align*}
But due to Lemma \ref{lm:revers}, the integrands can be rewritten as follows 
\begin{align*}
    \int_{\Gamma\times \calT} F(\nu\pm \tfrac{1}{n}\delta_x) G_r^{\pm}(\nu,x) \, \dd \teta_{\Pi_n}^{\pm}&= \int_{\Gamma\times \calT} F(\nu) G_r^{\pm}(\nu\mp\tfrac{1}{n}\delta_x,x) \, \dd (\sfT^{n,\pm}_{\#} \teta_{\Pi_n}^{\pm})\\
    &=\int_{\Gamma\times \calT} F(\nu) G_r^{\pm}(\nu\mp \tfrac{1}{n}\delta_x,x) \, \dd \teta_{\Pi_n}^\mp, 
\end{align*}
and therefore
\begin{align*}
    \int_{\Gamma} F (U_t-U_s) \,\dd \Pi_n &= \int_s^t \int_{\Gamma\times \calT} F(\nu)(G_r^-(\nu+\tfrac{1}{n}\delta_x,x)-G_r^+(\nu,x)) \kappa^+_{\nu}(\dd x)\, \dd \Pi_n(\dd \nu)\,  \, \dd r\\
    &\hspace{2em} +\int_s^t \int_{\Gamma\times \calT} F(\nu)(G_r^+(\nu-\tfrac{1}{n}\delta_x,x)-G_r^-(\nu,x)) \, \kappa^-_{\nu}(\dd x)\, \dd \Pi_n(\dd \nu) \, \dd r,
\end{align*}
which is the weak formulation of \eqref{eq:fke_timereg3}. Putting in the pre-factors $(1+\nu(\calT)^2)^{-1}$ to state the expression in terms of the finite measures $\ell$ and $\Sigma$, and noting that due to time-regularity $(1+\nu(\calT)^2)^{-1} \sfP_t$ is TV-regular, we can proceed as in Corollary 4.14 of \cite{PRST2020} and conclude the proof after redefining $U,G^{\pm}$ on negligible sets. 
\end{proof}

\subsection{Weak solutions}\label{ss:fke_sol}

In this section we will discuss the existence of weak solutions to \eqref{eq:FKEn}, i.e. solutions to
\begin{equation*}
    \partial_t \sfP = \ddiv^{n,+} \vartheta_{\sfP}^{+}+\ddiv^{n,-} \vartheta_{\sfP}^-,
\end{equation*}
in appropriate weak form, but with the property that $\calI_n(\sfP,\vartheta_{\sfP}^+,\vartheta_{\sfP}^-)\leq 0$. In the next section we will show that $\calI_n\geq 0$ and that gradient-flow solutions, i.e.\ those with $\calI_n=0$, are in fact unique. 

\begin{defi}
A curve $(\sfP_t)_{t\in [0,T]}$ is a weak solution to \eqref{eq:FKEn} if $\mathrm{supp} \,\sfP_t\in \Gamma_n$ for all $t\in [0,T]$, $\sfP_t$ is continuous in the narrow topology and for all $s,t\in [0,T]$, and all $F\in C_c(\Gamma)$,
\begin{equation*}
 \int_{\Gamma} F(\nu) \,\dd \sfP_t - \int_{\Gamma} F(\nu) \,\dd \sfP_s = \int_s^t \int_{\Gamma\times \calT} \left((\dder^{n,+} F) \,\dd \vartheta_{\sfP_t}^++(\dder^{n,-} F) \, \dd \vartheta_{\sfP_t}^- \right) \, \dd r.
    \end{equation*}
\end{defi}

\begin{remark}
Recall that $\int (1+\nu(X)^2)\,\dd \vartheta_{\sfP_t}^{\pm}\leq M_n$ independently of $\sfP_t$. Hence it is easy to check that $(\sfP)$ is a weak solution if and only if $(\sfP,\vartheta_{\sfP}^+,\vartheta_{\sfP}^-)\in \mathsf{CE}_{n}$.
\end{remark}

Moreover, solutions turn out to inherit polynomial mass-estimates from the initial condition, see e.g.\ Theorem 3.1 of \cite{Fournier2004} for the case in $\R^d$. While throughout we do not assume more from the initial condition than having finite entropy with respect to $\Pi_n$ (which does imply the finiteness of the first moment) and unfortunately arbitrary curves $(\sfP,\sfJ^+,\sfJ^-)$ with finite $\calI_n$ do not preserve moment estimates, we will include the statement for completeness. 

\begin{lm}
Fix any $p\geq 0$, and assume that $(\sfP)$ is a weak solution with initial datum satisfying
\begin{equation*}
    \int_{\Gamma} \nu(X)^p \,\sfP_0(\dd \nu)<\infty.
\end{equation*}
Then 
\begin{equation*}
 \sup_{t\in [0,T]} \int_{\Gamma} \nu(X)^p \,\sfP_t(\dd \nu)<\infty.
\end{equation*}
\end{lm}

\begin{proof}
Set $F(\nu):=f(\nu(X))$ with $f(z):=z^p$ and let $F_k(\nu):=f_k(\nu(X))$ be its sequence of truncations. Then for every $k$, 
\begin{align*}
    \int_{\Gamma} F_k\, \dd (\sfP_t-\sfP_0)&=\int_0^t \left(\int_{\Gamma\times \calT}\dder^{n,+} F_k \, \dd \vartheta_{\sfP_r}^++\int_{\Gamma\times \calT}\dder^{n,-} F \dd \vartheta_{\sfP_r}^-  \right) \, \dd r\\
    &\leq\int_0^t \int_{\Gamma\times \calT} \big(f_k(\nu(X)+\tfrac{1}{n})-f_k(\nu(X))\big) \,\sfP_r(\dd \nu)\,\kappa^+_{\nu}(\dd x)\, \dd r,
\end{align*}
since $f_k$ is non-decreasing and hence $\dder^{n,-} F_k\leq 0$. Moreover, note that $z(f(z+\tfrac{1}{n})-f(z))\leq C_{p,n}\, f(z)$ for a suitable constant $C_{p,n}$, and by monotonicity and non-negativeness of $f$ the same inequality holds for the truncations $f_k$. By a Gronwall-type argument we then obtain
\begin{equation*}
    \int_{\Gamma} F_k(\nu) \,\dd \sfP_t \leq e^{t K_{p,n}}\int_{\Gamma} F_k(\nu) \, \dd \sfP_0 \leq e^{t K}\int_{\Gamma} F(\nu) \, \dd \sfP_0,    
\end{equation*}
with the constant $K_{p,n}=C_{p,n}\|c\|_{\infty}\gamma(X)$ independent of $k$. Taking $k\to\infty$ we derive the desired inequality by monotone convergence.
\end{proof}

We can now state the existence result of a weak solution satisfying one-half of the Energy-Dissipation principle, which is complemented by the chain rule proved in Section \ref{ss:fke_chain}. The existence proof is one of EDP-convergence (see also Section \ref{s:edp}), bootstrapping from problems with bounded kernels and the results of \cite{PRST2020}.

\begin{thm}\label{thm:fke_solex}
Suppose that 
\[ \Ent(\bar{\sfP}|\Pi_n) < \infty. \]
Then there exist a weak solution $(\sfP)$ with initial datum $\bar{\sfP}$ such that 
\begin{equation*}
    \int_0^T \calR_n(\sfP_t,\vartheta_{\sfP_t}^+,\vartheta_{\sfP_t}^-)\, \dd t + \calF_n(\sfP_T)-\calF_n(\sfP_0)+\int_0^T \calD_n(\sfP_t) \, \dd t \leq 0.
\end{equation*}
\end{thm}

\begin{proof}
Fix any $\bar{\sfP}$ with $\Ent(\bar{\sfP}|\Pi_n)<\infty$. We proceed by approximating the unbounded kernel $\bar\kappa_n$ with bounded ones. For every $\eps>0$, we introduce the regularized jump kernel $\bar \kappa_{n}^{\eps}(\nu,\dd \eta)$ over $\Gamma$ defined by
\begin{equation*}
    \bar \kappa_{n}^{\eps}(\nu,\dd \eta):=\frac{1}{1+\eps \nu(\calT)\eta(\calT)} \bar \kappa_{n}(\nu,\dd \eta).
\end{equation*}
In terms of birth/death kernels this can be rewritten as
\begin{equation*}
      \bar\kappa_{n}^{\eps}(\nu,\dd \eta) = n \int_{\calT}\delta_{\nu + \tfrac{1}{n} \delta_x}(\dd\eta)\, \kappa^{+,\eps}_{\nu}(\dd x) + n \int_{\calT} \delta_{\nu-\tfrac{1}{n} \delta_x}(\dd\eta)\,\kappa^{-,\eps}_{\nu}(\dd x),
\end{equation*}
where
\begin{equation*}
    \kappa^{\pm,\eps}_{\nu}:=\frac{1}{1+\eps \nu(\calT)(\nu(\calT)\pm \tfrac{1}{n})}\kappa^{\pm}_{\nu}.
\end{equation*}
Note that
\begin{equation}\label{eq:fke_soledp1}
 \sup_{\nu\in \Gamma} \kappa^{\pm,\eps}_{\nu} < \infty \qquad \fA \eps>0.
\end{equation}
Correspondingly, we denote $\vartheta_{\sfP}^{\pm,\eps}$, $\Theta^{n,\pm,\eps}_{\sfP}$, $Q_{n,\eps}$, $Q_{n,\eps}^*$, $\calR_{n,\eps}$, $\calD_{n,\eps}$, $\calI_{n,\eps}$, $(\mathsf{FKE}_{n,\eps})$ as the relevant quantities, operators, functionals and forward Kolmogorov equations induced by $\kappa^{\pm,\eps}_{\nu}$. We will first show existence of gradient-flow solutions for the regularized problems, i.e.\ curves such that $\calI_{n,\eps}=0$, and then construct an appropriate limit curve as $\eps\to 0$.  

Thus, fix any $\eps>0$. Due to the bound \eqref{eq:fke_soledp1} it is clear that $Q_{n,\eps}$ is a bounded operator since 
\begin{equation*}
    \sup_{\nu\in \Gamma} \int_{\Gamma} \bar \kappa_n^{\eps}(\nu,\dd \eta) < \infty.
\end{equation*}
Moreover, since the prefactor $\nu(\calT)\eta(\calT)$ is symmetric under swapping of $\nu$ and $\eta$, it straightforward to verify that $\bar\kappa_{n}^{\eps}$ is still reversible with respect to the same invariant measure $\Pi_n$, i.e.\ we have
\[ \Pi_n(\dd \nu)\bar \kappa_{n}^{\eps}(\nu,\dd \eta) = \Pi_n(\dd \eta)\bar \kappa_{n}^{\eps}(\eta,\dd \nu). \]
The triple $(\Gamma,\Pi_n,\bar \kappa_{n}^{\eps})$ therefore satisfies the assumptions of \cite{PRST2020}. Keeping in mind the difference in definitions of $\Psi^*$ due to extra the factor $2$, by \cite[Theoren 6.6]{PRST2020} there exist a unique curve $U^{\eps}\in C^1([0,T],L^1(\Gamma,\Pi_n))$ such that $U_0=\dd \bar{\sfP}/\dd  \Pi_n$, and
\begin{equation*}
    \left\{\begin{aligned}
       \partial_t U_t(\nu) &= \int_{\Gamma} (U_t(\eta)-U_t(\nu))\bar \kappa_n^{\eps}(\nu,\dd \eta), \qquad \mbox{for a.e.\ $t\in [0,T]$},  \\
        \Ent(\sfP_0|\Pi_n)- \Ent(\sfP_T|\Pi_n)&= \int_0^T \int_{\Gamma\times \Gamma} \Psi\left(U_t(\eta)-U_t(\nu) \right) \sqrt{U_t(\nu)U_t(\eta)} \, \Pi_n(\dd \nu) \, \bar \kappa_{n}^{\eps}(\nu,\dd \eta) \, \dd t\\
 & \quad + \int_{\Gamma\times \Gamma} \left(\sqrt{U_t(\eta)}-\sqrt{U_t(\nu)}\right)^2 \Pi_n(\dd \nu) \, \bar \kappa_{n}^{\eps}(\nu,\dd \eta) \, \dd t,
    \end{aligned}\right.
\end{equation*}
with $\sfP_t:=U_t \Pi_n$ as usual. In particular the entropy $\Ent(\sfP|\Pi_n)$ decreases along the solution and hence 
\begin{equation*}
   \sup_{t\in [0,T]} \Ent(\sfP_t|\Pi_n) \leq  \Ent(\bar{\sfP}|\Pi_n).
\end{equation*} 

By evenness of $\Psi$, symmetry of $\Pi_n \bar \kappa_n^{\eps}$ and the identity \eqref{eq:fke_psiphi}, we can express for any $U$ after substituting for $\bar \kappa_{n}^{\eps}(\nu,\dd \eta)$ 
\begin{align*}
   \frac{1}{2} \int_{\Gamma\times \Gamma} &\Psi\left(U(\eta)-U(\nu) \right) \sqrt{U(\nu)U(\eta)} \, \Pi_n(\dd \nu) \, \bar \kappa_{n}^{\eps}(\nu,\dd \eta) \\
   &=  \int_{U(\eta)>0,U(\nu)>0}\phi\left(\sqrt{U(\nu)/U(\eta)}\right)\sqrt{U(\eta)U(\nu)} \, \Pi_n(\dd \nu) \, \bar \kappa_{n}^{\eps}(\nu,\dd \eta)\\
   &= \int_{U(\nu+n^{-1}\delta_x)>0,U(\nu)>0} \phi\left(\frac{\dd \teta_{\sfP}^{+,\eps}}{\dd \Theta^{+,n,\eps}_{\sfP}} \right)\sqrt{U(\nu+\tfrac{1}{n}\delta_x)U(\nu)} \, \Pi_n(\dd \nu) \, \kappa^{+,\eps}_{\nu}(\dd x) \\
   &\qquad +  \int_{U(\nu-n^{-1}\delta_x)>0,U(\nu)>0}  \phi\left(\frac{\dd \teta_{\sfP}^{-,\eps}}{\dd \Theta^{-,n,\eps}_{\sfP}} \right)\sqrt{U(\nu-\tfrac{1}{n}\delta_x)U(\nu)} \, \Pi_n(\dd \nu) \, \kappa^{-,\eps}_{\nu}(\dd x) \\
   &=\calR_{n,\eps}\left(\sfP,\teta_{\sfP}^{+,\eps},\teta_{\sfP}^{-,\eps}\right).
\end{align*}
Moreover, it is straightforward to check that 
\[ \int_{\Gamma\times \Gamma} \left(\sqrt{U_t(\eta)}-\sqrt{U_t(\nu)}\right)^2 \Pi_n(\dd \nu) \, \bar \kappa_{n}^{\eps}(\nu,\dd \eta) = \calD_{n,\eps}(\sfP), \] 
and therefore with $J^{\pm}:=\teta_{\sfP}^{\pm,\eps}$ we conclude
\[ \calI_{n,\eps}(\sfP,\sfJ^+,\sfJ^-)=0. \]
Finally, note that by Lemma \ref{lm:fke_est} and Remark \ref{rem:fke_fischereq}
\begin{equation*}
  \begin{aligned}
 \Ent\left(\sfJ^{\pm}|\Theta_{\sfP}^{n,+,\eps}\right)&=\int_{\Gamma\times \calT} \Upsilon\left(\frac{\dd \teta_{\sfP}^{\pm,\eps}}{\dd \Sigma},\frac{\dd \vartheta_{\sfP}^{\pm,\eps}}{\dd \Sigma},\frac{\dd (\sfT^{n,\mp}_{\#}\vartheta_{\sfP}^{\mp,\eps})}{\dd \Sigma}\right)\dd  \Sigma,\\
\calD_{n,\eps}&=2H^2(\vartheta_{\sfP}^{\pm,\eps},\sfT^{n,\mp}_{\#}\vartheta_{\sfP}^{\mp,\eps}),
  \end{aligned}
\end{equation*}
for any dominating measure $\Sigma$, which are both non-negative, convex and vaguely lower-semicontinuous functionals of $\vartheta_{\sfP}^{\pm,\eps},\sfT^{n,\mp}_{\#}\vartheta_{\sfP}^{\mp,\eps}$ in $\calM_{loc}(\Gamma\times \calT)$, see \cite[Theorem 3.4.3]{Buttazzo1989}. 

\medskip

Next, we consider the sequence of pairs $(\sfP^{\eps},\sfJ^{\pm,\eps})$ stemming from the regularized problems above, satisfying 
\[ \calI_{n,\eps}(\sfP^{\eps},\sfJ^{+,\eps},\sfJ^{-,\eps})=0 \qquad \fA \eps>0. \] 
As for a priori estimates, we have 
\begin{equation}\label{eq:fke_edpent2}
   \sup_{\eps,t\in [0,T]} \Ent(\sfP_t^{\eps}|\Pi_n) \leq  \Ent(\bar{\sfP}|\Pi_n),
\end{equation} 
and
\begin{equation*}
 \kappa^{\pm,\eps}_{\nu}(\calT)\leq \kappa^{\pm}_{\nu}(\calT) \qquad \fA \eps>0.
\end{equation*}
From the latter, it can be shown similarly as in Lemma \ref{lm:fke_timereg} that we have the equicontinuity result
\begin{equation*}
    d_{TV,w}(\sfP^{\eps}_t,\sfP^{\eps}_s)\leq 2 n \max\{1+\tfrac{2}{n^2},2\}   |t-s|.
\end{equation*}
Here $d_{TV,w}$ is the weighted total variation-metric defined in \eqref{eq_fked} as
\begin{equation*}
    \begin{aligned}
         d_{TV,w}(\sfP^{\eps}_t,\sfP^{\eps}_s)&:=\int_{\Gamma} (1+\nu(\calT)^2)^{-1}\, \dd |\sfP^1-\sfP^2|, \fA \eps>0,\, \fA s,t\in [0,T].
    \end{aligned}
    \end{equation*}
Recall that $d$ is lower semicontinuous with respect to the narrow topology and convergence in $d$ implies narrow convergence on narrowly pre-compact sets. Since $\Ent(\sfP^{\eps}_t|\Pi_n)$ is bounded uniformly in $\eps$ and $t$ and $\Ent(\cdot|\Pi_n)$ is narrowly coercive we obtain by a standard Arzel\'a-Ascoli argument, up to choosing a subsequence, the existence of a curve $t\mapsto\sfP_t$ such that 
\[
\sfP_t^{\eps} \to \sfP_t \quad \mbox{narrowly\; for all $t\in [0,T]$}.
\]
Note that by the estimate \eqref{eq:fke_edpent2} and lower-semicontinuity of the entropy, we have that for every $t\in [0,T]$, the sequence of measures $\sfP_t^{\eps}$ converge \emph{setwise} to $\sfP_t$ and $\Ent(\sfP_t|\Pi_n)\leq \Ent(\bar \sfP|\Pi_n)<\infty$. Moreover, $\kappa^{\pm,\eps}_{\nu} \nearrow \kappa^{\pm}_{\nu}$ as $\eps\to 0$ for every $\nu$, and hence setwise convergence of $\sfP^{\eps}_t$ implies setwise convergence on pre-compact sets of $\Gamma\times \calT$ for 
\[ \teta_{\sfP^{\eps}_t}^{\pm,\eps}(\dd \nu,\dd x)=\sfP_t^{\eps}(\dd \nu)\kappa^{\pm,\eps}[\nu](\dd x), \]
see e.g.\ \cite[Lemma 2.4]{PRST2020} for the case of set-wise convergence for bounded jump kernels. In particular we have the vague convergence
\begin{align*}
    \vartheta^{\pm,\eps}_{\sfP_t^{\eps}} \to \vartheta_{\sfP_t}^\pm,\qquad 
     \sfT^{n,\pm}_{\#}\vartheta^{\pm,\eps}_{\sfP_t^{\eps}} \to \sfT^{n,\pm}_{\#}\vartheta_{\sfP_t}^\pm.
\end{align*}
It is straightforward to check that we can pass to the limit in the continuity equation \eqref{eq:fkecont1}, and in particular, derive that $\sfP$ is a weak solution to the unregularized problem. 

Finally, recall that $\calF_n(\sfP_T)$ is convex in and narrowly lower semicontinuous in $\sfP^{\eps}_T$, and as shown above the action $\calR^{\eps}_n$ is jointly convex and lower semicontinuous in $(\vartheta^{\pm,\eps}_{\sfP^{\eps}},\sfT^{n,\mp}_{\#}\vartheta^{\mp,\eps}_{\sfP^{\eps}})$. Proceeding as in Remark \ref{rem:fke_fischereq}, we also find that the Fisher information is jointly convex and lower semicontinuous in $(\vartheta^{\pm,\eps}_{\sfP^{\eps}},\sfT^{n,\mp}_{\#}\vartheta^{\mp,\eps}_{\sfP^{\eps}})$ if $\sfP^{\eps}$ are contained in sub-level sets of $\calF_n$. 
Therefore, we conclude that 
\begin{align*}
    \calI_{n}(\sfP)&\leq \liminf_{\eps\to 0} \left( \int_0^T \calR_{n,\eps}(\sfP^{\eps}_t,\sfJ_t^{+,\eps},\sfJ_t^{-,\eps}) \, \dd t+\calF_n(\sfP_T^{\eps})-\calF_n(\bar \sfP)+\int_0^T \calD_{n,\eps}(\sfP^ {\eps}_t) \, \dd t  \right)\\ &=\calI_{n,\eps}(\sfP^{\eps},\sfJ^{+,\eps},\sfJ^{-,\eps}) =0,
\end{align*}
thus establishing the claim.
\end{proof}

\subsection{Variational characterization}\label{ss:fke_chain}

We will now present the chain rule for the entropy. The strategy of the proof is similar to the mean-field case and the proof for jump processes of \cite{PRST2020}, with the difference that due to the unboundedness of $\bar \kappa$ we need a two-fold regularization of the entropy, namely via truncations and compactly supported multipliers. 

\begin{thm}\label{thm:fke_chain}
For any $(\sfP,\sfJ^+,\sfJ^-)\in \mathsf{CE}_{n}$ with $\calF_n(\sfP_0)<\infty$ and $\calI_{n}(\sfP,\sfJ^+,\sfJ^-) < \infty$, it holds that $t \mapsto \calF_{n}(\sfP_t)$ is absolutely continuous and
\begin{equation*}
    \frac{\dd \,}{\dd t} \calF_n(\sfP_t)=\int_{\Gamma\times \calT} \bigl(\log U(\nu+\delta_x)-\log U(\nu)\bigr)\,\dd \jnet_t \, \dd t, \qquad \mbox{for a.e.\ $t\in [0,T]$.}
\end{equation*}
Moreover, $\calI_n\geq 0$, and if $\calI_n=0$ we have
\[ \sfJ^{\pm}_t=\sfP_t \kappa_{\nu}^{\pm} \qquad \mbox{for a.e.\ $t\in [0,T]$.}\]
\end{thm}

\begin{proof}
For any curve $\sfP$ with $\sfP\ll \Pi_n$ for all $t\in [0,T]$ we will use 
\begin{equation*}
    S^{k,m}_t=:\int_{\Gamma} \phi_{k,m}(U_t)\, \dd \Pi_n, \qquad S^{m}_t=:\int_{\Gamma} \phi_{m}(U_t)\, \dd \Pi_n,
\end{equation*}
where $\phi_{k,m}(U,\nu)=\chi_k(\nu)\phi_m(U)$, $k,m\in \N$ with $\phi_m$ the previously defined regularized entropy functions, and $\chi_k:=f_k(\nu(\calT))  \in C_c(\Gamma)$ compactly supported multipliers defined via
\begin{equation*}
  f_k(z):=\left\{ \begin{aligned}
      1&, \qquad &&0\leq z\leq k, \\
      2-\frac{z}{k} &, \qquad && k\leq z\leq 2k, \\
      0 &, \qquad && z\geq 2k. \\
    \end{aligned}\right.
\end{equation*}
Note that $|f_k|\leq 1$,  $|f_k'(z)|z\leq 2$ uniformly in $k$, $f_k$ converges monotonically to $1$, and $|\dder^{n,+} \xi_k|(\nu,x)\leq 3/(1+\nu(\calT))$ if $k\geq 1$. In addition, recall that $\phi_m'$ converges pointwise to $\phi'$ and $|\phi'_m|,\phi_m$ converge monotonically to $|\phi'|,\phi$ respectively, and in particular, 
\begin{equation*}
    \lim_{k\to \infty} S^{k,m}_t:=S^m_t, \qquad \lim_{m\to \infty} S^m_t=\Ent(\sfP_t|\Pi_n).
\end{equation*}

Moreover, let the distributional derivatives with respect to $\sfP$ be defined as 
\[ DS^{k,m}_t(\nu):=\chi_k(\nu) \phi'_m(U_t(\nu)), \qquad DS^{m}_t(\nu):= \phi'_m(U_t(\nu))  \]
Note that pointwise $\lim_{k\to \infty} \dder^{n,\pm}  DS^{k,m}_t=\dder^{n,\pm}  DS^{m}_t$ and $\lim_{m\to \infty} \dder^{n,\pm} DS^{m}_t = \dder^{n,\pm}  \phi'(U_t)$. 

\medskip

Now, consider a curve $(\sfP,\sfJ^+,\sfJ^-)\in \mathsf{CE}_n$ with $\calF_n(\sfP_0)<\infty$ and $\calI_n<\infty$. Since $\Ent$ is bounded from below 
\begin{equation*}
\int_0^T \calR_n(\sfP_t,\sfJ^+_t,\sfJ^-_t)\, \dd t<\infty, \qquad \int_0^T \calD_{n}(\sfP_t) \, \dd t < \infty,
\end{equation*}
and therefore $\sfP_t\ll \Pi_n$, $\sfJ^{\pm}_t\ll \Theta^{n,\pm}_{\sfP_t} \ll \vartheta^{\pm}_{\Pi_n}$ for a.e.\ $t\in [0,T]$, with
\[
    \Theta^{n,\pm}_{\sfP_t}(\dd \nu,\dd x)=\sqrt{U_t(\nu)U_t(\nu\pm\tfrac{1}{n}\delta_x)}\, \vartheta_{\Pi_n}^{\pm}(\dd \nu,\dd x). \]
In particular $U_t(\nu)$, $U_t(\nu\pm\tfrac{1}{n}\delta_x)>0$ for $\sfJ^{\pm}_t, \Theta^{n,\pm}_{\sfP_t}$-a.e. $\nu,x$. 

Moreover, set $\sfJ^{\pm}_t=G_t^{\pm} \vartheta_{\Pi}^{\pm}$, $\jnet_t=G^{\mathrm{net}}_t \vartheta_{\Pi}^+$ (or $G^{\mathrm{net}}_t:=G_t^+-G_t^-\circ \sfT^{n,+})$, and
\begin{align*}
    \ell:=(1+\nu(\calT)^2)^{-1} \Pi_n,\qquad \Sigma^{\pm}:=(1+\nu(\calT)^2)^{-1} \teta^{\pm}_{\Pi_n}.
\end{align*} 
By Lemma \ref{lm:fke_timereg}, the map $t\mapsto U_t$ is absolutely continuous and a.e.\ differentiable in $L^1(\calP(\Gamma),\ell)$ with 
\begin{equation*}
\begin{aligned}
    \partial_t U_t(\nu) &= n \int_{\calT} (G_t^-(\nu+\tfrac{1}{n}\delta_x,x)-G_t^+(\nu,x)) \, \kappa^+_{\nu}(\dd x)\\
    &\hspace{4em}+ n\int_{\calT} (G_t^+(\nu-\tfrac{1}{n}\delta_x,x)-G_t^-(\nu,x)) \, \kappa^-_{\nu}(\dd x),
\end{aligned}
\end{equation*}
or in terms of the net-flux,
\begin{equation*}
\begin{aligned}
    \partial_t U_t(\nu)= n \int_{\calT} G^{\mathrm{net}}_t(\nu-\tfrac{1}{n}\delta_x,x) \, \kappa^-_{\nu}(\dd x)-n \int_{\calT} G^{\mathrm{net}}_t(\nu,x) \, \kappa^+_{\nu}(\dd x).
\end{aligned}
\end{equation*}
Therefore, since $(1+\nu(\calT)^2)$ is bounded from above and below on the support of $\xi_k$, it is clear that for every $m,n$ the maps $t\mapsto S_t^{k,m}$ are Lipschitz, absolutely continuous and for a.e.\ $t\in [0,T]$
\begin{align*}
    \frac{\dd \,}{\dd t} S_t^{k,m}&=n \int_{\calT} DS_t^{k,m}(\nu)  G^{\mathrm{net}}_t(\nu-\tfrac{1}{n}\delta_x,x) \, \kappa^-_{\nu}(\dd x)-n \int_{\calT} DS_t^{k,m}(\nu) G^{\mathrm{net}}_t(\nu,x) \, \kappa^+_{\nu}(\dd x)\\
    &=\int_{\Gamma\times \calT} \dder^{n,+} DS_t^{k,m} \, \dd \sfJ_t,
\end{align*}
and in particular, for all $s,t\in [0,T]$, 
\begin{equation}\label{eq:fke_crmods}
    S_t^{k,m}-S_s^{k,m} = \int_s^t \int_{\Gamma\times \calT} \dder^{n,+} DS_r^{k,m} \, \dd \sfJ_r.
\end{equation} 
Recall that the following convergences hold pointwisely:
\[
\lim_{m\to \infty} \lim_{k\to\infty} \dder^{n,+} DS_t^{k,m}=\dder^{n,+} \phi'(U_t),\qquad\text{and}\qquad \lim_{k\to \infty} \dder^{n,+} \xi_k=0.
\]
Moreover, the following estimate holds for every $(\nu,x)$:
\begin{align*}
    |\dder^{n,+} DS_t^m|(\nu,x) &\leq \|\xi_k\|_{\infty} |\dder^{n,+} DS_t^m|(\nu,x)+\|\phi'_m\|_{\infty} |\dder^{n,+} \xi_k|(\nu,x) \\
    &\leq |\dder^{n,+} DS_t^m|(\nu,x)+3 m (1+\nu(\calT))^{-1}\\
    &\leq |\dder^{n,+} \phi'(U_t)|(\nu,x)+3 m (1+\nu(\calT))^{-1},
\end{align*}
where the final inequality follows from the truncation inequality for discrete derivatives, i.e.\ $|\phi_m(\eta)-\phi_m(\nu)|\leq |\phi(\eta)-\phi(\nu)|$. Note that by Lemma \ref{lm:fke_est}, for any $\sfP,\sfJ^{\pm}$ with finite $\calR_n$ that 
\[ \int_{\Gamma \times \calT} (1+\nu(\calT))^{-1} \dd |\jnet|<\infty,  \]
and moreover 
\begin{equation*}
 \frac{1}{2n}\int_{\Gamma\times \calT} |\dder^{n,+} \phi'(U)|\,\dd |\jnet| \, \leq \calR_n(\sfP,\sfJ^+,\sfJ^-)+\int_{\Gamma\times \calT} \Psi^*\left(\frac{1}{2n}\dder^{n,+} \phi'(U)\right) \dd \Theta_{\sfP}^{n,+},
\end{equation*}
with
\begin{align*}
   \calD_n^-(\sfP)&:= \int_{\Gamma\times \calT} \Psi^*\left(\frac{1}{2n}\dder^{n,+} \phi'(U)\right) \dd \Theta_{\sfP}^{n,+} \\
   &=\int_{U(\nu+n^{-1}\delta_x)>0,U(\nu)>0} \Psi^*\left(\log U(\nu+\tfrac{1}{n}\delta_x)-U(\nu)\right) \sqrt{U(\nu+\tfrac{1}{n}\delta_x)U(\nu)}\, \dd \teta_{\sfPi_n}^+\\
   &=\int_{U(\nu+n^{-1}\delta_x)>0,U(\nu)>0}  \left(\sqrt{U(\nu+\tfrac{1}{n}\delta_x)}-\sqrt{U(\nu)}\right)^2 \dd \vartheta^+_{\Pi_n}\\
   &\leq \calD_n(\sfP).
\end{align*}
Therefore, since $\Ent(\sfP_0|\Pi_n)<\infty$ we find by a dominated convergence argument and taking subsequent limits in $k$ and $m$ in \eqref{eq:fke_crmods} that $\Ent(\sfP_t|\Pi_n)<\infty$ for all $t\in [0,T]$,
\begin{align*}
    \Ent(\sfP_t)-\Ent(\sfP_s) &= \int_s^t \int_{\Gamma\times \calT} \dder^{n,+} \phi'(U_r) \, \dd \jnet_r\, \dd r, \qquad s,t\in [0,T]\\
    \int_{\Gamma\times \calT} |\dder^{n,+} \phi(U_t)| \,\dd |\jnet_r| &\leq \calR_n(\sfP_t,\sfJ^+_t,\sfJ^-_t)+\calD^-_n(\sfP_t), \qquad t\in [0,T]. 
\end{align*}
and 
\begin{equation*}
  \calI_n\geq \int_0^T \calR_n(\sfP_t,\sfJ^+_t,\sfJ^-_t)\, \dd t + \tfrac{1}{2n}\left(\Ent(\sfP_T)-\Ent(\sfP_0)\right) +  \int_0^T \calD^-_n(\sfP_t)\, \dd t\geq 0.
\end{equation*}
\medskip

Next, assume that $\calI_n=0$. Then the above arguments imply that for a.e.\ $t\in [0,T]$,
\begin{equation}\label{eq:fke_chainzero}
    \calR_n(\sfP_t,\sfJ^+_t,\sfJ^-_t)\,  +\frac{1}{2n}\int_{\Gamma\times \calT} \dder^{n,+} \phi'(U_t) \, \dd \jnet_t+  \calD^-_n(\sfP_t) =0.
\end{equation} 
To simplify manipulations, let $U^{\pm}(\nu,x):=U\circ \sfT_x^{n,\pm}=U(\nu\pm\tfrac{1}{n}\delta_x)$. Note that for the actions, 
\begin{align*}
    \Ent(\sfJ^{+}|\Theta^{n,+}_{\sfP})&=\int_{\Gamma\times \calT} 1_{U,V>0} \, \phi\left(G^{+} U/U^+\right) \sqrt{U U^+} \, \dd \teta_{\Pi_n}^{+},\\
    \Ent(\sfJ^{-}|\Theta^{n,-}_{\sfP})&=\int_{\Gamma\times \calT} 1_{U,U^->0} \, \phi\left(G^{-} U^-/U\right) \sqrt{U U^-} \, \dd \teta_{\Pi_n}^{-}\\
    &=\int_{\Gamma\times \calT} 1_{U,U^+>0} \, \phi\left(G^{-} U^+/U\right) \sqrt{U U^+} \, \dd \teta_{\Pi_n}^{+},
\end{align*}
for the modified Fisher information $\calD_n^-$,
\begin{align*}
  \calD_n^-(\sfP)&=\int_{\Gamma\times \calT}  1_{U,U^+>0}  \left(\sqrt{U^+}-\sqrt{U}\right)^2 \, \dd \vartheta^+_{\Pi_n},
\end{align*}
and finally 
\[
    \frac{1}{2n}\int_{\Gamma\times \calT} \dder^{n,+} \phi'(U) \, \dd \sfJ=\frac{1}{2} \int_{\Gamma\times \calT} (\phi'(U^+)-\phi'(U)) (G^+-G^-\circ\sfT^{n,-}) \, \dd \teta_{\Pi_n}^+,
\] 
which due to $\sfJ^{\pm}\ll \Theta^{n,\pm}_{\sfP}$ is equal to 
\[\frac{1}{2} \int_{\Gamma\times \calT} 1_{U,U^+>0} (\phi'(U^+)-\phi'(U) (G_t^+-G_t^-\circ\sfT^{n,-}) \, \dd \teta_{\Pi_n}^+.\]
Therefore, after some cumbersome rewriting, the integrands of the left-hand side of \eqref{eq:fke_chainzero} reads as the indicator functions over $\{U,U^+>0\}$ multiplied by the terms
\begin{align*}
&\phi\left(G^{+} U/U^+\right) \sqrt{U U^+} + \tfrac{1}{2} (\phi'(U^+)-\phi'(U) )G_t^++\phi^*\left(-\tfrac{1}{2}(\phi'(U^+)-\phi'(U))\right)\\
&\quad+\;\phi\left(G^{-}\circ \sfT^{n,+} U^+/U\right)\sqrt{U U^+} - \tfrac{1}{2} (\phi'(U^+)-\phi'(U)) G^{-}\circ \sfT^{n,+}+\phi^*\left(-\tfrac{1}{2}(\phi'(U^+)-\phi'(U))\right),
\end{align*}
since 
\[\phi^*\left(-\tfrac{1}{2}(\phi'(U^+)-\phi'(U))\right)=U-\sqrt{U U^+}, \quad \phi^*\left(\tfrac{1}{2}(\phi'(U^+)-\phi'(U))\right)=U^+-\sqrt{U U^+}. \]
By duality of $\phi,\phi^*$ we have  $G^+=U$ and $G^{-}\circ \sfT^{n,+}=U^+$, hence $G^-=U$ as well. Subsequently we can conclude that $\calI_n=0$ if and only if $\sfJ^{\pm}_t=\teta_{\sfP_t}^{\pm}$ for a.e. $t\in [0,T]$ and a.e.\ $\nu,x$.
\end{proof}

Together, Theorems \ref{thm:fke_chain} and \ref{thm:fke_solex} provide a proof of the variational characterization for the forward Kolmogorov equation. 
\begin{proof}[Proof of Theorem \ref{thm:fke_main}]
Under the assumption of $\calF_n(\sfP_0)<\infty$ we have by Theorem \ref{thm:fke_chain} a chain rule for the entropy, the inequality $\calI_n\geq 0$, and the statement that $\calI_n(\sfP,\sfJ^+,\sfJ^-)=0$ implies that $\sfP$ is a weak solution. Moreover, due to Theorem \ref{thm:fke_solex} there exists a weak solution with $\calI_n\leq 0$.

It remains to show that gradient-flow solutions are unique, which is a classical argument using the strict convexity of $\calF_n$, e.g. see Theorem 5.9 of \cite{PRST2020}. Suppose that there exist two curves $\sfP^1,\sfP^2$ such that $\sfP^1_0=\sfP^2_0=\bar \sfP$, $\calI_n(\sfP^1,\vartheta_{\sfP^1}^+,\vartheta_{\sfP^1}^-)$ and $\calI_n( \sfP^2,\vartheta_{\sfP^2}^+,\vartheta_{\sfP^2}^-)=0$. Applying the chain rule it is straightforward to verify that for a gradient-flow solution $\calI_{n}^t=0$ for every $t \in [0,T]$, where 
\[ \calI_n^t(\sfP,\sfJ^+,\sfJ^-):=  \int_0^t \calR_n(\sfP_r,\sfJ^+_r,\sfJ_r^-)\, \dd r + \calF_n(\sfP_t)-\calF_n(\bar \sfP)+\int_0^t \calD_n(\sfP_r) \, \dd r,\]
and that $\calI_n^t\geq 0$ for arbitrary curves with initial condition $\bar \sfP$. 

Now, define $\tilde \sfP_t=\tfrac{1}{2}\sfP^1+\tfrac{1}{2}\sfP^2$ and note that $(\tilde \sfP,\vartheta_{\tilde \sfP}^+,\vartheta_{\tilde \sfP}^-)\in \mathsf{CE}_n$  as well, and 
\[ \vartheta_{\tilde \sfP}^{\pm}=\tfrac{1}{2}\vartheta_{\sfP^1}^{\pm}+\tfrac{1}{2}\vartheta_{\sfP^2}^{\pm}. \]
Fix any $t \in [0,T]$ and suppose that $\sfP_t^1\neq \sfP_t^{2}$. Then by convexity of $\calR_n$ and $\calD_n$, and strict convexity of $\calF_n$, we have 
\begin{align*}
    \calI_n^t(\tilde \sfP,\vartheta_{\tilde \sfP}^{+},\vartheta_{\tilde \sfP}^{-})&=  \int_0^t \calR_n(\tilde \sfP_r,\vartheta_{\tilde \sfP_r}^+,\vartheta_{\tilde \sfP_r}^-)\, \dd r + \calF_n(\tilde \sfP_t)-\calF_n(\bar \sfP)+\int_0^t \calD_n(\sfP_r) \, \dd r\\
    &< \tfrac{1}{2}\calI_n^t(\sfP^1,\vartheta_{\sfP^1}^+,\vartheta_{\sfP^1}^-)+\tfrac{1}{2}\calI_n^t(\sfP^2,\vartheta_{\sfP^2}^+,\vartheta_{\sfP^2}^-)=0,
\end{align*}
which leads to a contradiction, and hence $\sfP_t^1= \sfP_t^{2}$ for all $t \in [0,T]$. 
\end{proof}

\section{Liouville equation and lifted dynamics}\label{s:liouv}

In this section, we will consider the variational formulation for our proposed limit of the forward Kolmogorov equation \ref{eq:FKEn}, namely the \emph{Liouville equation} 
\begin{equation}\label{eq:liouv2}\tag{{\sf Li}}
    \partial_t \sfP_t +\mathrm{div}_{\Gamma} \left(\sfP_t \bigl(\kappa^+-\kappa^- \bigr)\right)=0.
\end{equation}
It can be interpreted as a transport equation lifted from the mean-field dynamics, in the sense that it describes the evolution of the law of a deterministic process satisfying the mean-field equation but with possibly random initial conditions. We will consider the same ingredients as in previous sections, namely a non-negative EDP functional consisting of an action term, a difference of free energies, and a corresponding Fisher information term. The main technical tool that we use is a new superposition principle, which allows us to prove the chain rule via the results on mean-field curves of Section \ref{s:mf}. 

\medskip

Solutions to \eqref{eq:liouv2} are defined as appropriate weak solutions to
\begin{equation*}
    \partial_t \sfP_t =Q_{\infty}^* \, \sfP_t,
\end{equation*}
where $\sfP_t \in \calP(\Gamma)$ for all $t\in [0,T]$ and the operator $Q_{\infty}^*$ is the dual of $Q_{\infty}$ given by
\begin{equation*}
\begin{aligned}
   (Q_{\infty} F)(\nu) &= \int_{\calT}(\grad F)(\nu,x) V[\nu](\dd x),\\
   V[\nu]&:=\kappa^+[\nu]-\kappa^-[\nu],
\end{aligned}
\end{equation*}
for all $F\in \Cyl_c(\Gamma)$. Here $\Cyl_c(\Gamma)$ is the space of all compactly supported smooth cylinder functions, i.e.\ those of the form 
\begin{equation*}
    F(\nu)=g\left(\langle 1,\nu\rangle,\langle f_1,\nu\rangle,\dots,\langle f_m,\nu\rangle \right),
\end{equation*}
where $g\in C^{\infty}_c(\mathbb{R}^{m})$ with $m \in \N$, and $f_1,\dots,f_m\in C_b(\calT)$, and $\grad$ is the distributional gradient defined by
\begin{equation*}
    \mathrm{grad}_{\Gamma}\, F(\nu,x)= (\nabla g)\left(\langle 1,\nu\rangle,\langle f_1,\nu\rangle,\dots,\langle f_m,\nu\rangle \right) \cdot (1,f_1(x),\dots,f_m(x))^\top.
\end{equation*}
To be precise, we consider the following type of solutions.
\begin{defi}
A curve $(\sfP_t)_{t\in [0,T]}$ is a weak solution to \eqref{eq:liouv} if $\sfP_t$ is continuous in the narrow topology and for all $s,t\in [0,T]$, and all $F\in \Cyl_c(\Gamma)$,
\begin{equation}\label{eq:liouv3}
 \int_{\Gamma} F(\nu) \,\dd \sfP_t - \int_{\Gamma} F(\nu) \,\dd \sfP_s = \int_s^t \int_{\Gamma\times \calT} (\grad F)(\nu,x) V[\nu](\dd x) \sfP_t(\dd \nu) \, \dd r.
    \end{equation}
\end{defi}

\begin{remark}\label{rem:li_solsup}
Note that \eqref{eq:liouv2} is the transport equation associated to the measure-valued vector field $V[\nu]$. Now let the flow $G:[0,T]\times \Gamma\to \Gamma$ be the unique strong solution to the mean-field equation, i.e.\ with
\begin{equation}\label{eq:li_sol}
   \partial_t G_t[\nu]=V[G_t[\nu]]  .
\end{equation}
As will be shown in Section \ref{ss:li_sol}, $\sfP_t:=(G_t)_{\#} \bar \sfP$ is a weak solution to \eqref{eq:liouv} for any initial data $\bar \sfP\in \calP(\Gamma)$. In particular, if $\nu_t$ is a solution to \eqref{eq:mf} than $\sfP_t:=\delta_{\nu_t}$ is a weak solution to \eqref{eq:liouv2}. 
\end{remark}

Instead of the solution to \eqref{eq:liouv2}, we will now consider arbitrary curves satisfying
\begin{equation}\label{eq:liouv_cei}\tag{$\mathsf{CE}_{\infty}$}
    \partial_t \sfP_t + \mathrm{div}_{\Gamma}(J_t^+-J_t^-)=0,
\end{equation}
in the following appropriate distributional sense. 

\begin{defi}[Continuity equation]\label{defi:contli}\,

A triple  $(\sfP,\sfJ^+,\sfJ^-)$ satisfies the continuity equation $\mathsf{CE}_{\infty}$, if 
\begin{enumerate}
\itemsep0.1em 
\item\label{li:cond0} the curve $[0,T]\ni t\mapsto \sfP_t\in \calP(\Gamma)$ is narrowly continuous,
\item\label{li:cond1} the Borel family $(\sfJ^{\pm}_t)_{t\in [0,T]}\in \calM^+_{loc}(\Gamma\times \calT)$ satisfies 
\[ \int_0^T \int_{\Gamma\times \calT} (1+\nu(\calT)^2)^{-1}\,\dd \sfJ^{\pm}_{t} \, \dd t<\infty, \]
    \item\label{li:cond2} for every $s,t\in [0,T]$ and all $F\in \Cyl_c(\Gamma)$
     \begin{equation*}
        \int_{\Gamma} F(\nu) \,\dd \sfP_t - \int_{\Gamma} F(\nu) \,\dd \sfP_s = \int_s^t \int_{\Gamma\times \calT} \mathrm{grad}_{\Gamma}F \,(\dd \sfJ_r^+-\dd \sfJ_r^-) \, \dd r.
    \end{equation*}
\end{enumerate}
\end{defi}

Moreover, let us introduce the EDP-functional. Recall from Section \ref{s:fke} the notation $\vartheta_{\sfP}^{\pm}(\dd \nu,\dd x):=\kappa^{\pm}[\nu](\dd x)\sfP(\dd \nu)$. 

\begin{defi}\label{defi:liouv}
Let $\Theta^{\infty}_{\sfP}\in \calM_{loc}(\Gamma\times \calT)$ be the geometric average of $\vartheta^{+}_{\sfP}$ and $\vartheta^{-}_{\sfP}$, i.e.\ 
\begin{equation*}
    \Theta^{\infty}_{\sfP}(\dd \nu,\dd x):=\sqrt{\frac{\dd \vartheta^{+}_{\sfP}}{\dd \Sigma}\frac{\dd \vartheta^{-}_{\sfP}}{\dd\Sigma}}\, \,\dd \Sigma,
\end{equation*}
for any dominating measure $\Sigma$. We define the following objects:
\begin{itemize}
    \item The dissipation potential $\calR_{\infty}:\calP(\Gamma)\times \calM_{loc}^+(\Gamma\times \calT)^2\to[0,+\infty]$,
\begin{equation*}
  \calR_{\infty}(\sfP,\sfJ^+,\sfJ^-):=\Ent(\sfJ^{+}|\Theta^{\infty}_{\sfP})+\Ent(\sfJ^{-}|\Theta^{\infty}_{\sfP}).
\end{equation*}
\item The dual dissipation potential $\calR^*_{\infty}:\calP(\Gamma)\times \calB_c(\Gamma \times \calT)^2\to \R$,
\begin{equation*}
      \calR_{\infty}^*(\sfP,\omega^+,\omega^-):=\int_{\Gamma\times \calT} (e^{\omega^{+}}-1)\, \dd \Theta^{\infty}_{\sfP}+\int_{\Gamma\times \calT} (e^{\omega^{-}}-1)\, \dd \Theta^{\infty}_{\sfP}.
\end{equation*}
\item The free energy $\calF_{\infty}:\calP(\Gamma)\to[0,+\infty]$,
\begin{equation*}
\calF_{\infty}(\sfP):=\int_{\Gamma} \calF_{MF}(\nu)\, \sfP(\dd \nu).
\end{equation*}
\item The Fisher information $\calD_{\infty}:\calP(\Gamma)\to [0,+\infty]$,
\begin{equation*}
    \calD_{\infty}(\sfP):=\int_{\Gamma} \calD_{MF}(\nu) \, \sfP(\dd \nu).
\end{equation*}
\item The \emph{EDP-functional} $\calI_{\infty}:\mathsf{CE}_{\infty}\to [0,+\infty]$ for all curves with $\calF_{\infty}(\sfP_0)<\infty$,
\begin{equation*}
    \calI_{\infty}(\sfP,\sfJ^+,\sfJ^-):=\int_0^T \calR_{\infty}(\sfP_t,\sfJ_t^+,\sfJ^-_t) \, \dd t + \calF_\infty(\nu_T)-\calF_\infty(\nu_0)+\int_0^T \calD_{\infty}(\sfP_t) \, \dd t.
\end{equation*}
\end{itemize}
\end{defi}

\begin{remark}\label{rem_lifisher}
Recall from Section \ref{s:mf} that $\calF_{MF}(\nu):=\tfrac{1}{2}\Ent(\nu|\gamma)$ and 
\[  \calD_{MF}(\nu):=\left\{\begin{aligned}
    &2H^2(\kappa_{\nu}^+,\kappa_{\nu}^-),&& \qquad \mbox{if $\nu\ll\gamma,$}\\
    &+\infty,&& \qquad \mbox{otherwise.}
    \end{aligned}\right. \]
In particular, if $\calF_{\infty}(\sfP)<\infty$ we have 
\begin{align*}
    \calD_{\infty}(\sfP)&=2  \int_{\Gamma} H^2(\kappa_{\nu}^+,\kappa_{\nu}^-) \, \sfP(\dd \nu)
    =2  H^2(\teta_{\sfP}^+,\teta_{\sfP}^-).
\end{align*}
\end{remark}

\begin{remark}\label{rem:li_irep}
Note that  $\Theta_{\sfP}^ {\infty}(\dd \nu,\dd x)=\sfP(\dd \nu) \theta_{\nu}(\dd x)$. Moreover, if $\Ent(\sfJ_t^{\pm}|\Theta_{\sfP_t}^{\infty})$ is finite, we can set 
\[ \lambda^{\pm}_{t}[\nu](\dd x):=\frac{\dd \sfJ_t^{\pm}}{\dd \Theta_{\sfP_t}^ {\infty}}(\nu,x)\,\theta_{\nu}(\dd x),
\]
and it is straightforward to verify that we have the disintegration
\begin{equation*}
    \sfJ^{\pm}(\dd \nu,\dd x)=\lambda_t^{\pm}[\nu](\dd x)\sfP_t(\dd \nu),
\end{equation*}
and the equivalence
\begin{equation}\label{eq:li_arep}
      \Ent(\sfJ_t^\pm|\Theta^{\infty}_{\sfP_t})=\int_{\Gamma} \Ent(\lambda_{t}^\pm[\nu]|\theta_{\nu})\, \dd \sfP_t.
\end{equation}
Together with the definitions of $\calF_{\infty}$ and $\calD_{\infty}$ this implies that if $\calI_{\infty}(\sfP,\sfJ^+,\sfJ^-)$ is finite then the $\lambda_{t}^{\pm}[\nu]$ are well-defined for a.e. $t\in [0,T]$, and 
\begin{equation*}
   \int_0^T \int_{\Gamma} \left(\calR_{MF}(\nu_t,\lambda^+_{t}[\nu],\lambda^-_{t}[\nu])+\calD_{MF}(\nu)\right)  \sfP_t(\dd \nu) \, \dd t + \int_{\Gamma} \calF_{MF}(\nu) \sfP_T(\dd \nu)-\int_{\Gamma} \calF_{MF}(\nu) \sfP_0(\dd \nu).
\end{equation*}
Throughout the rest of this section we will simply write $\lambda_{t,\nu}^{\pm}=\lambda^\pm_t[\nu]$.
\end{remark}

We will show the following equivalence, which subsumes Theorem \eqref{thm:liouvi}.
\begin{thm}\label{thm:li_equiv}
For any $(\sfP,\sfJ^+,\sfJ^-)\in \mathsf{CE}_{\infty}$ with $\calF_{\infty}(\sfP_0)<\infty$, the EDP-functional $\calI_{\infty}$ is finite if and only if there exists a Borel probability measure $Q$ over $C([0,T];\Gamma)$ such that
\begin{enumerate}
    \item for the time-evaluations $e_t$ we have $(e_t)_{\#}Q=\sfP_t$ for all $t\in [0,T]$,
    \item the measure $Q$ is concentrated on the family of curves $\nu\in AC([0,T];(\Gamma,\|\cdot\|_{TV}))$ such that $(\nu,\lambda^+_{\nu},\lambda^-_{\nu}) \in \mathscr{CE}$, where $\lambda_\nu^\pm$ is defined via the disintegration
    \[  \sfJ_t^{\pm}(\dd \nu,\dd x)=\lambda_{t,\nu}^{\pm}(\dd x)\sfP_t(\dd \nu)\qquad\text{for a.e. $t\in [0,T]$,}  \] 
    \item we have the representation
   \begin{equation*}
     \calI_{\infty}(\sfP,\sfJ^+,\sfJ^-)= \int \calI_{MF}\left(\nu,\lambda^+_{\nu},\lambda^-_{\nu}\right) \dd Q, 
\end{equation*}
 with the latter term finite. 
\end{enumerate} 

In particular, $\calI_{\infty}\geq 0$, and
\[ \calI_{\infty}(\sfP,\sfJ^+,\sfJ^-)=0 \iff \left\{ \begin{aligned}
\quad &\mbox{$\sfP_t$\; is the weak solution to \eqref{eq:liouv2} with $\sfP_t=(G_t)_{\#} \sfP_0$}  \\
\quad \sfJ^{\pm}_t&=\sfP_t \kappa_{\nu}^{\pm} \quad \mbox{for a.e.\ $t\in [0,T]$}\\
\end{aligned} \right. \]
\end{thm}

Here $G_t:\Gamma\to\Gamma$ maps $\bar \nu$ to the unique mean-field solution $\nu_t$ at time $t$, see Remark \ref{rem:li_solsup}. It is determined by
\begin{equation*}
   \partial_t G_t[\nu]=V[G_t[\nu]].
\end{equation*}

We do not have a priori uniqueness of the Liouville equation. However, we do have uniqueness of weak solutions for which a superposition holds, in particular for curves with finite $\calI_{\infty}$. Therefore gradient-flow solutions (null-minimizers of $\calI_{\infty}$) are in fact unique.

In the case of $\sfP_t:=\delta_{\nu_t}$ with $\nu_t$ the solution to the mean-field equation there is a trivial superposition principle, and we have the following consequence. 

\begin{cory}
Suppose $\sfP_0=\delta_{\nu_0}$ with $\calF_{MF}(\nu_0)<\infty$. Then 
\[ \calI_{\infty}(\sfP,\sfJ^+,\sfJ^-)=0 \iff \left\{ \begin{aligned}
\quad &\mbox{$\sfP_t=\delta_{\nu_t}$,\quad $\nu_t$ is the unique strong solution to \eqref{eq:mf}} \quad \\
\quad &\sfJ^{\pm}_t=\sfP_t \kappa_{\nu}^{\pm} \quad \mbox{for a.e.\ $t\in [0,T]$} \quad
\end{aligned} \right. \]
\end{cory}

\subsection{A priori estimates}

Due to the representation \eqref{eq:li_arep} of the dissipation potential in terms of mean-field objects, we can directly derive the following estimates from Lemma's \ref{lm:mf_est} and \ref{lm:mf_lift}. 
\begin{cory}\label{cory:li_est1}
Let $\sfP\in \calP(\Gamma),\sfJ^{\pm}\in \calM^+_{loc}(\Gamma\times \calT)$  be such that $\calR_{\infty}(\sfP,\sfJ^+,\sfJ^-)<\infty$, and set 
    \[ \lnet_{\nu}:=\lambda^{+}_{\nu}-\lambda^{-}_{\nu}. \]
Then the following estimates hold:
\begin{align*}
   \int_{\Gamma} M\phi\left(\frac{\lambda_{\nu}^{\pm}(\calT)}{M(1+\nu(X)^2)}\vee 1\right) \sfP(\dd \nu) &\leq \mathcal{R}_{\infty}(\sfP,\sfJ^+,\sfJ^-),\\
   \int_{\Gamma} M\Psi\left(\frac{\|\lnet_{\nu}\|_{TV}}{M(1+\nu(X))}\right) \sfP(\dd \nu) &\leq \mathcal{R}_{\infty}(\sfP,\sfJ^+,\sfJ^-).
\end{align*}
\end{cory}

Moreover, the following equivalence follows straightforwardly from Lemma \ref{lm:fke_est}.
\begin{cory}\label{cory:li_est2}
For any $\sfP\in \calP(\Gamma),\sfJ^{\pm}\in \calM^+_{loc}(\Gamma\times \calT)$ 
\begin{equation*}
    \Ent(\sfJ^{\pm}|\Theta^{\infty}_{\sfP})=\int_{\Gamma\times \calT} \Upsilon\left(\frac{\dd \sfJ^{\pm}}{\dd \Sigma},\frac{\dd \vartheta_{\sfP}^+}{\dd \Sigma},\frac{\dd \vartheta_{\sfP}^-}{\dd \Sigma}\right)\dd  \Sigma,
\end{equation*}
for any common dominating measure $\Sigma$.
\end{cory}

Finally, we consider the time-regularity for arbitrary curves, with respect to the following metric. 

\begin{defi}\label{defi:Wm} We define the following metric:
\begin{equation}\label{eq:c_uni_m}
    W(\sfP^1,\sfP^2):=\sup_{F\in \mathbb{F}} \left\{ \int_{\Gamma} F \,\dd (\sfP^1-\sfP^2) \right\},\qquad \sfP^1,\sfP^2\in \calP(\Gamma),
\end{equation}
where 
\begin{equation*}
    \mathbb{F}:=\left\{ F\in \Cyl_c(\Gamma)\, : \,  \left\|(1+\nu(\calT)^2) \mathrm{grad}_{\Gamma}F\right\|_{\infty}\leq 1 \right\}.
\end{equation*}
\end{defi}

Note that $W$ is narrowly lower semicontinuous. Moreover, for any $F\in \Cyl_c(\Gamma)$ automatically $\left\|(1+\nu(\calT)^2) \mathrm{grad}_{\Gamma}\,F\right\|_{\infty}<\infty$, and hence by a density argument it is straightforward to verify that convergence in $W$ implies vague convergence on $\Gamma$, and therefore narrow convergence on narrowly pre-compact subsets. 

\begin{remark}
Formally, one can represent $W$ as a transport distance, in the sense that 
\begin{equation*}
    W(\sfP^1,\sfP^2)=W_{d_{\Gamma}}(\sfP^1,\sfP^2),
\end{equation*}
where $W_{d_{\Gamma}}$ is the 1-Wasserstein metric on $\calP(\Gamma)$ induced by the metric $d_{\Gamma}$ over $\Gamma$ given by
\begin{equation*}
    d_{\Gamma}(\nu^1,\nu^2):=\inf_{(\nu_t)_{t\in [0,1]}} \left\{ \int_0^1 \frac{|\dot \nu_t|_{TV}}{1+\nu_t(\calT)^2}\, \dd t\, : \, \nu_0=\nu^0, \, \nu_1=\nu^2\right\}.
\end{equation*}
However, we do not require such representations in this current work.
\end{remark}

\begin{lm}\label{lm:liouv_timereg}
For any $(\sfP,\sfJ^+,\sfJ^-)\in \mathsf{CE}_{\infty}$ we have 
\begin{equation*}
    W(\sfP_s,\sfP_t)\leq 2 \int_s^t \int_{\Gamma\times \calT} (1+\nu(\calT)^2)^{-1} \dd (\sfJ_r^{+}+\sfJ^-_r)\, \dd r, \qquad \fA s,t\in [0,T]. 
\end{equation*}
\end{lm}

\begin{proof}
This follows directly from the continuity equation, since for any $F\in \mathbb{F}$, $s,t\in [0,T]$: 
\begin{align*}
\left|\int_{\Gamma} F(\nu) \,\dd \sfP_t - \int_{\Gamma} F(\nu) \,\dd \sfP_s \right|&\leq  \int_s^t \int_{\Gamma\times \calT} \left|(1+\nu(\calT)^2) \mathrm{grad}_{\Gamma}F \right| (1+\nu(\calT)^2)^{-1}\,\dd (\sfJ_r^++\dd \sfJ_r^-) \, \dd r\\
&\leq \int_s^t \int_{\Gamma\times \calT} (1+\nu(\calT)^2)^{-1}\,\dd (\sfJ_r^++\dd \sfJ_r^-) \, \dd r.
\end{align*}
Taking the supremum over all $F\in \mathbb{F}$ we obtain the desired statement. 
\end{proof}

\subsection{Weak solutions}\label{ss:li_sol} Here we briefly consider existence and representations for solutions to the Liouville equation. 
\begin{lm}
For any $\bar \sfP_t\in \calP(\Gamma)$ there exists a solution $\sfP$ to \eqref{eq:liouv2} with initial data $\bar \sfP$. 
\end{lm}
\begin{proof}
Recall the flow $G:[0,T]\times \Gamma\to \Gamma$ determined by
\begin{equation*}
   \partial_t G_t[\nu]=V[G_t[\nu]],
\end{equation*}
Set $\sfP_t:=(G_t)_{\#} \bar \sfP$. We will show that $\sfP_t$ is weak solution in the sense of \eqref{eq:liouv3}. Namely, consider any $F\in \Cyl_c(\Gamma)$. Due the strong regularity of solutions to the mean-field equation it is straightforward to show that for all $s,t\in [0,T]$ we have the chain rule 
\begin{equation*}
    F\circ G_t(\nu)-F(\nu)=\int_s^t (\mathrm{grad}_{\Gamma} F)(G_r \circ \nu,x) \, \dd V[G_r\circ \nu] \, \dd r,
\end{equation*}
and hence 
\begin{align*}
    \int_{\Gamma} F \dd \sfP_t-\int_{\Gamma} F \dd \sfP_s &= \left(\int_s^t (\mathrm{grad}_{\Gamma} F)(G_r \circ \nu,x) \, V[G_r[\nu]](\dd x)\, \dd t\right) \bar \sfP(\dd \nu)\\
    &=\int_s^t \int_{\Gamma\times \calT}(\mathrm{grad}_{\Gamma} F)(\nu,x) V[\nu](\dd x) \sfP_r(\dd \nu) \, \dd t,
\end{align*}
and thus $\sfP_t$ is indeed a weak solution. 
\end{proof}

\subsection{Superposition principle}\label{ss:super}

One of our main tools in proving the chain rule, uniqueness of solutions, and the variational representation of Theorem \ref{thm:li_equiv} is the superposition principle. It guarantees that we can represent the action as an expectation of the mean-field action under some measure over curves in $\mathscr{CE}$, and allows us to use the theory on mean-field dynamics of Section \ref{s:mf}. In this section, we will make this notion precise. 

\begin{thm}\label{thm:lsuper}
Let $(\sfP,\sfJ^+,\sfJ^-)\in \mathsf{CE}_{\infty}$ with
\begin{equation*}
    \int_0^T \mathcal{R}_{\infty}(\sfP_t,\sfJ^+_t,\sfJ^-_t) \, \dd t < \infty.
\end{equation*}
Then there exists a Borel probability measure $Q\in\calP(C([0,T];\Gamma))$ satisfying $(e_t)_{\#}Q=\sfP_t$ for all $t\in [0,T]$, and concentrated on curves $\nu\in AC([0,T];(\Gamma,\|\cdot\|_{TV}))$, for which $(\nu,\lambda^+_{\nu},\lambda^-_{\nu}) \in \mathscr{CE}$.
Moreover, 
\begin{equation}\label{eq:liouvsequiv1}
     \int_0^T  \mathcal{R}_{\infty}(\sfP_t,\sfJ^+_t,\sfJ^-_t)\, \dd t = \int_{C([0,T];\Gamma)} \left( \int_0^T \calR_{MF}\left(\nu_t,\lambda^+_{t,\nu},\lambda^-_{t,\nu}\right) \, \dd t \right)Q(\dd \nu).
\end{equation}

Conversely, if there is a Borel probability measure $Q\in\calP(C([0,T];\Gamma))$ concentrated on curves $\nu\in AC([0,T];(\Gamma,\|\cdot\|_{TV}))$ and a Borel family $\{\lambda^{\pm}_{t,\nu}\}$, for which $(\nu,\lambda^+_{\nu},\lambda^-_{\nu})\in \mathscr{CE}$, with
\begin{equation*}
    \int_{C([0,T];\Gamma)} \left( \int_0^T \calR_{MF}\left(\nu_t,\lambda^+_{t,\nu},\lambda^-_{t,\nu}\right) \, \dd t \right)Q(\dd \nu)<\infty,
\end{equation*}
then $(\sfP,\sfJ^+,\sfJ^-)\in \mathscr{CE}$ for $\sfP_t:=(e_t)_{\#}Q$, $\sfJ_t^{\pm}:=\sfP_t \lambda^{\pm}_{t,\nu}$, and \eqref{eq:liouvsequiv1} holds as well.
\end{thm}

The inspiration for using a superposition principle stems from similar approaches in \cite{Erbar2016}, \cite{Erbar2016a}, where it is applied to transport equations lifted from the Boltzmann-equation or mean-field jump dynamics respectively, and the main ingredient is the abstract superposition principle over $\R^\N$ of \cite{Ambrosio2014}. However, these results are not directly applicable to our setting, since the mass of $\nu_t(\calT)$ for a mean-field curve is not fixed, and $V[\nu](\calT)$ is finite but unbounded over $\Gamma$. We remedy this by combining two known superposition principles: on the one hand, the abstract superposition principle over $\R^\N$ of \cite{Ambrosio2014}, and on the other hand one for finite-dimensional vector fields with linear growth, found in \cite{ambrosio2008}. Our result is stated in Theorem \ref{thm_super}.

\begin{proof}
Consider any $(\sfP,\sfJ^+,\sfJ^-)\in \mathsf{CE}_{\infty}$ with finite $\mathcal{R}_{\infty}$, and for a.e. $t\in [0,T]$ set $\lnet_{t,\nu}:=\lambda_t^+-\lambda^-_t$. By Corollary \ref{cory:li_est1},
\begin{equation}\label{eq:lsuper1}
   \int_{\Gamma} M\Psi\left(\frac{\|\lnet_{t,\nu}\|_{TV}}{M(1+\nu(\calT))}\right) \sfP_t(\dd \nu) \leq \mathcal{R}_{\infty}(\sfP_t,\sfJ^+_t,J_t^-).
\end{equation}

Now, take a countable and dense set $f_1,f_2,\ldots\in C_b(\calT)$, with $f_1=1$, $\|f_i\|_{\infty}\leq 1$, $i\ge 2$, and define $\mathbb{T}:\Gamma\to \R^{\N}$
\begin{equation*}
    \mathbb{T}(\nu):=\left(\int_{\calT} f_1 \,\dd \nu, \int_{\calT} f_2 \,\dd \nu \ldots\right).
\end{equation*}
Note that $\mathbb{T}(\nu)$ is injective, continuous when $\Gamma$ is equipped with the narrow topology and $\R^{\N}$ with product topology, and is an isometry between $(\Gamma,\|\cdot\|_{TV})$ and $(\mathbb{T}(\Gamma),|\cdot|_{\infty})$, where $|\cdot|_{\infty}$ is the uniform norm over $\R^N$. We set $\sigma_t:=\mathbb{T}_{\#}\sfP_t \in \calP(\R^{\N})$, and for a.e. $t\in [0,T]$ define the vector field $\mathbf{W}_t:\R^n\to \R^n$ via its components
\begin{equation*}
   W_i(t,z) :=\int_{X}  f_i(x) \,\lambda_{t,\mathbb{T}^{-1}(z)}(\dd x).
 \end{equation*}
Note that the support of $\mathbf{W}_t$ is in  $\mathbb{T}(\Gamma)$, that $|\mathbf{W}_t(z)|_{\infty}\leq \|\lambda_{t,\mathbb{T}^{-1}(z)}\|_{TV}$ and $(\mathbb{T}(\nu))_1=\nu(\calT)$. Therefore,  
by \eqref{eq:lsuper1} we have the estimate
\begin{equation*}
   \int_0^T \int_{\R^{\N}} M \Psi\left(\frac{|\mathbf{W}_t(z)|_{\infty}}{M(1+|z_1|)}\right)\sigma(\dd z)\, \dd t \leq \int_0^T \mathcal{R}_{\infty}(\sfP_t,\sfJ^+_t,\sfJ^-_t)\, \dd t < \infty. 
\end{equation*}

Moreover, $(\sigma,\mathbf{W})$ satisfy the continuity equation, in the sense that for all $g\in\Cyl_c(\R^\N)$, we have
\begin{equation*}
    \int_{\R^\N} g\, \dd \sigma_t - \int_{\R^\N} g\, \dd \sigma_s = \int_s^t\int_{\R^\N} (\mathbf{W}_r,\nabla g)\, \dd \sigma_r\,\dd r\qquad\text{for every $s,t \in [0,T]$.}
\end{equation*}
Indeed, take any $g\in\Cyl_c(\R^\N)$ and define $F:=g\circ \mathbb{T}$, i.e.\ 
\begin{equation*}
    F(\nu)=g\left(\langle f_1,\nu\rangle,\dots,\langle f_m,\nu\rangle \right).
\end{equation*}
Note that $F\in \Cyl_c(\Gamma)$, and therefore since $(\sfP,\sfJ^+,\sfJ^-)\in \mathsf{CE}_{\infty}$, 
\begin{align*}
   \int_{\R^{\N}} g(z)\,  \sigma_t(\dd z)-\int_{\R^{\N}} g(z)\,  \sigma_s(\dd z) &=\int_{\Gamma} F \dd \sfP_t-\int_{\Gamma} F \dd \sfP_s \\
     &= \int_s^t \int_{\Gamma\times \calT} (\mathrm{grad}_{\Gamma}\,F)(\nu,x)(\sfJ^+_r-\sfJ^-_r)(\dd \nu,\dd x)\, \dd r\\
     &=\int_s^t  \int_{\Gamma} \sum_{i} (\partial_i g)(\mathbb{T}(\nu)) \left(\int_{\calT} f_i(x) \lambda_{r,\nu}(\dd x)\right) \sfP_r(\dd \nu)\, \dd r\\
    &=\int_s^t \int_{\R^{\N}} \nabla g(z) \cdot \mathbf{W}_r(z)\,   \sigma_r(\dd z) \, \dd r. 
\end{align*}

Thus, we are now in a position to apply Theorem \ref{thm_super}, and obtain a Borel probability measure $\Omega$ over $C([0,T];\R^{\N})$ satisfying $(e_t)_{\#} \Omega=\sigma_t$ for all $t\in[0,T]$, and which is concentrated on the family of curves $z\in AC([0,T];\R^{\N})$ that are solutions to the ODE
\[
    \dot z_t=\mathbf{W}_t(z_t)\qquad \text{for almost every $t\in[0,T]$.}
\]
Note that since $\mathrm{supp}(\sigma)\subseteq \mathbb{T}(\Gamma)$, we have $\mathrm{supp}(\Omega)\subseteq AC([0,T];\mathbb{T}(\Gamma))$. Now let
$\tilde{\mathbb{T}}:C([0,T];\Gamma)\to C([0,T];\R^{\N})$ be defined via $(\tilde  {\mathbb{T}}(\nu))_t:=\mathbb{T}(\nu_t)$. Similar as for $\mathbb{T}$, $\tilde{\mathbb{T}}$ is injective and an isometry when seen as a map $\tilde{\mathbb{T}}:AC([0,T];(\Gamma,\|\cdot\|_{TV}))\to AC([0,T];(\R^{\N},|\cdot|_{\infty}))$. Therefore, it is clear the measure $Q:=\tilde{\mathbb{T}}^{-1}_{\#} \Omega \in \calP(C([0,T];\Gamma))$ is well defined, satisfies $\sfP_t=(e_t)_{\#}Q$ and is concentrated on the family of curves $\nu\in AC([0,T];(\Gamma,\|\cdot\|_{TV}))$, for which 
\[
    \int_{\calT} f_i \, \dd \nu_t - \int_{\calT} f_i \, \dd \nu_s = \int_s^t f_i \, \dd (\lambda^+_{r,\nu}-\lambda^+_{r,\nu}) \, \dd r \qquad \mbox{for all $s,t\in [0,T]$, $i\in \N$}.
\]
Moreover, 
\begin{align*}
     \int_{C([0,T];\Gamma)} \left( \int_0^T \calR_{MF}\left(\nu_t,\lambda^+_{t,\nu},\lambda^-_{t,\nu}\right) \, \dd t \right)Q(\dd \nu)&=\int_0^T \int_{\Gamma} \calR_{MF}(\nu,\lambda^+_{t,\nu},\lambda^-_{t,\nu})\, \sfP_t(\dd \nu) \, \dd t\\
     &=\int_0^T \mathcal{R}_{\infty}(\sfP_t,\sfJ^+_t,\sfJ^-_t) \, \dd t,
\end{align*}
where the latter is finite by assumption, and hence, by Lemma \ref{lm:contweakf1}, we deduce that $(\nu,\lambda_{\nu}^{+},\lambda_{\nu}^{+})\in \mathscr{CE}$ $Q$-almost everywhere.

The reverse statement can be derived straightforwardly and we omit the proof. 
\end{proof}

\subsection{Variational characterization}

Having all the ingredients at hand, we can now prove the variational characterization for the Liouville equation, namely Theorem~\ref{thm:li_equiv}.

\begin{proof}[Proof of Theorem~\ref{thm:li_equiv}]
Suppose $(\sfP,\sfJ^+,\sfJ^-)$ is such that $\calF_{\infty}(\sfP_0)<\infty$ and $\mathcal{I}_{\infty}<\infty$. Since $\calF_{\infty}$ is non-negative we have in particular that 
\begin{equation*}
    \int_0^T \calR_{\infty}(\sfP_t,J_t^+,J_t^-)\, \dd t <\infty, \quad \calF_{\infty}(\sfP_T)<\infty, \quad  \int_0^T \calD_{\infty}(\sfP_t)\, \dd t <\infty.
\end{equation*}
Hence, from the superposition principle of Theorem \ref{thm:lsuper}, we obtain a Borel probability measure $Q$ over $C([0,T];\Gamma)$ satisfying $(e_t)_{\#}Q=\sfP_t$ for all $t\in [0,T]$ and concentrated on the family of curves $\nu\in AC([0,T];(\Gamma,\|\cdot\|_{TV}))$ for which $(\nu,\lambda^+_{\nu},\lambda^-_{\nu})\in \mathscr{CE}$. Moreover,
\begin{equation*}
      \int_{C([0,T];\Gamma)} \left( \int_0^T \calR_{MF}\left(\nu_t,\lambda^+_{t,\nu},\lambda^-_{t,\nu}\right) \dd t \right)Q(\dd \nu)=\int_0^T  \calR_{\infty}(\sfP_t,\sfJ^+_t,\sfJ^-_t) \, \dd t < \infty. 
\end{equation*}
Since $\calF_{\infty}(\sfP_0)<\infty$ we have that for $Q$-a.e.\ curve $\calF_{MF}(\nu_0)<\infty$. Moreover, since both $\calF_{\infty}$ and $\calD_{\infty}$ are simply their mean-field counterparts integrated by $\sfP$, we find
\begin{align*}
 &\int_{C([0,T];\Gamma)} \calI_{MF}\left(\nu,\lambda^+_{\nu},\lambda_{\nu}^-\right) Q(\dd \nu) \\
 &\, =\int_{C([0,T];\Gamma)} \left(\int_0^T \calR_{MF}\left(\nu_t,\lambda^+_{t,\nu},\lambda^-_{t,\nu}\right) \dd t+\calF_{MF}(\nu_t)-\calF_{MF}(\nu_0)+\int_0^T \calD_{MF}(\nu_t) \, \dd t\right)Q(\dd \nu) \\
 &\, =\int_{C([0,T];\Gamma)} \left(\int_0^T \calR_{MF}\left(\nu_t,\lambda^+_{t,\nu},\lambda^-_{t,\nu}\right) \dd t \right)Q(\dd \nu) + \calF_{\infty}(\sfP_T)-\calF_{\infty}(\sfP_0)+\int_0^T \mathcal{D}_{\infty}(\sfP_t)\, \dd t\\
 &\, = \mathcal{I}_{\infty}(\sfP,\sfJ^+,\sfJ^-),
\end{align*}
where the second equality follows from Fubini-Tonelli and the fact that $\calR_{MF},\calD_{MF},\calF_{MF}\geq 0$ and $\calF_{\infty}(\sfP_0)<\infty$. In particular, by the non-negativeness of $\calI_{MF}$ it holds that $\mathcal{I}_{\infty}\geq 0$. 

Moverover, since $\calI_{MF}=0$ if and only if $\nu$ is the unique strong solution for an initial datum $\bar \nu$ with $\Ent(\bar \nu|\gamma)<\infty$, we derive by non-negativeness of $\calI_{MF}$ that $\mathcal{I}_{\infty}=0$ if and only if $Q$ is concentrated on the unique solutions of the mean-field equation. In this case $Q$ is characterized by
\begin{equation*}
    Q=\tilde G_{\#} \sfP_0, 
\end{equation*}
where $G_t:\Gamma\to\Gamma$ defined by \eqref{eq:li_sol} maps any $\bar \nu$ to the unique solution to \eqref{eq:mf1} for initial condition $\nu_0=\bar \nu$ and $\tilde G:\Gamma\to C([0,T],\Gamma)$ is defined via $(\tilde G(\nu_0))_t:=G_t(\nu_0)$. Note that $\sfP_t=(G_t)_{\#}\sfP_0$, $\sfJ_t^{\pm}=\sfP_t \kappa_{\nu}^{\pm}$ for almost every $t\in[0,T]$, and in particular $\sfP_t$ is a weak solution to \eqref{eq:liouv2}.

Vice versa, if $\sfP$ is a weak solution such that $\sfP_t=(G_{t})_{\#} \sfP_0$, we simply set 
\begin{equation*}
    Q:=\tilde G_{\#} \sfP_0, \qquad \lambda_{\nu}^{\pm}:=\kappa_{\nu}^{\pm} \quad\fA t\in [0,T].
\end{equation*}
Since $\calF_{\infty}(\sfP_0)<\infty$, we still have $\Ent(\nu|\gamma)<\infty$ for $\sfP_0$-almost every $\nu$, and we repeat the same calculations to conclude that indeed $\calI_{\infty}=0$.  
\end{proof}

 \section{EDP convergence}\label{s:edp}

In the previous sections, we have established variational formulations for the solution to the forward Kolmogorov equation of the interacting particle system, for the solutions to the mean-field equation, and the corresponding Liouville equation. Moreover, for the latter, we have shown how the corresponding EDP-functional can be represented as the expectation over a functional of mean-field paths. 

We are now in a position to rigorously discuss the convergence of the forward Kolmogorov equation to the Liouville equation, in terms of EDP-convergence of their gradient structures.
Namely, let us denote a sequence of curves  $(\sfP^n,\sfJ^{n,+},\sfJ^{n,-})\in \mathsf{CE}_n$ converging to a curve $(\sfP,\sfJ^+,\sfJ^-)$, denoted by $\lim_{n\to\infty}(\sfP^n,\sfJ^{n,+},\sfJ^{n,-}) = (\sfP,\sfJ^+,\sfJ^-)$, if the following holds:
\begin{itemize}
    \item $\sfP_t^n \to \sfP_t$ narrowly for all $t\in [0,T]$,
    \item $\sfJ_t^{n,\pm}(\dd \nu,\dd x) \, \dd t \to \sfJ_t^{\pm}(\dd \nu,\dd x) \, \dd t$ vaguely on $\calM^+_{loc}(\Gamma\times \calT\times [0,T])$.
\end{itemize}

\begin{thm}\label{thm:main_conv1}
Suppose that a sequence $(\sfP^n,\sfJ^{n,+},\sfJ^{n,-})\in \mathsf{CE}_n$, $n\ge 1$, is such that 
\begin{equation*}
    \limsup_{n\to \infty} \calF_n(\sfP_0^n)<\infty, \qquad  \limsup_{n\to \infty} \calI_n(\sfP^n,\sfJ^{n,+},\sfJ^{n,-})<\infty,
\end{equation*}
then the family of curves $\{(\sfP_t)_{t\in [0,T]}\}_{n}$ is W-equicontinuous \eqref{eq:c_uni_m}, and there exists a (not relabelled)  subsequence $(\sfP^{n},\sfJ^{n,+},\sfJ^{n,-})$ and a $(\sfP,\sfJ^+,\sfJ^-)\in \mathsf{CE}_{\infty}$ such that
\[ \lim_{n\to \infty} (\sfP^{n},\sfJ^{n,+},\sfJ^{n,-}) = (\sfP,\sfJ^+,\sfJ^-), \]
Moreover, for any such converging sequence
\begin{equation}\label{main_conv1_linf}
    \begin{aligned}
    \liminf_{n\to \infty} \calF_{n}(\sfP^{n}_t) &\geq \calF_{\infty}(\sfP_t),\qquad \fA t\in [0,T],\\
    \liminf_{n\to \infty} \int_0^T \calR_{n}(\sfP_t^{n},\sfJ^{n,+}_t,\sfJ^{n,-}_t)\, \dd t &\geq \int_0^T \calR_{\infty}(\sfP_t,\sfJ^{+}_t,\sfJ^{-}_t)\, \dd t, \\
     \liminf_{n\to \infty} \int_0^T \calD_{n}(\sfP_t^{n}) \, \dd t &\geq \int_0^T \calD_{\infty}(\sfP_t^{n})\, \dd t.
\end{aligned}
\end{equation}
\end{thm}

\begin{remark}
In fact, the compactness result is slightly stronger. As shown in the proof of Theorem \ref{thm:main_conv1} the measures $\sfJ_r^{n,\pm}(\dd \nu,\dd x) \, \dd r$ converge vaguely on $\calM^+_{loc}(\Gamma\times \calT\times [s,t])$ for any $s,t\in [0,T]$.
\end{remark}

Note that if in addition the initial data is well-prepared, in the sense that 
\[ \lim_{n\to \infty} \calF_n(\sfP_0^n)=\calF_{\infty}(\sfP_0), \] 
then for any converging subsequence, we clearly have the liminf-estimate
\begin{equation}\label{eq:edp_ge1}
    \liminf_{n\to \infty} \calI_{n}(\sfP^{n},\sfJ^{n,+},\sfJ^{n,-})\geq  \calI_{\infty}(\sfP,\sfJ^{+},\sfJ^{-}),
\end{equation}
or in other words, obtain evolutionary $\varGamma$-convergence of $\calI_n$ to $\calI_{\infty}$. 

\medskip

Now, recall by Theorem \ref{thm:fke_main} that unique gradient-flow solutions to the forward Kolmogorov equations \eqref{eq:FKEn} exist, and similarly, gradient-flow solutions to the Liouville equation \eqref{eq:liouv2} are unique by Theorem \ref{thm:li_equiv}. Therefore, modifying classical arguments from \cite{Sandier2004,Serfaty2011}, we can directly conclude the following convergence for the sequence of solutions. 

\begin{thm}\label{lm:main_conv2}
Consider a converging sequence $\calP(\Gamma_n) \ni \bar \sfP^n\to \bar \sfP \in \calP(\Gamma)$ such that 
\begin{equation}\label{eq:main_conv2_fn}
    \lim_{n\to \infty} \calF_n(\bar \sfP^n)=\calF_{\infty}(\bar \sfP), 
\end{equation}
and for each $n\ge 0$ let $\sfP_t^n$ be the unique gradient-flow solution to $\eqref{eq:FKEn}$ with initial data $\bar \sfP^n$. Then there exists a unique gradient-flow solution $\sfP$ to \eqref{eq:liouv2} with initial data $\bar \sfP$. Moreover, we have the convergence
\begin{equation*}
\begin{aligned}
    \lim_{n\to \infty} (\sfP^n,\vartheta_{\sfP^n}^{+},\vartheta_{\sfP^n}^{-}) &=(\sfP,\vartheta_{\sfP}^{+},\vartheta_{\sfP}^{-})\\
    \lim_{n\to \infty} \calF_n(\sfP^n_t)&=\calF_{\infty}(\sfP_t), \qquad \fA t\in [0,T].
\end{aligned}
\end{equation*}
\end{thm}
\begin{proof}
Recall that $\calI_n(\sfP^n,\vartheta_{\sfP^n}^{+},\vartheta_{\sfP^n}^{-})=0$ for all $n\ge 0$. Therefore, by \eqref{eq:main_conv2_fn} and Theorem \ref{thm:main_conv1} we have for any subsequence indexed by $n'$ converging to a $(\sfP,\sfJ^{+},\sfJ^-)\in \mathsf{CE}_{\infty}$ that \eqref{eq:edp_ge1} holds, and hence 
\begin{align*}
   0 = \liminf_{n'\to \infty} \calI_n(\sfP^{n'},\vartheta_{\sfP^{n'}}^{+},\vartheta_{\sfP^{n'}}^{-}) \geq \calI_{\infty}(\sfP,\sfJ^{+},\sfJ^{-}),
\end{align*}
and thus $\calI_{\infty}(\sfP,\sfJ^{+},\sfJ^{-})=0$, which implies that $\sfP$ is the unique gradient-flow solution to \eqref{eq:liouv2} and $\sfJ_t^{\pm}=\vartheta_{\sfP_t}^{\pm}$ for a.e. $t\in [0,T]$. The convergence of $\sfP_t^n$ now follows from a compactness and equicontinuity argument, and by lower semicontinuity we conclude that for every $t\in [0,T]$
\begin{align*}
    \limsup_{n\to \infty} \calF_n(\sfP_t^n)&= \lim_{n\to \infty}\calF_n(\sfP_0^n)-\liminf_{n\to \infty} \int_0^t \left(\calR_n(\sfP^n,\vartheta_{\sfP^n}^{+},\vartheta_{\sfP^n}^{-})+\calD_n(\sfP_t^n)\right)\, \dd t\\
    &= \calF_{\infty}(\sfP_0)-\int_0^t \left(\calR_{\infty}(\sfP,\vartheta_{\sfP}^{+},\vartheta_{\sfP}^{-})+\calD_{\infty}(\sfP)\right)\, \dd t = \calF_{\infty}(\sfP_t).
\end{align*}
\end{proof}

Now suppose that in addition the initial sequence of measures $\bar \sfP^n$ is chaotic, in the sense that 
\[\bar \sfP^n \to \delta_{\bar \nu}\quad\text{narrowly for some $\bar \nu \in \Gamma$.}\]
Then as a consequence of Theorem \ref{lm:main_conv2} we have propagation of chaos, namely 
\[\bar \sfP^n \to \delta_{\bar \nu_t} \quad \text{narrowly}\fA t\in [0,T],\]
 where $\nu_t$ is the unique solution to the mean-field equation \eqref{eq:mf2} with initial datum $\bar \nu$. As mentioned in the introduction, while for interacting particle systems with the number of particles fixed at $n\in\N$ this would imply narrow convergence of the $k$-marginals at time $t$ to $\nu_t^{\otimes k}$ (e.g.\ see \cite{Sznitman1991}), in our setting this implies convergence of the $k$-correlation functions \cite{BGSS2020}. 
 
 Moreover, note that we have a stronger notion of convergence, since the free energies $\calF_n$ converge as well. Under appropriate conditions on the initial datum $\bar \nu$, this guarantees a version of propagation of entropic chaoticity. Namely, for any $\nu$ we define the rescaled Poisson measures
\begin{equation*}
    \Pi_{n,\nu}:=(L_n)_{\#} \pi_{n,\nu},\qquad\text{where}\qquad\pi_{n,\nu}:=\frac{1}{e^{n \nu(\calT)}-1}\sum_{N=1}^{\infty} \frac{n^N}{N!}\nu^{\otimes N}.    
\end{equation*}
It is straightforward to check that $\Pi_{n,\nu^*}\to \delta_{\nu^*}$ narrowly. We then have the following result. 

\begin{thm}[Propagation of chaos]\label{thm:edp_prop}
Consider the setting of Theorem \ref{lm:main_conv2} and assume additionally that $\bar \sfP=\delta_{\bar \nu}$ for some $\bar \nu\in \Gamma$ with $\Ent(\bar \nu|\gamma)<\infty$. Let $\nu_t$ be the unique solution to \eqref{eq:mf2} with initial datum $\bar \nu$. Then for all $t\in [0,T]$,
\begin{equation*}
    \begin{aligned}
        \sfP^n_t \to \delta_{\nu_t}\quad\text{narrowly},\qquad\text{and}\qquad
        \lim_{n\to \infty} \Ent(\sfP_t^n|\Pi_n)&=\Ent(\nu_t|\gamma). 
    \end{aligned}
\end{equation*}
If additionally there exists a constant $C>1$ such that $C^{-1}\leq \dd \bar \nu/\dd \gamma \leq C$ then 
\[ \lim_{n\to \infty} \Ent(\sfP^n_t|\Pi_{n,\nu_t})=0, \qquad \fA t\in [0,T].\]
\end{thm}

Theorems \ref{thm:main_conv1} and \ref{thm:edp_prop} are proved in Section \ref{ss:c_main}. However, first we show $\Gamma$-convergence of the free energies in Section \ref{ss:c_gammas}, and establish the necessary estimates in Section \ref{ss:c_uni}.

\subsection{\texorpdfstring{$\varGamma$}{Gamma}-convergence of \texorpdfstring{$\calF_n$}{free energies}}\label{ss:c_gammas}

While only the liminf-estimates for the free energy $\calF_n$ are necessary for the proof of Theorem \ref{thm:main_conv1} and the convergence of solutions, we provide here the full $\varGamma$-convergence result. We rely strongly on the characterization of \cite{Mariani2012}, which connects a large deviation principle with rate function $I$ to the fact that 
\begin{equation*}
    \mathop{\varGamma\text{-lim}}_{n\to \infty} \frac{1}{n}\Ent(\sfP|\Pi^n) = \int_{\Gamma} I(\nu) \sfP(\dd \nu),
\end{equation*}
and provides useful sufficient conditions for both.

Recall in our setting that
\[
    \calF_n(\sfP)=\frac{1}{2n}\Ent(\sfP|\Pi_n),\qquad \calF_{\infty}=\frac{1}{2}\int_{\Gamma} \Ent(\nu|\gamma).
\] We then have the following result, which we prove after Lemma \ref{lm:gammas_Gc} below.

\begin{thm}\label{thm:c_gamma}
The family $\{\calF_n\}_{n\ge 1}$ is equicoercive and $\varGamma$-converges to $\calF$ in the sense that
\begin{itemize}
    \item for any converging sequence $\sfP^n\to \sfP\in \calP(\Gamma)$:  
    \begin{equation*}
        \calF_{\infty}(\nu)\leq \liminf_{n\to \infty} \calF_n(\sfP^n),
    \end{equation*}
    \item for any $\sfP\in \calP(\Gamma)$ with $\calF_{\infty}(\sfP)<\infty$ there exists a sequence $\sfP^n\in \Gamma$ converging to $\sfP$ such that 
    \begin{equation*}
         \lim_{n\to \infty} \calF_n(\sfP^n) = \calF_{\infty}(\sfP).
    \end{equation*}
\end{itemize}
\end{thm}

By the results of \cite[Theorems 3.4, 3.5]{Mariani2012} it is sufficient to merely show the corresponding bounds or limits for any $\sfP$ of the form $\sfP=\delta_{\nu}$ for some $\nu\in \Gamma$. Because of this reduction, we can make use of the so-called cumulant generating functionals $G_n$ given by 
\begin{equation*}
    G_n(f):=\frac{1}{n}\log \int_{\Gamma} e^{n \langle f,\nu\rangle}\, \Pi_n(\dd \nu),
\end{equation*}
for any $f\in \calB_b(\Gamma)$, and their limit counterpart
\begin{equation*}
    G(f):=\int_{\calT} (e^f-1)\, \dd \gamma. 
\end{equation*}
Note that by duality of the entropy, we have for all $n> 0$ the inequality
\begin{equation}\label{eq:gammas_dual}
    \int_{\Gamma} \langle f,\nu\rangle \, \dd \sfP \leq \frac{1}{n}\Ent(\sfP|\Pi_n)+G_n(f), \quad 
\end{equation}
and for the Legendre-dual of $G$ we have 
\[G^*(\nu):=\sup_{f\in \calC_b(\calT)} \bigl\{\langle f,\nu\rangle - G(f)\bigr\}=\Ent(\gamma|\nu).\]
We will first simplify $G_n$ and show that it indeed converges to $G$. 
\begin{lm}\label{lm:gammas_Gc}
Let $f\in \calB_b(\calT)$. Then for each $n>0$ 
\begin{equation*}
     G_n(f)=\tfrac{1}{n}\log \frac{e^{n \int_{\calT} e^f \dd \gamma}-1}{e^{n \gamma(\calT)}-1}.
\end{equation*}
In particular
\begin{equation*}
    \lim_{n\to \infty} G_n(f)=G(f).
\end{equation*}
\end{lm}

\begin{proof}
Using the representation for the rescaled Poisson measure $\Pi_n$ we have 
\begin{equation*}
    \begin{aligned}
      \int_{\Gamma} e^{n \langle f,\nu\rangle} \Pi_n(\dd \nu)&=\frac{1}{e^{n \gamma(\calT)}-1} \sum_{i=1}^N  \frac{n^N}{N!}\int_{\calT^N} e^{\sum_{i=1}^N f(x_i)} \dd \gamma^{\otimes N}\\
      &=\frac{1}{e^{n \gamma(\calT)}-1} \sum_{i=1}^N  \frac{n^N\left(\int_{\calT} e^f \dd \gamma\right)^n}{N!}
      =\frac{e^{n \int_{\calT} e^f \dd \gamma}-1}{e^{n \gamma(\calT)}-1},
    \end{aligned}
\end{equation*}
and after taking logarithms and dividing by $n$ we obtain the desired statement. Moreover, recall that by assumption $\gamma(\calT)>0$ and note that by the boundedness of $f$, 
\[ 0 < \int_{\calT} e^f \dd \gamma < \infty. \]
Hence we can take limit $n\to\infty$ to deduce
\begin{align*}
    \lim_{n\to \infty} G_n(f)&=\lim_{n\to \infty} \frac{1}{n}\log \left(e^{n \int_{\calT} e^f \dd \gamma}-1\right)-\frac{1}{n}\log\left(e^{n \gamma(\calT)}-1\right)\\
    &=\int_{\calT} e^f \dd \gamma-\gamma(\calT)=G(f),
\end{align*}
thereby concluding the proof.
\end{proof}

\begin{proof}[Proof of Theorem \ref{thm:c_gamma}]
First, we will show that the family $\{\calF_n\}_{n\ge 1}$ is equicoercive, by establishing a first moment bound for $\sfP$ in terms of mass $\nu(\calT)$. Namely, setting $f=1$ in \eqref{eq:gammas_dual} we have for any $\sfP\in \calP(\Gamma)$, $n\ge 1$, the inequality 
\begin{align*}
    \int_{\Gamma} \nu(\calT)\, \dd \sfP &\leq \frac{1}{n}\Ent(\sfP|\Pi_n)+G_n(1)
    \leq 2\calF_n(\sfP)+\frac{1}{n}\log \frac{e^{n e \gamma(\calT) }-1}{e^{n \gamma(\calT)}-1},
\end{align*} 
where the final term is bounded from above independently of $\sfP$. 

\medskip

Next, for the limit inferior, consider a converging sequence $\sfP^n\to \sfP=\delta_{\bar \nu}$ for some $\bar \nu\in \Gamma$. Fix any $f\in C_b(\calT)$, then by the duality \eqref{eq:gammas_dual},
\begin{align*}
  \liminf_{n\to \infty} \frac{1}{n}\Ent(\sfP^n|\Pi_n)&\geq   \liminf_{n\to \infty}  \int_{\Gamma} {\langle f,\nu\rangle} \,\dd \sfP^n-\limsup_{n\to \infty} G_n(f) 
  =\langle f,\bar \nu\rangle-G(f). 
\end{align*}
Taking the supremum over all $f\in C_b(\calT)$ we find 
\begin{equation*}
    \calF_{\infty}(\delta_{\bar \nu})=\frac{1}{2}\Ent(\bar \nu|\gamma)\leq \liminf_{n\to \infty} \calF_n(\sfP^n).
\end{equation*}

Finally, consider any $\bar \nu\in \Gamma$ with $\Ent(\bar \nu|\gamma)<\infty$ and set $\sfP=\delta_{\bar \nu}$. We will construct a sequence of measures $\sfP^n$ that locally consists of Poisson measures induced by $\bar \nu$. Namely, set 
\begin{equation*}
    \Pi_{n,\bar \nu}:=(L_n)_{\#} \pi_{n,\bar \nu},\qquad\text{with}\qquad \pi_{n,\bar \nu}:=\frac{1}{e^{n \bar \nu(\calT)}-1}\sum_{N=1}^{\infty} \frac{n^N}{N!}\bar \nu^{\otimes N},
\end{equation*}
and consider the sequence $\sfP^n:=\Pi_{n,\bar \nu}$. It is straightforward to verify that indeed $\sfP^n\to \delta_{\bar \nu}$. Moreover, note that although $L_n$ is not bijective, we do have the equality 
\[ \Ent(\sfP^n|\Pi_n)=\Ent(\pi_{n,\bar \nu}|\pi_n), \]
due to the symmetry of the $N$-particle distributions $\bar \nu^{\otimes N}$, $\gamma^{\otimes N}$. Therefore, we derive
\begin{equation*}
\begin{aligned}
    \Ent(\sfP^n|\Pi_n)&=\Ent(\pi_{n,\bar \nu}|\pi_n)\\
    &=\frac{1}{e^{n \bar \nu(\calT)}-1}\sum_{N=1}^{\infty} \frac{n^N}{N!}\int_{\calT^N} \log \left(\frac{e^{n \gamma(\calT)}-1}{e^{n \bar \nu(\calT)}-1}\frac{\dd \bar \nu^{\otimes N}}{\dd \gamma^{\otimes N}}\right) \dd \bar \nu^{\otimes N}\\
    &=\frac{1}{e^{n \bar \nu(\calT)}-1}\sum_{N=1}^{\infty} \frac{n^N}{N!} \left(N \bar \nu(\calT)^{N-1} \int_{\calT}  \log \left(\frac{\dd \bar \nu}{\dd \gamma}\right) \dd \bar \nu+\bar \nu(\calT)^{N}\log \frac{e^{n \gamma(\calT)}-1}{e^{n \bar \nu(\calT)}-1}  \right)\\
    &=n\frac{e^{n \bar \nu(\calT)}}{e^{n \bar \nu(\calT)}-1}\int_{\calT} \log \left(\frac{\dd \bar \nu}{\dd \gamma}\right)\dd \bar \nu +\log \frac{e^{n \gamma(\calT)}-1}{e^{n \bar \nu(\calT)}-1}.
\end{aligned}
\end{equation*}
Rescaling and taking the limit $n\to\infty$, we obtain
\[
\lim_{n\to \infty} \frac{1}{n}\Ent(\sfP^n|\Pi_n)=\int_{\calT} \log \left(\frac{\dd \bar \nu}{\dd \gamma}\right)\dd \bar \nu -\bar \nu(\calT)+\gamma(\calT) =\Ent(\bar \nu|\gamma),
\]
therewith concluding the proof.
\end{proof}

\subsection{Uniform estimates}\label{ss:c_uni}

In Section \ref{ss:fke_est} we provided uniform-in-$n$ estimates for the flux. Namely, from Lemma \ref{lm:fke_est}, we directly have the following. 

\begin{cory}\label{cory:c_uni0}
Consider a sequence $(\sfP^n,\sfJ^{n,+},\sfJ^{n,-})\in \mathsf{CE}_n$ such that 
\[ \limsup_{n\to \infty} \int_0^T \calR_n(\sfP_t^n,\sfJ^{n,+}_t,\sfJ^{n,-}_t)< \infty.\] 
Then 
\begin{equation*}
    \limsup_{n\to\infty}\int_0^T 3M \tilde \phi\left(\frac{1}{3M} \int_{\Gamma\times \calT} (1+\nu(X)^2)^{-1}\,\sfJ^{n,\pm}_t(\dd \nu,\dd x) \right)\, \dd t < \infty,
\end{equation*}
where $M:=(1+\gamma(\calT))\|c\|_{\infty}$.
\end{cory}

However, the weighted total variation metric $d_{TV,w}$ that was introduced is not appropriate for taking limits, and instead, we take the weaker metric defined in \eqref{eq:c_uni_m},
\begin{equation*}
    W(\sfP^1,\sfP^2):=\sup_{F\in \mathbb{F}} \left\{ \int_{\Gamma} F \,\dd (\sfP^1-\sfP^2) \right\},
\end{equation*}
where 
\begin{equation*}
    \mathbb{F}:=\left\{ F\in \Cyl_c(\Gamma)\, : \,  \left\|(1+\nu(\calT)^2)\, \mathrm{grad}_{\Gamma}\,F\right\|_{\infty}\leq 1 \right\}.
\end{equation*}
Recall that $W$ is narrowly lower semicontinuous and implies narrow convergence on narrowly pre-compact subsets. 
We now have the follow equicontinuity result. 
\begin{lm}\label{lm:c_uni1}
Consider a sequence $(\sfP^n,\sfJ^{n,+},\sfJ^{n,-})\in \mathsf{CE}_n$ such that 
\[ \limsup_{n\to \infty} \int_0^T \calR_n(\sfP_t^n,\sfJ^{n,+}_t,\sfJ^{n,-}_t)\, \dd t< \infty.\]
Then
\begin{equation*}
\limsup_{n\to \infty} \int_0^T \tilde \phi\left(\frac{|\dot \sfP^n_t|_{W}}{12M}\right)<\infty,
\end{equation*}
where $|\dot \sfP_t|_{W}$ is the $W$-metric speed and $\tilde \phi(s):=\phi(s \vee 1)$ is the monotone relaxation of $\phi$.
\end{lm}

\begin{proof}
The proof is similar to Lemmas \ref{lm:fke_timereg} and \ref{lm:liouv_timereg}, now for the distance $W$ instead of the weighted total variation metric $d_{TV,w}$. Namely, fix $n>0$ and consider a curve $(\sfP,\sfJ^+,\sfJ^-)\in \mathsf{CE}_n$. Then we have for any $s,t\in [0,T]$ and any $F\in C_c(\Gamma)$, 
\begin{equation*}
  \left|  \int_{\Gamma} F \dd (\sfP_t-\sfP_s) \right| \leq \int_s^t \int_{\Gamma\times \calT} |\dder^{n,+}F(\nu,x)| \, \dd \sfJ_r^{+} \, \dd r + \int_s^t \int_{\Gamma\times \calT} |\dder^{n,-}F(\nu,x)| \, \dd \sfJ_r^{-} \, \dd r.
\end{equation*}
Substituting any $F\in \mathbb{F}$ it is straightforward to verify that 
\begin{align*}
    |\dder^{n,+}F(\nu,x)| = n|F(\nu+\tfrac{1}{n}\delta_x)-F(\nu)|&\leq (1+\nu(\calT)^2)^{-1}\\
    |\dder^{n,-}F(\nu,x)| = n|F(\nu)-F(\nu-\tfrac{1}{n}\delta_x)|&\leq (1+(\nu(\calT)-\tfrac{1}{n})^2)^{-1}
    \leq 2(1+\nu(\calT)^2)^{-1},
\end{align*}
for sufficiently large $n$, and therefore 
\begin{equation*}
  \left|\int_{\Gamma} F \dd (\sfP_t-\sfP_s) \right| \leq 2 \int_s^t \int_{\Gamma\times \calT} (1+\nu(\calT)^2)^{-1} \, \dd (\sfJ_r^{+}+\sfJ_r^{-}) \, \dd r.
\end{equation*}
Taking the supremum over $F\in \mathbb{F}$, we find that $(\sfP_t)_{t\in [0,T]}$ is absolutely continuous w.r.t.\ $W$ with
\begin{equation*}
 |\dot\sfP_t|_{W}\leq 2 \int_{\Gamma\times \calT} (1+\nu(\calT)^2)^{-1} \, \dd (\sfJ_t^{+}+\sfJ_t^{-}) \qquad \mbox{for a.e. $t\in [0,T]$},
\end{equation*}
where $|\dot \sfP^n_t|_W$ is the $W$-metric speed. Applying the estimates in Lemma \ref{lm:fke_est} concludes the proof.
\end{proof}

\subsection{Proof of main results}\label{ss:c_main} We finally conclude the manuscript with the proof of the main results.

\begin{proof}[Proof of Theorem \ref{thm:main_conv1}]
We will first establish the liminf-estimates. Namely, consider a sequence 
$(\sfP^n,\sfJ^{n,+},\sfJ^{n,-})\in \mathsf{CE}_n$ that converges to the curve $(\sfP,\sfJ^+,\sfJ^-)\in \mathsf{CE}_{\infty}$. In particular $\sfP_t^n\to \sfP_t$ for all $t\in [0,T]$, and hence by Theorem \ref{thm:c_gamma} on the $\Gamma$-convergence of $\calF_n$ we immediately obtain 
\begin{equation*}
    \liminf_{n\to \infty} \calF_n(\sfP_t^n)\geq \calF_{\infty}(\sfP_t), \qquad \fA t\in [0,T].
\end{equation*}
Now suppose that 
\begin{equation*}
    \limsup_{n\to \infty} \calF_n(\sfP_0^n)<\infty, \qquad  \limsup_{n\to \infty} \calI_n(\sfP^n,\sfJ^{n,+},\sfJ^{n,-})<\infty.
\end{equation*}
In particular we have the bounds 
\begin{equation}\label{eq:mainconv_f2}
     \limsup_{n\to \infty} \int_0^T \calR_n(\sfP^n_t,\sfJ^{n,+},\sfJ^{n,-}) \, \dd t < \infty, \quad \limsup_{n\to \infty} \int_0^T \calD_n(\sfP^n_t) \, \dd t < \infty.
\end{equation}
Due to the chain rule and the assumption on $\calF_n(\sfP^n_0)$, we obtain
\begin{equation}\label{eq:mainconv_f5}
    \limsup_{n\to \infty} \sup_{t\in[0,T]} \calF_n(\sfP^n_t) < \infty. 
\end{equation}
The latter guarantees, by Corollary~\ref{cory:nc_2}, that we have the vague convergence
\begin{equation*}
\begin{aligned}
        \lim_{n\to \infty} \vartheta^{\pm}_{\sfP_t^n} = \vartheta^{\pm}_{\sfP_t}, \qquad
          \lim_{n\to \infty}  \sfT^{n,\pm}_{\#}\vartheta_{\sfP_t^n}^{\pm} =\vartheta^{\pm}_{\sfP_t}.
\end{aligned}
\end{equation*}
Recall that from Lemma \ref{lm:fke_est} and Remark \ref{rem:fke_fischereq} we have for each $n\ge 1$:
\begin{equation*}
  \begin{aligned}
\Ent\left(\sfJ_t^{n,\pm}|\Theta_{\sfP}^{n,+}\right)&=\int_{\Gamma\times \calT} \Upsilon\left(\frac{\dd \sfJ_t^{n,\pm}}{\dd \Sigma},\frac{\dd \vartheta_{\sfP_t}^{\pm}}{\dd \Sigma},\frac{\dd (\sfT^{n,\mp}_{\#}\vartheta_{\sfP_t}^{\mp})}{\dd \Sigma}\right)\dd  \Sigma,\\
\calD_{n}(\sfP_t)&=2H^2(\vartheta_{\sfP_t}^{\pm},\sfT^{n,\mp}_{\#}\vartheta_{\sfP_t}^{\mp}),
  \end{aligned}
\end{equation*}
for any dominating measure $\Sigma$, and similarly, from Corollary \ref{cory:li_est2} and Remark \ref{rem_lifisher} that  
\begin{equation*}
 \begin{aligned}
\Ent\left(\sfJ_t^{\pm}|\Theta_{\sfP}^{+}\right)&=\int_{\Gamma\times \calT} \Upsilon\left(\frac{\dd \teta_{\sfP_t}^{\pm}}{\dd \Sigma},\frac{\dd \vartheta_{\sfP_t}^{\pm}}{\dd \Sigma},\frac{\dd \vartheta_{\sfP_t}^{\mp}}{\dd \Sigma}\right)\dd  \Sigma,\\
\calD_{\infty}(\sfP_t)&=2H^2(\vartheta_{\sfP_t}^{\pm},\vartheta_{\sfP_t}^{\mp}).
  \end{aligned}
\end{equation*}
By the convexity and lower semi-continuity of $\Upsilon$ and $H$ we conclude by standard semi-continuity results (e.g.\ see \cite[Theorem 3.4.3]{Buttazzo1989}) that for each $t\in [0,T]$,
\[
    \liminf_{n\to \infty} \calR_n(\sfP_t^n,\sfJ^{n,+}_t,\sfJ^{n,-}_t) \geq \calR_n(\sfP_t,\sfJ^{+}_t,\sfJ^{-}_t), \qquad    \liminf_{n\to \infty} \calD_n(\sfP_t^n) \geq \calD_n(\sfP_t),
\]
from which \eqref{main_conv1_linf} directly follows after applying the Fatou lemma.

\medskip

Next, we consider the question of compactness. As in the previous part, let us consider a sequence $(\sfP^n,\sfJ^{n,+},\sfJ^{n,-})\in \mathsf{CE}_n$ with 
\begin{equation*}
    \limsup_{n\to \infty} \calF_n(\sfP_0^n)<\infty, \qquad  \limsup_{n\to \infty} \calI_n(\sfP^n,\sfJ^{n,+},\sfJ^{n,-})<\infty,
\end{equation*}
which imply that the estimates \eqref{eq:mainconv_f2} and \eqref{eq:mainconv_f5} still hold. The bound on the free energy ensures by Theorem \ref{thm:c_gamma} that $\{\sfP_{t}^n\}_{t\in [0,T],n\ge 1}$ is pre-compact. Moreover,  due to the bound on the action $\calR_n$, we have by the results of Corollary \eqref{cory:c_uni0} and Lemma \eqref{lm:c_uni1} that 
\begin{equation}\label{eq:mainconv_f3}
\limsup_{n\to \infty}  \int_0^T \tilde \phi\left(\frac{1}{3M} \int_{\Gamma\times \calT} (1+\nu(X)^2)^{-1}\,\sfJ^{n,\pm}_t(\dd \nu,\dd x) \right) \, \dd t < \infty,
\end{equation} 
\begin{equation}\label{eq:mainconv_f4}
\limsup_{n\to \infty} \int_0^T \tilde \phi\left(\frac{|\dot \sfP^n_t|_{W}}{12M}\right) \, \dd t<\infty,
\end{equation}
where $|\dot \sfP^n_t|_W$ is again the $W$-metric speed. From \eqref{eq:mainconv_f3}, we then conclude from the non-decreasing, convex and super-linear at infinity property of $\tilde \phi$ that, up to choosing a subsequence $n'$, there exists a family $\{\sfJ^{\pm}_t\}_{t\in [0,T]} \in \calM_{loc}^+(\Gamma\times\calT)$ such that for all $s,t$ the sequence of measures $\sfJ^{n',\pm}_r(\dd \nu,\dd x)\, \dd r$ converges to $\sfJ^{\pm}_r(\dd \nu,\dd x)\, \dd r$ in $\calM_{loc}(\Gamma\times \calT\times [s,t])$, and 
\[ \int_0^T \int_{\Gamma\times \calT} \tilde \phi\left(\frac{1}{3M} \int_{\Gamma\times \calT} (1+\nu(X)^2)^{-1}\,\sfJ^{\pm}_t(\dd \nu,\dd x) \right) \, \dd t < \infty. \]
Similarly, since the metric $W$ is narrowly lower semicontinuous and induces narrow convergence on narrowly pre-compact subsets, we find by an Arzela-Ascoli argument and the estimate \eqref{eq:mainconv_f4} that, up to choosing a subsequence $n''$, there exist a narrowly continuous curve $(\sfP_t)_{t\in [0,T]}$ such that  $\sfP^{n''}_t$ converges to $\sfP_t$ for all $t\in [0,T]$. 

All that remains is showing that $(\sfP,\sfJ^+,\sfJ^-)\in \mathsf{CE}_{\infty}$. Therefore, fix any $s,t\in [0,T]$ and $F \in \Cyl_c(\Gamma)$. It is straightforward to verify that there exist constants $K_{F}$ and $C_F$ such that the following Taylor approximation holds:
\[ \left| \mathrm{grad}_{\Gamma}(\nu,x)\mp n\left(F(\nu{\pm}\tfrac{1}{n}\delta_x)-F(\nu)\right)\right|\leq \tfrac{C_{F}}{n} 1_{\nu(\calT)\leq K_F}(\nu,x), \qquad \fA \nu\in \Gamma,\, x\in \calT. \]
Thus, we can take the limit in the continuity equation $\mathsf{CE}_n$, to conclude that 
\begin{align*}
   \int_{\Gamma} F(\nu) \,\dd \sfP_t - \int_{\Gamma} F(\nu) \,\dd \sfP_s &=\lim_{n\to \infty}  
        \int_{\Gamma} F(\nu) \,\dd \sfP^{n''}_t - \int_{\Gamma} F(\nu) \,\dd \sfP^{n''}_s \\
        &= \lim_{n\to \infty}   \int_s^t \left( \int_{\Gamma\times \calT} (\dder^{n'',+}F)\,\dd \sfJ_r^{n'',+}+(\dder^{n'',-}F)\, \dd \sfJ_r^{n'',-} \right) \, \dd r \\
        &=\int_s^t \Big(\int_{\Gamma\times \calT} (\mathrm{grad}_{\Gamma} F)\,\dd \sfJ_r^+-(\mathrm{grad}_{\Gamma} F)\, \dd \sfJ_r^- \Big) \, \dd r,
\end{align*}
thereby concluding the proof.
\end{proof}

\begin{proof}[Proof of Theorem \ref{thm:edp_prop}]
Suppose that $\bar \sfP^n\to \bar \sfP =\delta_{\bar \nu}$ with 
\begin{equation*}
    \lim_{n\to \infty} \calF_n(\bar \sfP^n)=\frac{1}{2}\Ent(\nu|\gamma).
\end{equation*}
For each $n\in\N$ let $\sfP_t^n$ be the unique gradient-flow solution to $\eqref{eq:FKEn}$ with initial data $\bar \sfP^n$. Moreover, let $\nu_t$ be the unique solution to \eqref{eq:mf2} with initial data $\bar \nu$, and set $\sfP_t:=\delta_{\nu_t}$, which is the unique gradient-flow solution to the Liouville equation \eqref{eq:liouv2} with initial data $\bar \sfP$. 
Then by Theorem \ref{lm:main_conv2} we have for every $t\in [0,T]$ that $\sfP_t^n\to \sfP_t$, and 
\begin{equation*}
    \lim_{n\to \infty} \calF_n(\sfP^n_t)=\calF_{\infty}(\sfP_t)=\frac{1}{2}\Ent(\nu_t|\gamma).
\end{equation*}

\medskip

Next, suppose that in addition there exists a constant $C>1$ such that $C^{-1}\leq \dd \bar \nu/\dd \gamma \leq C$. By Lemma \ref{lm:mfsolb} we find that there exists $C'<\infty$ with
\[\sup_{t\in [0,T]} \left\|\log u_t\right\|_{\infty}<C', \qquad u_t:=\dd \nu_t/\dd \gamma. \]
Now fix any $t\in [0,T]$, and recall that
\begin{equation*}
    \Pi_{n,\nu_t}:=(L_n)_{\#} \pi_{n,\nu_t},\qquad \pi_{n,\nu_t}=\frac{1}{e^{n \bar \nu(\calT)}-1}\sum_{N=1}^{\infty} \frac{n^N}{N!}\nu_t^{\otimes N}.
\end{equation*}
It is straightforward to check that $\Pi_n \ll \Pi_{n,\nu_t} \ll \Pi_n$ and hence for any $\Gamma_n \ni  \Gamma_n=L_n(x_1,\dots,x_N)$, 
\[
    \log\left(\frac{\dd \Pi_{n,\nu_t}}{\dd \Pi_n}\right)(\nu)=\log \left(\frac{e^{n \gamma(\calT)}-1}{e^{n \nu_t(\calT)}-1}\frac{\dd \nu_t^{\otimes N}}{\dd \gamma^{\otimes N}}\right)
    =\log \left(\frac{e^{n \gamma(\calT)}-1}{e^{n \nu_t(\calT)}-1}\right)+\sum_{i=1}^N \log u_t(x_i),
\]
with all terms finite, and $|\sum \log u_t(x_i)|\leq N C'$. Therefore, by applying a similar density argument for $\log u_t$ as in Theorem \ref{thm:nc_1} we derive 
\begin{align*}
    \lim_{n\to \infty} \frac{1}{n}\int_{\Gamma} \log\left(\frac{\dd \Pi_{n,\nu_t}}{\dd \Pi_n}\right)\, \dd \sfP_t^n &=\lim_{n\to \infty} \frac{1}{n}\log \left(\frac{e^{n \gamma(\calT)}-1}{e^{n \nu_t(\calT)}-1}\right)+\lim_{n\to \infty} \frac{1}{n}\int_{\Gamma} \langle \log u_t,\nu\rangle\, \dd \sfP_t^n\\
    &=\gamma(\calT)-\nu_t(\calT)+\langle \log u_t,\nu_t\rangle\\
    &=\Ent(\nu_t|\gamma).
\end{align*}
Subsequently, we can compute as follows:
\begin{align*}
  \lim_{n\to \infty}  \Ent(\sfP^n_0|\Pi_{n})&=\frac{1}{n}\int_{\Gamma} \phi\left(\frac{\dd \sfP^n_0}{\dd \Pi_{n}}\right)\, \dd \Pi_{n}\\
    &=\lim_{n\to \infty}  \frac{1}{n}\int_{\Gamma} \left(\log \left(\frac{\dd \sfP^n_0}{\dd \Pi_{n,\nu_0}}\right)+\log \left(\frac{\dd \Pi_{n,\nu_0}}{\dd \Pi_{n}}\right)\right)\, \dd \sfP^n_0\\
    &=\Ent(\nu_0|\gamma),
\end{align*}
and hence the initial data are well-prepared. Therefore, we can conclude for all $t\in [0,T]$
\begin{align*}
  \lim_{n\to \infty}  \Ent(\sfP^n_t|\Pi_{n,\nu_t})&=\frac{1}{n}\int_{\Gamma} \phi\left(\frac{\dd \sfP^n_t}{\dd \Pi_{n,\nu_t}}\right)\, \dd \Pi_{n,\nu_t}\\
    &=\lim_{n\to \infty}  \frac{1}{n}\int_{\Gamma} \left(\log \left(\frac{\dd \sfP^n_t}{\dd \Pi_{n}}\right)+\log \left(\frac{\dd \Pi_{n}}{\dd \Pi_{n,\nu_t}}\right)\right)\, \dd \sfP^n_t\\
    &=\Ent(\nu_t|\gamma)-\Ent(\nu_t|\gamma) =0,
\end{align*}
thus establishing the entropic propagation of chaos result.
\end{proof}

\appendix

\section{Motivation from large deviations}\label{s:ldpmot}

In Section \ref{s:fke}, we introduced a new generalized gradient structure for the forward Kolmogorov equation and later showed convergence in the large-population limit to a structure that was lifted from the mean-field dynamics. Here we briefly discuss the relation between existing variational structures, and their connection to the asymptotic probabilities of the underlying process as treated in large deviation theory. All calculations are purely formal and are meant for illustratory purposes.

\medskip

Throughout, for simplicity, let $\calT$ be a finite set. Recall the reacting particle system formulation described by \eqref{eq:ireaction}, i.e.\ as particles $A_t^1,\dots,A_t^{N_t} \in \calT$ at positions $X_t^1,\dots,X_t^{N_t} \in \calT$, and with 
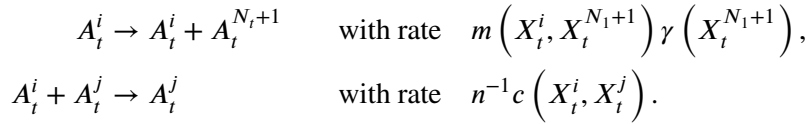
\begin{equation*}
\begin{aligned}
    A^i_t &\to A^i_t+A^{N_t+1}_t \quad &\mbox{with rate} \quad& m\left(X_t^i,X_t^{N_1+1}\right)\gamma\left(X_t^{N_1+1}\right),\\
  A^i_t+A^j_t &\to A^j_t \quad &\mbox{with rate} \quad &n^{-1}c\left(X^i_t,X^j_t\right).
\end{aligned}   
\end{equation*}
Let $L_t^n$ be the rescaled empirical measure 
\[ L_t^n(x):=\sum_{i=1}^{N_t} \delta_{X_t^i}(x), \]
and $W_t^{n,\pm}$ the \emph{integrated birth/death fluxes}:
\[W^{n,\pm}_t(x):=\frac{1}{n} \#\Big\{\mbox{Number of births($+$)/deaths($-$) at position $x$ in the time-window $[0,t)$}\Big\}.\]
Moreover, assume that the particles are initially distributed at time $t=0$ as $\pi_n$. Then by the work of \cite{Renger2019}, one can derive under suitable assumptions that the triple $(L_t^n,W_t^{n,\pm})$ is a well-defined Markov process and satisfies a \emph{large-deviation principle} as $n\to \infty$ with rate function $\calI(\nu,\lambda^+,\lambda^+)$ in the sense that asymptotically (as $n\to \infty$)
\begin{equation*}
    \mathrm{Prob}\left(L_t^n \approx \nu_t, \, W_t^{n,\pm} \approx \int_0^t \lambda_s^{\pm}\, \dd s, \, \forall t\in [0,T]\right) \asymp e^{-n \left(\calI^{0}(\nu_0)+\calI(\nu,\lambda^+,\lambda^-)\right)}
\end{equation*}
where $\calI^0(\nu):=\Ent(\nu|\gamma)$ and
\begin{equation*}
      \calI(\nu,\lambda^+,\lambda^-):=\int_0^T \calL_{MF}(\nu_t,\lambda^+_t,\lambda_t^-)\, \dd t, \qquad    \calL_{MF}(\nu,\lambda^+,\lambda^-):=\Ent(\lambda^+|\kappa_{\nu}^+)+\Ent(\lambda^-|\kappa_{\nu}^-).
\end{equation*}

\medskip

Now, under the detailed balance assumption $m(x,y)=c(y,x)$ for all $x,y\in \calT$, one can show that if $\calF_{MF}(\nu_0)<\infty$ the rate function $\calI$ is precisely the mean-field EDP-functional defined in \eqref{eq:mf_edpf}:
\begin{equation*}
    \calI(\nu,\lambda^+,\lambda^-)=\calI_{MF}(\nu,\lambda^+,\lambda^-).
\end{equation*} 
This can be seen via symmetrization under \emph{time-reversal}. Note that for any curve $(\nu,\lambda^+,\lambda^-)\in \mathscr{CE}$ the `reversed' curve $(\nu_{T-t},\lambda^{-}_{T-t},\lambda^{+}_{T-t})$ is still contained in $\mathscr{CE}$, and
\[ \calI^{\dag}(\nu,\lambda^+,\lambda^-):=\int_0^T \calL(\nu_{T-t},\lambda^-_{T-t},\lambda^+_{T-t})\, \dd t = \int_0^T \calL(\nu_{t},\lambda^-_{t},\lambda^+_{t}) \, \dd t.\]
Then for suitable curves we have the decomposition 
\begin{equation*}
    \begin{aligned}
     \frac{1}{2}\left( \calI(\nu,\lambda^+,\lambda^-)+ \calI^{\dag}(\nu,\lambda^+,\lambda^-)\right)&=\int_0^T \left(\calR_{MF}(\nu_t,\lambda^+_t,\lambda^-_t)+\calD_{MF}(\nu_t) \right)\, \dd t, \\
        \frac{1}{2}\left( \calI(\nu,\lambda^+,\lambda^-)- \calI^{\dag}(\nu,\lambda^+,\lambda^-)\right)&=\calF_{MF}(\nu_T)-\calF_{MF}(\nu_0),
    \end{aligned}
\end{equation*}
which follows from the fact that if $\calL(\nu_t,\lambda^+_t,\lambda_t^-)$ and $\calL(\nu_t,\lambda^-_t,\lambda_t^+)$ are finite
\begin{equation*}
    \begin{aligned}
   \frac{1}{2}\left( \calL_{MF}(\nu_t,\lambda^+_t,\lambda^-_t)+ \calL_{MF}(\nu_t,\lambda^-_t,\lambda^+_t)\right)&=\calR_{MF}(\nu_t,\lambda^+_t,\lambda^-_t)+\calD_{MF}(\nu),\\
  \frac{1}{2}\left( \calL_{MF}(\nu_t,\lambda^+_t,\lambda^-_t)- \calL_{MF}(\nu_t,\lambda^-_t,\lambda^+_t)\right)&=\frac{1}{2}\partial_t\, \Ent(\nu_t|\gamma).
    \end{aligned}
\end{equation*}

\medskip

The splitting above is a direct consequence of fact that under the assumption of $c(x,x)=0$, $m(x,y)=c(y,x)$ for all $x,y\in \calT$, the underlying jump process $L_t^n$ is reversible, i.e., $\bar{\kappa}_{n}$ satisfies the detailed balance condition $\Pi_n(\dd \nu)\bar{\kappa}_{n}(\dd \nu,\dd \eta)=\Pi_n(\dd \eta)\bar{\kappa}_{n}(\dd \eta,\dd \nu)$. Namely, consider the functional $\bar \calI_n$ given by
\begin{equation*}
    \bar{\calI}_n(\sfP,j):=\int_0^T \bar \calL_n(\sfP_t,j)\, \dd t, \qquad \bar{\calL}_n(\sfP_t,j_t):=\Ent(j_t|\sfP_t \bar{\kappa}_n),    
\end{equation*}
where $j(\dd \nu,\dd \eta)\in \calM^+(\Gamma\times \Gamma)$ and $(\sfP_t \bar{\kappa}_n)$ is short-hand for the measure $\sfP(\dd \nu)\bar{\kappa}_n(\dd \nu,\dd \eta)$. Let $j^{\dag}(\dd \nu,\dd \eta):=j(\dd \eta,\dd \nu)$, which again corresponds to a time-reversal procedure. We then have for suitable $(\sfP,j)$ the following decomposition
\begin{equation*}
    \begin{aligned}
   \frac{1}{2}\left(\calL_{n}(\sfP_t,j_t)+ \calL_{n}(\sfP_t,j^{\dag}_t)\right)&=\Ent(j_t|\bar{\Theta}^n_{\sfP_t})+2H^2(\sfP_t \bar{\kappa}_n,\bar{\kappa}_n \sfP_t),\\
  \frac{1}{2}\left( \calL_{n}(\sfP_t,j_t)- \calL_{n}(\sfP_t,j^{\dag}_t)\right)&=\frac{1}{2}\partial_t\, \Ent(\sfP_t|\Pi_n),
    \end{aligned}
\end{equation*}
where 
\[  \bar{\Theta}^n_{\sfP}:=\sqrt{(\sfP \bar{\kappa}_n)(\bar{\kappa}_n \sfP )},\]
and $(\bar{\kappa}_n \sfP )$ is short-hand for the measure $\sfP(\dd \eta)\bar{\kappa}_n(\dd \eta,\dd \nu)$. Now substituting 
\begin{equation*}
    j_t(\dd \nu,\dd \eta):=n \int_{x\in \calT} \delta_{\nu+ \tfrac{1}{n}\delta_x}(\dd \eta) \,\sfJ^{+}_t(\dd \nu,\dd x)+n \int_{x\in \calT} \delta_{\nu- \tfrac{1}{n}\delta_x}(\dd \eta) \,\sfJ^{-}_t(\dd \nu,\dd x),
\end{equation*}
we find that 
\begin{equation}\label{eq:mot_nscal}
   \bar \calL_n(\sfP,j)=n\, \Ent(j_t|\sfP_t \bar \kappa_n), \qquad
   \bar{\calI}_n(\sfP,j)=n\, \calI_n(\sfP,\sfJ^+,\sfJ^-).
\end{equation}
And as we have shown, in the large-population limit of $n\to \infty$, $\calI_n$ EDP-converges to a functional that is lifted from $\calI_{MF}$, establishing the microscopic origin of the splitting for $\calI_{MF}$. 

\medskip
This decomposition for reversible processes is well-known in the net-flux representation. Namely, one can show via a minimization approach that 
\begin{equation*}
    \inf_{j} \left\{ \Ent(j|\bar{\Theta}^n_{\sfP})\, : \, \int_{\eta \in \Gamma} \left(j(\dd \nu,\dd \eta)-j(\dd \eta,\dd \nu)\right) = \mathrm{j}^{net}(\dd \nu)  \right\} = \frac{1}{2}\int_{\Gamma^2} \Psi\left(\frac{\dd j^{\mathrm{net}} }{\dd \bar{\Theta}^n_{\sfP}}\right)\dd \bar{\Theta}^n_{\sfP},
\end{equation*}
using a dualization argument and the elementary equality 
\begin{equation}\label{eq:fke_psiphi}
    \Psi(e^z-e^{-z})=\phi(e^z)+\phi(e^{-z}), \qquad \fA z\in \R.
\end{equation}
Thus $\bar{\calI}_n$ is simply the EDP-functional for jump processes of \cite{PRST2020}. The works \cite{Mielke2014,Zimmer2018,Peletier2022} contain an extensive overview and discussion on how $\bar{\calI}_n$ is the expected rate functional for a large-deviation principle for the empirical measures of independent jump processes, how the reversibility of the process ensures a possible splitting in both the interacting and non-interacting case, and how for complex-balanced systems this can even be done in the irreversible setting. Moreover, for an implicit decomposition using measure-dependent Dirichlet forms in the case of the homogeneous Boltzmann equation and the underlying process, see \cite{Basile2021}. 

\medskip

On a final note, due to \eqref{eq:mot_nscal} and the origin of $\bar \calI_n$ in large deviations for independent particles (or via variational representations as found in \cite{Ellis1997}), one would expect that if $F_t\in C_b(\Gamma)$ for all $t\in [0,T]$, we would have for all $n>0$ the following representation formula for the expectation:
\[ \frac{1}{n}\log \mathbb{E}\left[ e^{-n \int_0^T F_t(L_t^n)\, \dd t}\right]  = \inf_{(\sfP,\sfJ^+,\sfJ^-)}\left\{ \int_0^T \int_{\Gamma} F_t(\nu_t) \sfP_t(\dd \nu)\, \dd t+ \frac{1}{n}\Ent(\sfP_0|\Pi_n) +\calI_n(\sfP,\sfJ^+,\sfJ^-)  \right\}. \] 
On the other hand, by the large deviation principle of $(L_t^n,W_t^{n,\pm})$ as $n\to \infty$, and Varadhan's Lemma (see \cite{Dembo10}), it holds that
\[ - \lim_{n\to \infty} \frac{1}{n}\log \mathbb{E}\left[ e^{-n \int_0^T F_t(L_t^n)\, \dd t}\right]\\
 =\inf_{(\nu,\lambda^+,\lambda^-)} \left\{ \int_0^T F(\nu_t)\, \dd t+\Ent(\nu_0|\gamma)+\calI_{MF}(\nu,\lambda^+,\lambda^-)  \right\}.\]
Consequently,
\begin{align*}
 &\lim_{n\to \infty} \inf_{(\sfP,\sfJ^+,\sfJ^-)}   \left\{ \int_0^T \int_{\Gamma} F_t(\nu_t)\sfP_t(\dd \nu)\, \dd t+ \frac{1}{n}\Ent(\sfP_0|\Pi_n) +\calI_n(\sfP,\sfJ^+,\sfJ^-)  \right\}\\
 &\hspace{10em}=\inf_{(\nu,\lambda^+,\lambda^-)} \left\{ \int_0^T F(\nu_t)\, \dd t+\Ent(\nu_0|\gamma)+\calI_{MF}(\nu,\lambda^+,\lambda^-)  \right\}.
\end{align*} 
Note that the lower bound of this equality follows from Theorem \ref{thm:main_conv1} and the superposition principle in Theorem \ref{thm:li_equiv}. Moreover, 
we expect that the large-deviation principle implies evolutionary $\varGamma$-convergence of $\calI_n$ in a suitable topology---an implication studied in \cite{Kraaij2019} in a general setting. 

It then begs the question if one can reverse this procedure, namely using evolutionary $\varGamma$-convergence to establish large-deviation principles similar to the non-evolutionary setting of \cite{Mariani2012}. This approach was successfully applied in the case of certain diffusion processes \cite{Fathi2016} and discussed for more general processes in \cite{Kaiser2019}.

\section{Superposition principle in \texorpdfstring{$\R^\N$}{infinite dimensions}}\label{s:super}

In this section, we present a superposition principle for continuity equations over $\R^{\N}$ with an additional weighted integrability condition on the associated vector fields. 

Following \cite[Section 7]{Ambrosio2014}, we equip $\R^{\N}$ with the product topology, and  $\pi_n:=(p_1,\dots,p_n)$ the canonical projections. The space $AC_w([0,T];\R^\N)$ consists of curves $\eta$ such that $p_i\circ \eta \in AC[0,T]$ for all $i\in\N$. Note that both $\R^\N$ and $C([0,T];\R^\N)$ are Polish spaces. Moreover, let $|\cdot|_{\infty}$ be the uniform norm on $\R^{\N}$.

Smooth $n$-cylindrical functions with compact support $f:\R^\N\to \R$ are given in the form of 
\[
    f(x)=\phi(\pi_n(x))=\phi(p_1(x),\dots,p_n(x)),\qquad x\in \R^\N, 
\]
with $\phi\in C_c^{\infty}(\R^n\to \R)$, and define their gradient by 
\[
    \nabla f(x):=\left(\frac{\partial \phi}{\partial z_1}(\pi_n(x)),\dots, \frac{\partial \phi}{\partial z_n}(\pi_n(x)),0,0,\dots \right).
\]
We set $\Cyl_c(\R^\N)$ as the union over $n\in\N$ of all smooth $n$-cylindrical functions with compact support.

In the following, we consider pairs $(\nu,\bc)$, where $(\nu_t)_{t\in[0,T]} \subset \calP(\R^\N)$ is a weakly continuous family of probability measures and $\bc:[0,T]\times \R^\N\to \R^\N$ is a Borel vector field satisfying
\begin{equation*}
    \int_{\R^\N} f\, \dd \nu_t - \int_{\R^\N} f\, \dd \nu_s = \int_s^t\int_{\R^\N} (\bc_r,\nabla f)\, \dd \nu_r\,\dd r \qquad \mbox{for all $f\in\Cyl_c(\R^\N)$,}
\end{equation*}
and all $0\le s\le t\le T$.

\medskip

We then have the following result. 
\begin{thm}\label{thm_super}
Let $(\nu,\bc)$ be as above. Furthermore, suppose that for some $M>0$
\begin{equation*}
   \int_0^T \int_{\R^\N} M \Psi\left(\frac{|\bc_t|_{\infty}}{M(1+|p_1|)}\right)\, \dd \nu_t \,\dd t< \infty.
\end{equation*}
Then there exists a Borel probability measure $\lambda$ over $C([0,T];\R^\N)$ satisfying $(e_t)_{\#} \lambda=\nu_t$ for all $t\in[0,T]$, and is concentrated on the family of curves $\gamma\in AC([0,T];\R^\N)$ that satisfy
\[
    \dot \gamma=\bc_t(\gamma),\qquad \text{for almost every $t\in[0,T]$.}
\]
\end{thm}

The proof of Theorem \ref{thm_super} combines a slight adaptation of the proof for the superposition principle in $\R^\N$ found in \cite[Theorem 7.1]{Ambrosio2014}, developed for use in metric measure spaces, with a finite-dimensional result for vector fields over $\R^n$ found in \cite[Theorem 4.4]{ambrosio2008}. Due to the strong similarities with the proof found in \cite{Ambrosio2014}, we merely give a brief sketch.

\begin{proof}
By tightness of $\nu_0$, we can choose a sequence of coercive functionals $\Phi_i$ such that 
\begin{equation*}
    \int \Phi_i(p_i(x)) \,\dd\nu_0 \leq 2^{-i}, \qquad \mbox{for all $i\in\N$},
\end{equation*}
and consider the functional $\calA(\eta):C([0,T];\R^\N)\to [0,+\infty]$ given by
\begin{equation*}
        \calA(\eta):=\left\{
         \begin{aligned}
         &
        \sum_{i=1}^{\infty} \left(\Phi_i(p_i\circ \eta(0))+\int_0^T M\Psi\left(\frac{|\dot \eta(t)|_{\infty}}{M(1+|p_1\circ \eta(t)|)}\right) \dd t \right) \quad& &\mbox{if $\eta\in AC_w([0,T];\R^\N)$,} \\
        &+\infty \quad& &\mbox{otherwise.} 
    \end{aligned} \right.
\end{equation*}
It is clear that $\calA$ is coercive in $C([0,T];\R^\N)$, and its sublevel sets contain curves that are absolutely continuous with respect to $|\cdot|_{\infty}$. This follows from the fact that $\sup_{t\in [0,T]}{|p_1\circ \eta|}$ is bounded on the sublevel sets of the functional
\begin{equation*}
    \left(\Phi_1(p_1\circ \eta(0))+\int_0^T M\Psi\left(\frac{|(p_1\circ \eta)'(t)|}{M(1+|p_1\circ \eta(t)|)}\right) \dd t \right).
\end{equation*}

Now, for every $n\in\N$, we define the marginals $\calP(\R^n)\ni \nu^n_t:=(\pi_n)_{\#}\nu_t$ and corresponding vector fields $\bc_{t}^n: \R^n\to \R^n$ by
\[
    p_i \circ \bc_{t}^n:=\frac{\dd\,  (\pi_n)_{\#}\left((p_i\circ \bc_t)\, \nu_t\right)}{\dd  \nu_t^n}.
\]
Note that $(\nu^n,\bc_t^n)$ satisfies the continuity equation in $\R^n$. By Jensen's inequality, and the fact that $|z_1|\le |z|\leq n |z|_{\infty}$, for $z=(x_1,\ldots,x_n)\in\R^n$, we have that
\begin{equation*}
\begin{aligned}
   T \Psi\left(\frac{1}{n MT}\int_0^T \int \frac{|\bc_t|}{(1+|x_1|)}\, \dd \nu_t \, \dd t \right) &\leq T \Psi\left(\int_0^T \int \frac{|\bc_t|_{\infty}}{MT(1+|x_1|)}\, \dd \nu_t \, \dd t \right)\\
     &\leq \int_0^T \int \Psi\left(\frac{|\bc_t|_{\infty}}{M(1+|x_1|)}\right) \dd \nu_t \, \dd t.
\end{aligned}
\end{equation*}
and in particular
\[ \int_0^T \int \frac{|\bc_t|}{(1+|x_1|)}\, \dd \nu_t \, \dd t < \infty.\] 

Hence, we can apply the finite-dimensional version of \cite[Theorem 4.4]{ambrosio2008}. Embedding this into $\R^\N$, we obtain the probability measure $\lambda_n$ over $C([0,T],\R^\N)$, concentrated on absolutely continuous curves satisfying $\dot \gamma=\bc_t^n(\gamma)$, and such that $(e_t)_{\#}\lambda_t=\nu_t^n$. We immediately see that 
\begin{equation*}
    \sup_{n\in\N} \int \calA(\gamma)\, \dd \lambda_n(\gamma) <\infty,
\end{equation*}
which yields the tightness of $\lambda^n$. 

Consider any converging sequence $\lambda^n$ (up to renumbering) and its limit $\lambda\in\calP(C([0,T];\R^\N))$. Since the sequence $(\nu_t^n)_{n\in\N}$ clearly converges to $\nu_t:=(e_t)_{\#}\lambda$ in $\calP(\R^\N)$ for every $t\in[0,T]$, it remains to show that $\lambda$ is concentrated on solutions of $\dot \gamma=\bc_t(\gamma)$. In fact we will show that 
\begin{equation*}
    \int \frac{\left|p_i \circ \gamma(t)-p_i\circ \gamma(s)-\int_s^t p_i\circ \bc_r(\gamma(r))\, \dd r\right| }{1+ \|p_1\circ\gamma\|_{\infty}}\,\lambda(\dd \gamma)=0\qquad\text{for each $i\in\N$ and any $0\le s\le t\le T$.}
\end{equation*}
Note that it suffices to show that for any vector field $d:[0,T]\times\R^\N\to\R$ with $d_t$ being $k$-cylindrical for every $t\in[0,T]$, we have that
\begin{equation}\label{eq:spp_sd}
    \int \frac{\left|p_i \circ \gamma(t)-p_i\circ \gamma(s)-\int_s^t d_r(\gamma(r))\, \dd r\right| }{1+\|p_1\circ\gamma\|_{\infty}}\,\lambda(\dd \gamma)\leq \int_s^t\int_{\R^\N} \frac{|p_i\circ \bc_r-d_r|}{1+|p_1|} \,\dd \nu_r\, \dd r,
\end{equation}
since then we can use density of time-dependent cylindrical functions in $L^1((1+|p_1|)^{-1} \nu_s \,\dd s)$ and the fact that for all $s$ it holds that $|p_1\circ \gamma(s)|\leq \|p_1\circ\gamma\|_{\infty}$.

To prove \eqref{eq:spp_sd}, recall that $\lambda^n$ is concentrated on absolutely continuous solutions of $\dot \gamma_s=\bc^n_s(\gamma_s)$. Hence,  
\begin{equation*}
\begin{aligned}
   \int \frac{\left|p_i \circ \gamma(t)-p_i\circ \gamma(0)-\int_0^t d_s(\gamma(s))\, \dd s\right| }{1+\|p_1\circ\gamma\|_{\infty}}\lambda^n(\dd \gamma) &=   \int \frac{\left|\int_0^t(p_i\circ \bc^n_s(\gamma(s))-d_s(\gamma(s)))\, \dd s\right| }{1+\|p_1\circ\gamma\|_{\infty}}\lambda^n(\dd \gamma)\\
   &\leq  \int \frac{\int_0^t |p_i\circ \bc^n_s-d_s|(\gamma(s))\, \dd s }{1+|p_1(\gamma(s))|}\lambda^n(\dd \gamma)\\
   &\leq \int_0^t \int_{\R^\N} \frac{|p_i\circ \bc^n_s-d_s|}{1+|p_1|}\, \dd \nu^n_s\, \dd s\\
    \end{aligned}
\end{equation*}
Note that the integrand on the left-hand side is continuous in $\gamma$. Therefore, since for $n\geq k$
\begin{equation*}
    (p_i\circ \bc_s^n-d_s)\nu_s^n=(\pi_n)_{\#}((p_i\circ \bc_s-d_s)\nu_s),
\end{equation*}
the result then follows after taking the limit  $n\to\infty$.
\end{proof}

\begin{remark}
If one is only interested in curves in $AC_w([0,T];\R^\N)$, the theorem also holds whenever %
\begin{equation*}
   \int_0^T \int \frac{|p_i(\bc_t)|}{1+|p_1|}\, \dd \nu_t \,\dd t< \infty, \qquad \mbox{for all $i\in\N$}.
\end{equation*}
The finite dimensional analog of this statement, set in $\R^n$ with the prefactor $(1+|x|)^{-1}$, is presented in \cite[Theorem 4.4]{ambrosio2008}. Moreover, for $\R^\N$, in \cite[Theorem 7.1]{Ambrosio2014} the condition reads as 
\begin{equation*}
   \int_0^T \int |p_i(\bc_t)|\,\dd \nu_t \, \dd t < \infty, \qquad \mbox{for all $i\in \N$}.
\end{equation*}
\end{remark}

\section{Non-continuous competition kernel}\label{s:nonc}

In the proof of Theorem \ref{thm:main_conv1} we require the vague convergence of $\vartheta_{\sfP^n}^{\pm}$ and $\sfT^{n,\pm}_{\#}\vartheta_{\sfP^n}^{\pm}$ under the assumption of narrow convergence of $\sfP^n$ and equiboundedness of the free energy functionals $\calF_n$, where
\begin{equation*}
    \begin{aligned}
        \vartheta_{\sfP}^{+}(\dd \nu,\dd x)=\int_{y\in \calT} c(x,y) \gamma(\dd x)\nu(\dd y)  \sfP(\dd \nu) \\
          \vartheta_{\sfP}^{-}(\dd \nu,\dd x)=\int_{y\in \calT} c(x,y) \nu(\dd x)\nu(\dd y)  \sfP(\dd \nu).
    \end{aligned}
\end{equation*}
If the competition kernel $c$ is continuous, the desired statement would follow directly from the narrow convergence of $\sfP^n$. The case of merely bounded measurable $c$ is however less trivial. Note that the strategy we employed in the proof of Theorem \ref{thm:fke_solex} is not possible, since although for every fixed $n$ the sub-levels of $\calF_n$ are sequentially compact with respect to setwise convergence, this is not the case for equibounded sets of $\{\calF_n\}_{n\ge 1}$. 

Fortunately, due to the connection between $\varGamma$-convergence of $\calF_n$ and large deviations as discussed in Section \ref{s:ldpmot}, we can modify results from the authors' earlier work on large deviations for interacting systems induced by singular or irregular functionals \cite{Hoeksema2020}. In particular, we obtain the following convergence statement. 

\begin{thm}\label{thm:nc_1}
Let $\{\sfP^n\}_{n\ge 1}\subset \calP(\Gamma)$ be a sequence narrowly converging to $\sfP\in \calP(\Gamma)$ with
\begin{equation*}
    \limsup_{n\to \infty} \calF_n(\sfP^n) < \infty.
\end{equation*}
Then for any $\omega\in C_c(\Gamma\times \calT)$ and $g \in \calB_b(\calT^2)$
\begin{equation*}
\begin{aligned}
    \lim_{n\to \infty} \int_{\calT^2\times \Gamma} g(x,y) \omega(\nu,x) \nu(\dd x)\nu(\dd y) \sfP^n(\dd \nu)= \int_{\calT^2\times \Gamma} g(x,y) \omega(\nu,x) \nu(\dd x)\nu(\dd y) \sfP(\dd \nu), \\
    \lim_{n\to \infty} \int_{\calT^2\times \Gamma} g(x,y) \omega(\nu,x) \gamma(\dd x)\nu(\dd y) \sfP^n(\dd \nu)= \int_{\calT^2\times \Gamma} g(x,y) \omega(\nu,x) \gamma(\dd x)\nu(\dd y) \sfP(\dd \nu).
\end{aligned}
\end{equation*}
\end{thm}

\begin{remark}
The result can be easily generalized to bounded measurable functions $g\in \calB_b(\calT^k)$ for finite $k\in\N$, but we restrict ourselves to the case $k=2$. 
\end{remark}

\begin{cory}\label{cory:nc_2}
Let $\{\sfP^n\}_{n\ge 1}\subset \calP(\Gamma)$ be a sequence narrowly converging to $\sfP\in \calP(\Gamma)$ such that 
\begin{equation*}
    \limsup_{n\to \infty} \calF_n(\sfP^n) < \infty.
\end{equation*}
Then vaguely 
\begin{equation}\label{eq:nc_c1}
\begin{aligned}
        \lim_{n\to \infty} \vartheta^{\pm}_{\sfP^n} = \vartheta^{\pm}_{\sfP}, \qquad
          \lim_{n\to \infty}  \sfT^{n,\pm}_{\#}\vartheta_{\sfP^n}^{\pm} =\vartheta^{\pm}_{\sfP}.
\end{aligned}
\end{equation}
\end{cory}
\begin{proof}
The first statement of \eqref{eq:nc_c1} follows directly from Theorem \ref{thm:nc_1} by substituting $g:=c$. Moreover, by the uniform continuity and compact support of any $\omega\in C_c(\Gamma\times \calT)$ we have 
\begin{align*}
    \lim_{n\to \infty} \int_{\calT\times \Gamma} \omega\, \dd \sfT^{n,+}_{\#}\vartheta_{\sfP^n}^{+} &=\lim_{n\to \infty} \int_{\calT\times \Gamma} \omega(\nu+\tfrac{1}{n}\delta_x,x)\, \vartheta_{\sfP^n}^{+}(\dd \nu,\dd x)\\
    &= \lim_{n\to \infty} \int_{\calT^2\times \Gamma} g(x,y) \omega(\nu+\tfrac{1}{n}\delta_x,x) \nu(\dd x)\nu(\dd y) \,\sfP^n(\dd \nu)\\
    &= \int_{\calT^2\times \Gamma} g(x,y) \omega(\nu,x) \gamma(\dd x)\nu(\dd y) \,\sfP(\dd \nu),
\end{align*}
and a similar approach works for $\sfT^{n,-}_{\#} \vartheta_{\sfP^n}^{-}$.
\end{proof}

For the proof of Theorem \eqref{thm:nc_1} we will need some a priori bounds. Namely, recall from Section \ref{ss:c_gammas} the generating functionals and their limit
\begin{equation*}
    G_n(f):=\frac{1}{n}\log \int_{\Gamma} e^{n \langle f,\nu\rangle} \,\Pi_n(\dd \nu), \quad G(f):=\int_{\calT} (e^{f}-1)\,\dd \gamma.
\end{equation*}
For the ``interacting'' case, namely functionals of the form 
\[ \frac{1}{n}\log \int_{\Gamma} e^{n \langle g,\nu^{\otimes 2}\rangle}\, \Pi_n(\dd \nu),\] 
there is however a problem with the unboundedness of the mass of $\nu$. Nevertheless, upon controlling the mass we can provide the following technical estimate. 

\begin{lm}\label{lm:nc_1}
Let $F(\nu):=h(\nu(\calT)) \langle g,\nu^{\otimes 2}\rangle$ with $\mathrm{supp}(h)\in [0,K]$ and $g\in \calB_b(\calT^2)$. Then
\begin{equation}\label{eq:nc_est2}
    \tfrac{1}{n}\log \int_{\Gamma} e^{n |F|}\, \dd \Pi_n\leq  \left(\int_{\calT^2} e^{4K \|h\|_{\infty} |g|(x,y) } \dd \gamma^{\otimes 2} \right)^{1/2}+\frac{1}{n}\left(K \|g\|_{\infty}\|h\|_{\infty}-\log (e^{n \gamma(X)}-1)\right),
\end{equation}
and in particular 
\begin{equation*}
  \limsup_{n\to \infty}  \frac{1}{n}\log \int_{\Gamma} e^{n |F|} \,\dd \Pi_n\leq  \left(\int_{\calT^2} e^{2K \|h\|_{\infty} |g|(x,y) } \dd \gamma^{\otimes 2} \right)^{1/2}.
\end{equation*}
\end{lm}

\begin{proof}
Suppose that 
\[     \limsup_{n\to\infty} \calF_n(\sfP^n)=:C < \infty,\]
and let us consider the following interaction energy functional:
\begin{equation*}
    E_g^N(x_1,\dots,x_N):=\frac{1}{N^2} \sum_{i,j\neq i} |g|(x_i,x_j).
\end{equation*}
From a Hoeffding's decomposition argument, see \cite[Lemma 3.8]{Hoeksema2020}, we have for every $N\geq 2$, $\alpha\geq 0$ the estimate
\begin{align*}
\frac{1}{N} \log \frac{1}{\gamma(\calT)^N}\int_{\calT^N} e^{\alpha N E_g^N(x_1,\dots,x_N)} \, \dd \gamma^{\otimes N}
      \leq \tfrac{1}{2} \log \left(\frac{1}{\gamma(\calT)^2}\int_{\calT\times \calT} e^{\frac{2\alpha N}{N-1} |g|(x,y) } \,\dd \gamma^{\otimes 2}\right).
\end{align*}
Moreover, since $N/(N-1)\leq 2$ for $N\geq 2$, and
\[ \sum_{i,j} |g|(x_i,x_j)  = \sum_{i,j\neq i} |g|(x_i,x_j) + \sum_{i} |g|(x_i,x_i) \leq \sum_{i,j\neq i} |g|(x_i,x_j) + N \|g\|_{\infty},\]
we find that 
\begin{equation*}
    \frac{1}{N} \log \left( \frac{1}{\gamma(\calT)^N}\int_{\calT^N} e^{\frac{\alpha}{N} \sum_{i,j} |g|(x_i,x_j) } \, \dd \gamma^{\otimes N}\right) \leq  \frac{1}{2}\log \left(\frac{1}{\gamma(\calT)^2}\int_{\calT\times \calT} e^{4\alpha |g|(x,y) } \,\dd \gamma^{\otimes 2}\right)+\frac{\alpha \|g\|_{\infty}}{N}.
\end{equation*}
Recall that $L_n(x_1,\dots,x_N):=\tfrac{1}{n}\sum \delta_{x_i}$. Since the mass $L_n(x_1,\dots,x_N)(\calT)=N/n$ is bounded by $K$ on the support of $F$ we have for $N\geq 2$:
\[|F|(L_n)\leq |h|(L_n(\calT)) \tfrac{1}{n^2} \sum_{i,j} |g|(x_i,x_j) \leq \frac{K \|h\|_{\infty}}{n N} \sum_{i,j} |g|(x_i,x_j),  \]
while for $N=1$ we have the trivial estimate $|F|(L_n)\leq \tfrac{K}{n} \|h\|_{\infty} \|g\|_{\infty}$, and hence for all $N\geq 1$,
\[
   \frac{1}{\gamma(\calT)^N} \int_{\calT^N} e^{n |F|(L_N)} \,\dd \gamma^{\otimes N} \leq e^{K \|g\|_{\infty}\|h\|_{\infty}} \left(\frac{1}{\gamma(\calT)^2}\int_{\calT^2} e^{4K \|h\|_{\infty} |g|(x,y) } \dd \gamma^{\otimes 2} \right)^{N/2}.
\]
Using the representation for $\Pi_n$ we can therefore estimate
\begin{align*}
          \int_{\Gamma} e^{n |F|}\, \dd \Pi_n&=\frac{1}{e^{n \gamma(\calT)}-1} \sum_{i=1}^N  \frac{(n \gamma(\calT))^N}{N!}\int_{\calT^N} e^{n |F|} \dd \gamma^{\otimes N}/\gamma(\calT)^N\\
          &\leq \frac{1}{e^{n \gamma(\calT)}-1} \sum_{i=0}^N  \frac{(n \gamma(\calT))^N}{N!}  e^{K \|g\|_{\infty}\|h\|_{\infty}}\left(\frac{1}{\gamma(\calT)^2}\int_{\calT^2} e^{4K \|h\|_{\infty} |g|(x,y) } \dd \gamma^{\otimes 2} \right)^{N/2}\\
          &=\frac{ e^{K \|g\|_{\infty}\|h\|_{\infty}}}{e^{n \gamma(\calT)}-1}\exp\left\{ n\gamma(\calT) \left(\frac{1}{\gamma(\calT)^2}\int_{\calT^2} e^{4K \|h\|_{\infty} |g|(x,y) } \dd \gamma^{\otimes 2} \right)^{1/2} \right\},
\end{align*}
which proves \eqref{eq:nc_est2}. The final desired statement follows directly after taking limits.
\end{proof}

With the above estimate in hand, we can now prove our convergence statement by approximating $g$ with a sequence of continuous $g_{\eps}$ such that
\begin{equation}\label{eq:nc_dens}
   \lim_{\eps\to 0} \int_{\calT^2} e^{\beta |g-g_{\eps}|(x,y) } \dd \gamma^{\otimes 2} = 0, \qquad \fA \beta>0. 
\end{equation}
The existence of such a sequence follows similarly as for density statements in $L^p(\gamma)$, see for example \cite{Hoeksema2020}[Theorem C.5].

\begin{proof}[Proof of Theorem \ref{thm:nc_1}]
Consider a $g\in \calB_b(\calT^2)$ and let $\{g_{\eps}\}_{\eps>0}\subset C_b(\calT^2)$ be a sequence approximating $g$ in the sense of \eqref{eq:nc_dens}. Note that by the narrow convergence of $\sfP^n$ we have for any $\omega \in C_c(\Gamma\times \calT)$ and any $\eps>0$ that
\begin{equation*}
\begin{aligned}
    \lim_{n\to \infty} \int_{\calT^2\times \Gamma} g_{\eps}(x,y) \omega(\nu,x) \nu(\dd x)\nu(\dd y) \sfP^n(\dd \nu)= \int_{\calT^2\times \Gamma} g_{\eps}(x,y) \omega(\nu,x) \nu(\dd x)\nu(\dd y) \sfP(\dd \nu), \\
    \lim_{n\to \infty} \int_{\calT^2\times \Gamma}  g_{\eps}(x,y) \omega(\nu,x) \gamma(\dd x)\nu(\dd y) \sfP^n(\dd \nu)= \int_{\calT^2\times \Gamma}  g_{\eps}(x,y) \omega(\nu,x) \gamma(\dd x)\nu(\dd y) \sfP(\dd \nu).
\end{aligned}
\end{equation*}
Note that by the compact support of $\omega$, it suffices to show that for every $K>0$:
\begin{subequations}
\begin{align}
\lim_{\eps\to 0} \limsup_{n\to \infty} \int_{\nu(\calT)\leq K}\left(\int_{\calT^2} |g-g_{\eps}|(x,y) \nu(\dd x)\nu(\dd y)\right) \sfP^n(\dd \nu) =0, \label{eq:nc_conv1}\\
\lim_{\eps\to 0} \limsup_{n\to \infty} \int_{\nu(\calT)\leq K}\left(\int_{\calT^2} |g-g_{\eps}|(x,y) \gamma(\dd x)\nu(\dd y)\right) \sfP^n(\dd \nu) =0\label{eq:nc_conv2}.
\end{align}
\end{subequations}
Let us consider \eqref{eq:nc_conv1}, and set 
\[ F_{\eps,K}(\nu):= 1_{\nu(\calT)\leq K} \int_{\calT^2} |g-g_{\eps}|(x,y)\nu(\dd x)\nu(\dd y).  \]
By duality of the entropy and Lemma \ref{lm:nc_1}, we have for every $n\ge 1,\eps>0,K>0$, and $\beta>0$,
\begin{align*}
   \int_{\Gamma} F_{\eps,K}(\nu) \sfP^n(\dd \nu)&\leq \frac{1}{\beta n}\Ent(\sfP^n|\Pi_n)+\frac{1}{\beta n}\log \int_{\Gamma} e^{n \beta F_{\eps,K}} \, \dd \Pi_n \\
   &\leq \frac{1}{\beta n}\Ent(\sfP^n|\Pi_n)+\frac{1}{\beta}\left(\int_{\calT^2} e^{4\beta K |g-g_{\eps}|(x,y) } \dd \gamma^{\otimes 2} \right)^{1/2} \\
   &\qquad\qquad +\frac{1}{\beta n}\left(K \beta \|g-g_{\eps}\|_{\infty}-\log (e^{n \gamma(\calT)}-1)\right).
\end{align*}
Taking subsequently the limits $n\to\infty$ and $\eps\to 0$, we deduce 
\[ \limsup_{\eps\to 0} \limsup_{n\to \infty} \int_{\Gamma} F_{\eps,K}(\nu) \,\sfP^n(\dd \nu)\leq \frac{C}{\beta}. \] 
But, since $\beta>0$ was arbitrary, we conclude that the right-hand side reduces to zero.

Similarly, for \eqref{eq:nc_conv2}, let
\[ f_{\eps}(x):=\int_{\calT} |g-g_{\eps}|(y,x)\,\gamma(\dd y), \qquad F_{\eps}(\nu):=  \int_{\calT} f_{\eps}(x)\, \nu(\dd x).  \]
Then by duality and Lemma \ref{lm:gammas_Gc} we obtain
\begin{align*}
    \limsup_{n\to \infty} \int_{\Gamma} F_{\eps,K}(\nu) \sfP^n(\dd \nu) &\leq \frac{C}{\beta}+\int_{\calT} (e^{\beta f_{\eps}(x)}-1) \dd \gamma \leq \frac{C}{\beta}+\frac{1}{\gamma(\calT)}\int_{\calT^2} e^{\beta \gamma(\calT) |g-g_{\eps}|(x,y) } \dd \gamma^{\otimes 2},
\end{align*}
where the last inequality follows by applying Jensen's inequality inside the exponential. Again taking the limit $\eps\to 0$ and thereafter $\beta\to \infty$ we conclude the proof. 
\end{proof}

\bibliographystyle{alpha}
\bibliography{references.bib}

\end{document}